\documentclass[final,onefignum,onetabnum]{siamart190516}

\usepackage{color}
\usepackage[algo2e, vlined,ruled]{algorithm2e}
\usepackage{algorithmic}
\usepackage{amsmath,url,cases,supertabular}
\usepackage{amssymb,mathtools,multirow}
\usepackage{enumerate,makecell}
\usepackage{amsfonts}
\usepackage{subeqnarray}

\usepackage{graphicx,epstopdf}
\usepackage{color,verbatim}
\usepackage{subeqnarray,subfigure}
\usepackage{listings,array}
\usepackage{booktabs}
\usepackage{dashrule}
\usepackage{arydshln}
\usepackage{appendix}
\usepackage{bbm}
\usepackage{mathrsfs}
\usepackage{makecell}
\usepackage{enumitem}

\usepackage{caption}
\usepackage{makecell,booktabs}
\usepackage{hyperref}
\usepackage{subfigure}

\numberwithin{equation}{section}
\newtheorem{example}{Example}[section]
\newtheorem{assumption}[theorem]{Assumption}

\newtheorem{remark}[theorem]{Remark}

\numberwithin{equation}{section}

\newcommand{\iprod}[2]{\left\langle #1, #2 \right \rangle}

\newcommand{\rev}[1]{{\color{black}{#1}}}

\newcommand{\tran}{\top}

\newcommand{\ba}{\begin{array}}
\newcommand{\ea}{\end{array}}

\newcommand{\bit}{\begin{itemize}}
\newcommand{\eit}{\end{itemize}}
\newcommand{\be}{\begin{equation}}
\newcommand{\ee}{\end{equation}}
\newcommand{\bea}{\begin{eqnarray}}
\newcommand{\eea}{\end{eqnarray}}

\newcommand{\st}{\mathrm{s.t.}}

\newcommand{\argmax}{\mathop{\mathrm{argmax}}}
\newcommand{\mtr}{\mathrm{tr}}

\newcommand{\nn}{\nonumber}
\newcommand{\diag}{\mathsf{diag}}
\newcommand{\off}{\mathsf{Of\;\!\!f}}
\newcommand{\Diag}{\mathsf{Diag}}

\newcommand{\proj}{\Pi}

\newcommand{\stief}{\mathcal{S}^{n,k}}
\newcommand{\stiefkk}{\mathcal{S}^{k,k}}

\newcommand{\stiefplus}{\mathcal{S}^{n,k}_+}
\newcommand{\stiefplusk}{\mathcal{S}^{n,k}_+}

\newcommand{\oblique}{\mathcal{OB}^{n,k}}
\newcommand{\obliqueplus}{\mathcal{OB}^{n,k}_+}

\newcommand{\supp}{\mathsf{supp}}

\newcommand{\rgrad}{\mathrm{grad}\,}
\newcommand{\rhess}{\mathrm{Hess}\,}
\newcommand{\grad}{\mathrm{grad}}

\newcommand{\LBB}{\mathtt{LBB}}

\newcommand{\R}{\mathsf{R}}

\newcommand{\Rmn}[1]{\uppercase\expandafter{\romannumeral#1}}

\renewcommand{\theenumi}{\roman{enumi}}

\newcommand{\Ccal}{\mathcal{C}}

\newcommand{\Fcal}{\mathcal{F}}

\newcommand{\Kcal}{\mathcal{K}}

\newcommand{\Ncal}{\mathcal{N}}

\newcommand{\Pcal}{\mathcal{P}}

\newcommand{\Scal}{\mathcal{S}}
\newcommand{\Tcal}{\mathcal{T}}

\newcommand{\Xcal}{\mathcal{X}}
\newcommand{\Ycal}{\mathcal{Y}}

\newcommand{\Rbb}{\mathbb{R}}

\newcommand{\Ftt}{\mathtt{F}}
\newcommand{\Tbf}{\mathbf{T}}

\newcommand{\zbf}{\mathbf{z}}
\newcommand{\xbf}{\mathbf{x}}
\newcommand{\ybf}{\mathbf{y}}
\newcommand{\x}{\mathbf{x}}

\newcommand{\dbf}{\mathbf{d}}
\newcommand{\abf}{\mathbf{a}}

\newcommand{\cbf}{\mathbf{c}}

\newcommand{\hbf}{\mathbf{h}}

\newcommand{\Ebf}{\mathbf{E}}
\newcommand{\1}{\mathbf{e}}
\newcommand{\0}{0}

\newcommand{\LNCD}{\mathcal{C}}
\newcommand{\SNCD}{\mathcal{N}}
\newcommand{\TC}{\mathcal{T}}
\newcommand{\LC}{\mathcal{L}}
\newcommand{\sub}{\mathrm{sub}}
\newcommand{\tol}{\mathrm{tol}}

\newcommand{\feas}{\mathrm{feas}}

\newcommand{\sgn}{\mathsf{sgn}}
\newcommand{\dist}{\mathrm{dist}}

\theoremstyle{nonumberplain}
\theoremheaderfont{\normalfont\itshape}
\theorembodyfont{\normalfont}
\theoremseparator{.}
\theoremsymbol{\proofbox}
\newtheorem{myproof}{Proof of \cref{theorem:lp:exact}}

\numberwithin{theorem}{section}

\graphicspath{{picnew/}}

\begin{document}

\title{An exact  penalty approach for optimization with nonnegative orthogonality constraints}
\author{Bo Jiang\thanks{School of Mathematical  Sciences,  Key Laboratory  for	NSLSCS  of  Jiangsu  Province,  Nanjing  Normal  University,  CHINA (\email{jiangbo@njnu.edu.cn}). Research supported in part by the Young Elite Scientists Sponsorship Program by CAST (2017QNRC001), the NSFC grants 11971239 and 11671036.} 
\and Xiang Meng\thanks{School of Mathematical Sciences, Peking University, CHINA (\email{1700010614@pku.edu.cn}).}
\and Zaiwen Wen\thanks{Beijing International Center for Mathematical
		Research, Peking University, CHINA (\email{wenzw@pku.edu.cn}).
		Research supported in part by the NSFC grant 11831002 and  Beijing Academy  of Artificial Intelligence.} 
\and Xiaojun Chen\thanks{Department of Applied Mathematics, The Hong Kong Polytechnic University, Hong Kong. (\email{
xiaojun.chen@polyu.edu.hk})}. Research supported in part by
the Hong Kong Research Grant Council PolyU153000/17P. }
\maketitle
\begin{abstract} 
Optimization with nonnegative orthogonality constraints has wide applications in machine learning and data sciences. It is NP-hard due to some combinatorial properties of  the constraints. We first propose an equivalent optimization formulation  with nonnegative and multiple spherical constraints and an additional single  nonlinear constraint. Various constraint qualifications, the first- and second-order optimality conditions of the equivalent formulation are discussed. \rev{By establishing a local error bound of the feasible set,} we design a class of (smooth) exact penalty models via keeping the nonnegative and multiple spherical constraints. The penalty models are exact if the penalty parameter is sufficiently large other than going to infinity.  A practical penalty algorithm with postprocessing is then developed.  It uses a second-order method to approximately solve a series of subproblems with nonnegative and multiple spherical constraints.  \rev{We study the asymptotic convergence of the penalty algorithm and establish that any limit point is a weakly stationary point of the original problem and becomes a stationary point under some additional mild conditions.}  Extensive numerical results on the projection problem, orthogonal nonnegative matrix factorization problems and the K-indicators model show the effectiveness of our proposed approach. 
\end{abstract}
\pagestyle{myheadings}
\thispagestyle{plain}
\markboth{\sc{optimization with nonnegative orthogonality constraints}}{\sc{B. JIANG, X. MENG, Z. WEN, AND  X. CHEN}}

\section{Introduction}
In this  paper, we consider optimization with nonnegative orthogonality constraints: 
\be\label{equ:prob:orth+}
\min_{X \in \Rbb^{n \times k}} \, f(X) \quad \st \quad   X^{\tran} X = I_k, \ X \geq 0, 
\ee
where $1\leq k \leq n$, $I_k$ is the $k$-by-$k$ identity matrix and $f\colon
\Rbb^{n \times k} \rightarrow \Rbb$ is continuously differentiable. The
feasible set of \eqref{equ:prob:orth+} is denoted as $\stiefplus
\coloneqq  \stief \cap \Rbb_{+}^{n \times k}$, where $\stief \coloneqq \{X \in
\Rbb^{n \times k} : X^{\tran} X = I_k\}$ is the Stiefel manifold.
 The non-negativity in $\stiefplus$ destroys the smoothness of $\stief $ and  introduces some combinatorial features.  Specifically, a matrix $X \in \stiefplus$ means that each row of $X$ has at most one positive element and each column of $X$ takes the unit norm. 
 Problem \eqref{equ:prob:orth+}
has captured a wide variety of  applications and interests, see \cite{berk2019certifiably, bergmann2019intrinsic, liu2019simple, yang2014optimality, zhang2018sparse} and the references therein.

Due to the combinatorial features, solving  \eqref{equ:prob:orth+} is generally NP-hard. Actually,  problem \eqref{equ:prob:orth+} covers  some classical NP-hard problems, such as the problem of checking copositivity of a symmetric matrix \cite{Urruty2010variational}, the quadratic assignment problem and the more general optimization over permutation matrices \cite{jiang2016lp} as special cases.  Besides, the constraint $X \in \stiefplus$ also appears in the   $k$-means clustering \cite{boutsidis2009unsupervised, carson2017manifold}, the min-cut problem \cite{povh2007copositive}, etc. 
Several typical instances of  problem \eqref{equ:prob:orth+} are briefly reviewed as follows.

\subsection{Applications}\label{subsec:applications}
We mainly introduced three classes of problem \eqref{equ:prob:orth+}.  The first one  is the so-called {\it trace minimization with nonnegative orthogonality constraints}, 
 formulated as 
\be\label{equ:prob:npca}
\min_{X \in \stiefplus} \, \mtr (X^{\tran} M X),
\ee
where $M \in \Rbb^{n \times n}$ is symmetric. If $M = -AA^{\tran}$ with  $A\in \Rbb^{n \times r}$ being  some data matrix,  \eqref{equ:prob:npca} is known as {\it nonnegative principal component analysis} \cite{montanari2016non, zass2007nonnegative}.  If $M = D - W$ with $W$ being a similarity matrix corresponding to $n$ objects and $D$ is a diagonal matrix having the same main diagonal as  $W\1$, where $\1$ is the all-one vector, \eqref{equ:prob:npca} is known as the {\it nonnegative Laplacian embedding} \cite{luo2009non}.  If $M = D - W + \mu R$ with some particularly chosen matrix $R$ and nonnegative regularization parameter $\mu$,  \eqref{equ:prob:npca} is known as  the {\it discriminative nonnegative spectral clustering} \cite{yang2012discriminative}.

The second one  is the {\it orthogonal nonnegative matrix factorization (ONMF)} \cite{ding2006orthogonal}.   
Given the data matrix $A\in \Rbb_+^{n \times r}$,  
ONMF solves 
\be\label{equ:prob:onmf}
\min_{X \in \stiefplus, Y \in \Rbb_{+}^{r \times k}} \,  \|A - XY^{\tran} \|_{\Ftt}^2 .
\ee
 Based on the idea of approximating the data matrix $A$ by its nonnegative subspace projection, Yang and Oja \cite{yang2010linear} proposed the {\it orthonormal projective nonnegative matrix factorization (OPNMF)} model as follows: 
\be\label{equ:prob:onmf:XXt}
\min_{X \in \stiefplus} \,  \|A - XX^{\tran}A\|_{\Ftt}^2.
\ee
Models \eqref{equ:prob:onmf} and  \eqref{equ:prob:onmf:XXt} are equivalent since  the optimal solutions $\bar X$ and $\bar Y$ of \eqref{equ:prob:onmf}  satisfy the relation $\bar Y = A^{\tran} \bar X$. Yang and Oja \cite{yang2010linear} also proposed 
a special OPNMF model by replacing the Frobenius norm in
\eqref{equ:prob:onmf:XXt} by the Kullback-Leibler divergence of $A$ and
$XX^{\tran}A$. The orthogonal symmetric non-negative matrix factorization models
were considered in \cite{kuang2012symmetric, paul2016orthogonal}.  

The third one is an  efficient K-indicators model  for data clustering  proposed by Chen et al. \cite{chen2019big}.  Let $U\in\stief$ be the features matrix extracted from the data matrix $A\in \Rbb^{n \times r}$, the K-indicators model in  \cite{chen2019big} reads 
\be\label{equ:prob:k:indicator}
\min_{X \in \stiefplusk,  Y \in \stiefkk}\, \|U Y - X\|_{\Ftt}^2 \quad \st \quad \|X_{i,:}\|_0 = 1, i \in [n],
\ee
where $\|X_{i,:}\|_0$ is the  number of nonzero elements in the $i$-th row of $X$, namely, $X_{i,:}$.

\subsection{Related works}
Optimization on the Stiefel manifold \cite{absil2009optimization,
wen2013feasible} has already been well explored. However, a systematic study on problem \eqref{equ:prob:orth+} is lacked in the literature albeit it captures many applications. The existing works rarely considered the general problem  \eqref{equ:prob:orth+}, and most of them focused on some special formulations of  \eqref{equ:prob:orth+}. 
We  briefly review some main existing methods.   For solving ONMF model
\eqref{equ:prob:onmf}, motivated by the  multiplicative update methods  for
nonnegative matrix factorization, Ding et al. \cite{ding2006orthogonal} and Yoo
and Choi \cite{yoo2008orthogonal} gave  two different multiplicative update
schemes. By establishing the equivalence of ONMF with a weighted variant of
spherical $k$-means, Pompoli et al. \cite{pompili2014two}  proposed an EM-like
algorithm.  Pompoli et al. \cite{pompili2014two}  also designed an augmented
Lagrangian method via penalizing the nonnegative constraints but keeping the
orthogonality constraints. Li et al. \cite{li2015two} and Wang et al.
\cite{wang2019clustering2, wang2019clustering} considered the nonconvex penalty
approach by keeping the nonnegative constraints. Some theoretical properties of
the nonconvex penalty model were investigated in \cite{wang2019clustering} but
the results may not be applied directly to the general problem \eqref{equ:prob:orth+}.   Zhang et al. \cite{zhang19greedy} proposed a greedy orthogonal pivoting algorithm which can promote exact orthogonality.
For solving OPNMF model  \eqref{equ:prob:onmf:XXt},  Yang and Oja \cite{yang2010linear} designed a specific multiplicative update  method.  Pan and Ng \cite{pan2018orthogonal} introduced a convex relaxation model, wherein the relaxed model is solved by the alternating direction method of multipliers.   We remark that the multiplicative update scheme for solving problem \eqref{equ:prob:onmf} or \eqref{equ:prob:onmf:XXt} highly  depends on the specific formulation of the objective function,  so it is not easy to extend this class of methods to  solve the general problem \eqref{equ:prob:orth+}. 
In addition, Wen and Yin  \cite{wen2013feasible} designed an augmented Lagrangian method by penalizing the nonnegative constraints but keeping the orthogonality constraints for solving the quadratic assignment problem. Chen et al. \cite{chen2019big} proposed a semi-convex relaxation model and construct a double-layered alternating projection scheme to solve the K-indicators model \eqref{equ:prob:k:indicator}. 

\subsection{Our contribution}  
\rev{By well exploring  the structure of $\stiefplus$, we first give a new characterization of $\stiefplus$ as 
\be \label{equ:stiefplus:equivalent:general:V} 
\stiefplus  =  \Xcal_V \coloneqq  \obliqueplus \cap \{X\in \Rbb^{n \times k} : \|XV\|_{\Ftt} = 1\}, 
\ee
where $\obliqueplus = \{X \in \Rbb^{n \times k}: \|\xbf_{j}\| = 1, \xbf_j \geq
\0,   j \in [k]\}$, in which $\xbf_j$ denotes the $j$-th column of $X$,
and  the constant matrix $V \in \Rbb^{k\times r}$($1 \leq r \leq k$) satisfies  $\|V\|_{\Ftt} = 1$
and $\min_{i,j \in [k]} [VV^{\tran}]_{ij} > 0$.  
Based on this equivalent characterization, a reformulation 
of problem \eqref{equ:prob:orth+} is given as  
\be \label{equ:prob:orth+:new}
\min_{X \in \obliqueplus} \, f(X) \quad \ \st \quad  \|XV\|_{\Ftt} = 1.
\ee 
We show that the classical  constraint qualifications (CQs) including cone-continuity property (CCP) and Abadie CQ   (ACQ) only hold  when $\|X\|_0 =n$ while the weakest Guignard CQ (GCQ) always holds. The
first- and second-order optimality conditions are also given for problem
\eqref{equ:prob:orth+:new}.  We then explore the relationship between   problems
\eqref{equ:prob:orth+} and  \eqref{equ:prob:orth+:new} and show that the two
formulations not only share the same minimizers but also the same optimality
conditions.} 

\rev{To motivate the exact penalty approach, we prove that a local  error bound with exponent $1/2$  holds for $\stiefplus$.} 
Therefore, via keeping the simple constraints $\obliqueplus$ and penalizing the constraint $\|XV\|_{\Ftt} = 1$, we propose a class of exact penalty models:
\be \label{equ:prob:orth+:new:exact:penalty}
 \min_{X \in \obliqueplus}\, \left\{P_{\sigma,p,q,\epsilon}(X):=  f(X) + \sigma    \left(\zeta_q(X)+ \epsilon\right)^{p} \right\},
\ee
where $\zeta_q(X) := \|XV\|_{\Ftt}^q - 1 $, $\sigma > 0$ is the penalty parameter and $p, q >0$ and $\epsilon \geq 0$ are the model parameters. An important feature of  \eqref{equ:prob:orth+:new:exact:penalty} is that it allows smooth penalty by choosing appropriate model parameters, such as choosing $p \geq 1$ and $\epsilon = 0$. We show that if the penalty parameter $\sigma$ is chosen to be larger than a positive constant,  the optimal solution of the exact penalty problem (possibly a  postprocessing will be invoked) is also optimal for the original problem. A more general exact penalty model \eqref{prob:exact:penalty:Phi} is also discussed. Then we develop a practical exact penalty algorithm which approximately solves a series of penalty subproblems of the form \eqref{equ:prob:orth+:new:exact:penalty} and performs a postprocessing  procedure to further improve the solution quality.  \rev{We study the asymptotic  convergence of the penalty algorithm and show that any limit point of the sequence generated by the algorithm is a weakly stationary point of \eqref{equ:prob:orth+:new}.  We also provide some mild conditions under which the limit point is a stationary point of \eqref{equ:prob:orth+:new}.} To solve the subproblem \eqref{equ:prob:orth+:new:exact:penalty} efficiently, we develop a second-order algorithm for solving optimization over $\obliqueplus$, which is of independent interest.    
We also discuss how to use the proposed penalty algorithmic framework to solve a two block  model  
\be\label{equ:prob:orth+:general}
\min_{X \in \stiefplus, Y \in \Ycal} \, f(X,Y),
\ee
where $\Ycal$ is some simple closed set such that  the projection to the set $\Ycal$ is easy to compute. 
Finally, numerical results on the projection problem and ONMF on synthetic data, text clustering, hyperspectral unmixing and K-indicators model demonstrate the efficiency of our approach.

\subsection{Organization}
The rest of this paper is organized as follows.  A new characterization of $\stiefplus$ and the equivalent reformulation  of problem
\eqref{equ:prob:orth+}  are  given in  \cref{section:reform}.
We propose the exact penalty  model in  \cref{section:exact:penalty}. A practical penalty algorithm together with its convergence results is presented in  \cref{section:penaly:alg}.  We investigate a second-order method for solving the penalty subproblem, namely, optimization over $\obliqueplus$ in  \cref{section:opt:obliqueplus}.  A variety of numerical results are presented in  \cref{section:num}. Finally, we make some concluding remarks in  \cref{section:concluding}. 

\subsection{Notations}  \label{subsection:notation}
For a positive integer $n$, we denote $[n]\coloneqq \{1, \ldots, n\}$. 
The $j$-th  column (resp. $i$-th row) of a matrix $Z$ with appropriate dimension is denoted by $Z_{:,j}$ (resp. $Z_{i,:}$).  
For simplicity, we also denote $\zbf_j \coloneqq Z_{:,j}$. Let $\Scal$ be a closed set, we use $\proj_{\Scal}(\cdot)$ to denote the projection   operator. 
The  number of nonzero elements of $Z$ is $\|Z\|_0$. 
 The Frobenius norm of  $Z$ is  $\|Z\|_{\Ftt}$ while the 2-norm of a vector $z$ is $\|z\|$.   For $z \in \Rbb^n$, $\Diag(z)\in \Rbb^{n \times n}$ is a diagonal matrix with the main diagonal being $z$.  For $Z \in \Rbb^{n \times n}$, $\diag(Z)\in \Rbb^n$ is the main diagonal of $Z$. For simplicity,  we use  $\Diag(Z)$ to denote $\Diag(\diag(Z))$.  Let $\off(Z) = Z - \Diag(Z)$.  The inner product between  two matrices $A$ and $B$ with the same sizes is $\langle A, B \rangle = \mtr(A^{\tran}B)$.  
The notation   $\0 \leq A \perp B \geq \0$ means that $A  \geq \0$ and $B \geq 0$ component-wisely and $A \circ B = \0$, where $\circ$ means the Hadamard product operation. Similarly, $\min(A,B)$ takes the minimum of matrices $A$ and $B$  component-wisely.

\section{Reformulation of problem \eqref{equ:prob:orth+}} \label{section:reform}

Let $f^*$ and $\Xcal^{*}$ be the optimal value and optimal solution set of
problem \eqref{equ:prob:orth+} or \eqref{equ:prob:orth+:new}, respectively. We
define $\sgn(\stiefplus) := \{H \in \{0,1\}^{n,k}: H = \sgn(X)\ \mbox{with}\  X
\in \stiefplus\}$, where 
$\sgn(X)_{ij} = |X_{ij}|/X_{ij}$ if $X_{ij} \neq 0$ and  $\sgn(X)_{ij} = 0$ otherwise. 
 The set  $\sgn(\Xcal^{*})$ is defined accordingly. 
For ease of reference, we state a blanket assumption on problem \eqref{equ:prob:orth+} or \eqref{equ:prob:orth+:new}.  
\begin{assumption}\label{assump:blank}
We assume that $\emptyset \neq \sgn(\stiefplus)  \setminus \sgn (\Xcal^*)  :=  \{H \in \{0,1\}^{n,k}:  H \in \sgn(\stiefplus)\ \mbox{but}\ H \not \in  \sgn (\Xcal^*) \}$, namely, the constant 
$\chi_f := \tilde f^* - f^* > 0$ with 
\be
\tilde f^* = \min_{X\in \Rbb^{n \times k}} \, f(X)  \quad \st \quad X  \in \stiefplus,  \sgn(X) \in \sgn(\stiefplus) \setminus \sgn(\Xcal^{*}). \nn
\ee
\end{assumption}
If \cref{assump:blank} does not hold, then $\sgn(\stiefplus)  \setminus \sgn (\Xcal^*)  =  \emptyset$, which with $\sgn(\Xcal^{*}) \subseteq \sgn(\stiefplus)$ tells   $\sgn(\stiefplus)  = \sgn (\Xcal^*)$. 
In this case, problem \eqref{equ:prob:orth+} or \eqref{equ:prob:orth+:new} is trivial in the sense that any  $X$ with $\|X\|_0 = k$ and  having a $k$-by-$k$ permutation matrix as its submatrix is  a global minimizer. However, we can verify that  \cref{assump:blank} holds for the test problems   in section \ref{section:num} by randomly choosing some  matrices in $\stiefplus$ with different sign matrices and comparing their function values.

For $X \in \stiefplus$, we define  $\supp(X) := \{(i,j) \in [n] \times [k]: X_{ij} \neq 0\}$ and $\Omega_0(X) = \left\{(i,j) \in [n] \times [k]:  X_{ij} = 0\right\}$.  The set  $\Omega_0(X)$ is splited into two disjoint sets as $\Omega'_0(X) $$ = \left\{(i,j) \in \Omega_0(X): \|X_{i,:}\| > 0\right\}$ and $\Omega''_0(X) =
 \left\{(i,j) \in \Omega_0(X): \|X_{i,:}\| = 0\right\}$. 

We first give an equivalent  algebraic characterization of  $\stiefplus$. 

\begin{lemma} \label{lemma:key:observation:V}
For any $X \in \obliqueplus$, there holds that 
$\|XV\|_{\Ftt} \geq 1$, 
where the equality holds if and only if $X \in \stiefplus$. Furthermore,  the  characterization \eqref{equ:stiefplus:equivalent:general:V} holds. 
\end{lemma}
\begin{proof} 
With $\|V\|_{\Ftt} = 1$ and $X \in \obliqueplus$,  we have 
$\|XV\|_{\Ftt}^2 - 1 = \iprod{VV^{\tran}}{X^{\tran}X - I_p}$ $= \sum_{i,j \in [k], i \neq j} [VV^{\tran}]_{ij} (\x_i^{\tran} \x_j),$
which with  $VV^{\tran} > 0$ implies that   $\|XV\|_{\Ftt}^2 - 1  \geq 0$.
The  equality  holds 
if and only if $\x_i^{\tran} \x_j = 0$ for $i, j \in [k]$ and $i \neq j$, which with $X \in \obliqueplus$ means that $X \in \stiefplus$.  Hence \eqref{equ:stiefplus:equivalent:general:V} follows directly. The proof is completed. 
\end{proof}

With the equivalent characterization \eqref{equ:stiefplus:equivalent:general:V}
of $\stiefplus$ and \cref{lemma:stiefplus:XcalV},  we reformulate problem  \eqref{equ:prob:orth+} as problem  \eqref{equ:prob:orth+:new}. \rev{Throughout this paper, we mainly focus on the formulation \eqref{equ:prob:orth+:new} since it gives us more insight to design our exact penalty approach.  We are now going to discuss the CQs and first- and second-order optimality conditions \eqref{equ:prob:orth+:new} and investigate the relationship between the two formulations \eqref{equ:prob:orth+} and \eqref{equ:prob:orth+:new}.}

\rev{
\subsection{Constraint qualifications of problem \eqref{equ:prob:orth+:new}} \label{subsection:cqs}   
In this subsection, we investigate several CQs of problem
\eqref{equ:prob:orth+:new} which are important to establish the  optimality
conditions. We mainly consider,  Guignard CQ (GCQ), which is the weakest CQ,
Abadie CQ   (ACQ) and  the  cone-continuity property (CCP), which is the weakest
strict CQ \cite{andreani2016cone}.  Note that the following implications hold: CCP $\Longrightarrow$ ACQ $\Longrightarrow$ GCQ.

We first give the expression of the tangent cone $\TC_{\Xcal_V}(X)$ and linearized cone $\LC_{\Xcal_V}(X)$ at $X\in \Xcal_V$. 
Following the definition of linearized cone, we have 
\be  \label{equ:linearizedCone:Xcal}
\LC_{\Xcal_V}(X)  =  \left\{D \in \Rbb^{n \times k}: 
\begin{array}{l}
\x_j^{\tran} \dbf_j = 0 \ \forall   j \in [k], \\
  D_{ij} \geq 0\ \forall (i,j) \in \Omega_0(X), \langle D, XVV^{\tran}\rangle  = 0
\end{array}
\right\}. 
\ee
With the choice of $V$ and \eqref{equ:stiefplus:equivalent:general:V},  $\langle
D, XVV^{\tran}\rangle$ $= 0$ tells that $\dbf_i^{\tran} \sum_{j \in
[k]}(VV^{\tran})_{ji}\xbf_j = 0$ which further implies that $D_{ij} = 0$ if
$(i,j) \in \Omega_0'(X)$. This together with the definitions of $\Omega_0'(X)$ and $\Omega_0''(X)$ and \eqref{equ:linearizedCone:Xcal} yields  
  \be \label{equ:tangentCone:Xcal:stiefplus}
  \LC_{\Xcal_V}(X)   
:={}   \left\{D\in \Rbb^{n \times k}: \!\!\!\!\!\!\!
\begin{array}{ll} 
& \xbf_j^{\tran} \dbf_j = 0\  \forall j \in [k],D_{ij} = 0\ \forall (i,j) \in \Omega_0'(X), \\ 
 &  D_{ij} \geq 0 \ \forall (i,j) \in \Omega_0''(X) 
 \end{array}
 \!\!
 \right\}.
\ee
The tangent cone at $X$ is given as $\Tcal_{\Xcal_V}(X)  
 = \{D\in \Rbb^{n \times k}:  \exists \alpha^l > 0,  \alpha^l \rightarrow 0, D^l \rightarrow D \ \mbox{such that} \  X^l:= \bar X + \alpha^l D^l \in \Xcal_V\}$. 
Clearly we have $\Tcal_{\Xcal_V}(X) \subseteq \LC(X)$. For  each $l$ and $i \in [n]$, there is at most one element of $(X^l - X)_{ij}\ \forall (i,j) \in \Omega_0''(X)$ being nonzero. Hence,  any $D \in \Tcal_{\Xcal_V}(X)$ must satisfy $\|D_{i,:}\|_0 \leq 1$ if $X_{i,:}=0$.  On the other hand, for any $D \in \LC_{\Xcal_V}(X)$ with $\|D_{i,:}\|_0 \leq 1$ if $X_{i,:}=0$, choosing $D^l \equiv D$, $\alpha^l = 1/l$ and $X^l$  as $\xbf_j^l = (\xbf_j + \alpha^l \dbf_j)/\|\xbf_j + \alpha^l \dbf_j\|$, it is clear that $X^l \in \Xcal_V$. This means that $D \in \Tcal_{\Xcal_V}(X)$. In summary, we arrive at 
\be \label{equ:tangent:cone}
\TC_{\Xcal_V}(X) =   \LC_{\Xcal_V}(X) \cap \left\{D \in \Rbb^{n\times k}:  \|D_{i,:}\|_0 \leq 1 \ \mathrm{if}\ X_{i,:} = 0\ \forall i \in [n] \right\}. 
\ee

We now discuss the CQs in the following lemma.  

\begin{lemma}\label{lemma:CQ}
Consider a feasible $\bar X\in \Xcal_V$ of \eqref{equ:prob:orth+:new}.   If $k=1$,  then the linear independence constraint qualification (LICQ) holds at $\bar X$;  if $2 \leq k \le n$ and $\|\bar X\|_0 = n$, then CCP holds; if $2 \leq k \le  n$ and $\|\bar X\|_0 < n$, then GCQ holds but ACQ fails to hold.
\end{lemma}
\begin{proof}
{Case I}. $k = 1$. It is straightforward to check that 
LICQ  holds at $\bar X$.

{Case II}.  $2 \leq k \le n$ and $\|\bar X\|_0  = n$, namely, each row of $\bar X$ has exactly one positive element.  In this case $\Omega'_0(\bar X) = \Omega_0(\bar X)$ and $\Omega''_0(\bar X) = \emptyset$. 
For a sequence $\{X^l\} \subset \Xcal_V$ and $X^l \rightarrow \bar X$, we consider the closed convex cone\footnote{For the definition of this cone, one can refer to equation (2.11) in \cite{andreani2016cone}. },  which is related to CCP, as 
\[
\Kcal_{\Xcal_V}(X^l) =  \left\{X^l \Diag(\Lambda) + \lambda X^lVV^{\tran} - \sum_{(i,j) \in \Omega_0(\bar X)} Z_{ij} \mathbf{E}_{ij}: \Lambda \in \Rbb^k, \lambda \in \Rbb, Z_{ij} \in \Rbb_+ \right\}
\]
where $\Ebf_{ij} \in \Rbb^{n \times k}$ with  $(i,j)$ element being one while
the remaining elements being zeros.  Since $\Xcal_V \owns X^l \rightarrow \bar
X$ and $\Omega_0''(\bar X) = \emptyset$, we have  $\Omega_0(X^l) = \Omega(\bar
X)$ and $\supp(X^l) = \supp(\bar X)$ for sufficiently large $l$. Thus $X^lVV^{\tran}=X^l\Diag(VV^\tran)+\sum_{(i,j) \in \Omega_0(\bar X)} Z_{ij} \mathbf{E}_{ij}$ for some $Z_{ij} \in \Rbb_+$. By some easy calculations, one have 
\be\label{equ:Kcal:cone}
\Kcal_{\Xcal_V}(X^l) = \left\{X^l \Diag(\Lambda)  + \sum_{(i,j) \in \Omega_0(\bar X)} Z_{ij} \mathbf{E}_{ij}: \Lambda \in \Rbb^k, Z_{ij} \in \Rbb \right\}
\ee
for sufficiently large $l$, which with $X^l \rightarrow \bar X$ implies that $\limsup_{X^l \rightarrow \bar X} \Kcal_{\Xcal_V}(X^l) \subset \Kcal_{\Xcal_V}(\bar X)$.  This means that CCP holds in this case.

{Case III}. $2 \leq k \le n$ and $\|\bar X\|_0  <  n$.  In this case, $\Omega''_0(\bar X) \neq \emptyset$.  By  definition, it is  easy to verify that the polar cones of 
$\Tcal(X)$ and $\LC(X)$  coincide, namely, 
\[
\Tcal_{\Xcal_V}(X)^{\circ} = \LC_{\Xcal_V}(X)^{\circ} =  \left\{D \in \Rbb^{n \times k}:   
  \begin{array}{ll}
  D_{ij} =  \lambda_j X_{ij}, \lambda_j \in \Rbb \ \forall (i,j) \in \supp(X),\\ 
  D_{ij} \leq 0\ \forall  (i,j) \in \Omega''_0(X)
  \end{array}
  \right\}.
\]
This means GCQ holds. Recalling $\Omega''_0(\bar X) \neq \emptyset$, we know from  \eqref{equ:tangentCone:Xcal:stiefplus} and \eqref{equ:tangent:cone} that $\Tcal_{\Xcal_V}(X) \subsetneq  \LC_{\Xcal_V}(X)$, which tells that ACQ does not hold. The proof is completed. 
\end{proof}
}

\subsection{Optimality conditions of problem \eqref{equ:prob:orth+:new}} 
\label{subsection:1st2ndKKT}
Denote the Riemannian gradient  \cite{absil2009optimization} with respect to the oblique manifold  $\oblique$ as 
\be\label{equ:rgrad:f}
\rgrad f(X) = 
\nabla f(X) - X  \Diag\big(X^{\tran} \nabla f(X)\big).
\ee

\begin{theorem} [First-order necessary conditions]\label{theorem:prob:new:KKT} 
Suppose that $\bar X \in \Xcal_V$ is a local minimizer of  \eqref{equ:prob:orth+:new}. 
 Then $\bar X$ is a stationary point of \eqref{equ:prob:orth+:new}, namely,  $-\nabla f(\bar X) \in \LC_{\Xcal_V}(\bar X)^{\circ}$, which can be further represented as
\begin{subequations}
\begin{align}
&[\rgrad f(\bar X)]_{ij} = 0 \quad \forall (i,j) \in \supp(\bar X), \label{equ:strongStationary:1}\\
&[\nabla f(\bar X)]_{ij} \geq 0 \quad \forall (i,j)  \in \Omega''_0(\bar X).\label{equ:strongStationary:2}
\end{align}
\end{subequations}
\end{theorem}
\begin{proof}
    \cref{lemma:CQ} tells that GCQ holds at $\bar X$. Thus, $\bar X$ must be a
    stationary point and  $-\nabla f(\bar X) \in \LC_{\Xcal_V}(\bar X)^{\circ}$
    due to \cite[Proposition 3.3.14]{bertsekas1999nonlinear}.  Hence, there
    exists Lagrange multiplier vector $\bar \Lambda \in \Rbb^k$ corresponding to $\|\xbf_j\|^2 = 1, j \in [k]$, Lagrange multiplier  $\bar \lambda \in \Rbb$ corresponding to $\|XV\|_{\Ftt} = 1$ and Lagrange multiplier matrix  $\bar Z \in \Rbb_{+}^{n\times k}$  corresponding to $X \geq \0$  such that  $\0 \leq \bar X \perp \bar Z \geq \0$ and 
$\nabla_X L(\bar X, \bar \Lambda, \bar Z, \bar \lambda)  = \0$, namely, 
\be \label{equ:theorem:prob:new:KKT:c1}
\nabla f(\bar X) -  \bar X \Diag\big(2 \bar \Lambda - \bar \lambda \diag(VV^{\tran})\big) - \bar Z + \bar \lambda \bar X \off(VV^{\tran}) =  \0.
\ee
Here,  the Lagrangian function is given as 
\be \label{equ:Lagrangian:fun}
L(X, \bar \Lambda,Z, \bar \lambda) = f(X) - \sum_{j\in [k]} \bar \Lambda_j (\|\xbf_j\|^2 - 1) - \iprod{\bar Z}{X} + \bar \lambda (\|XV\|_{\Ftt} - 1).
\ee
Multiplying $\bar X^{\tran}$ on both sides of \eqref{equ:theorem:prob:new:KKT:c1} and then performing the $\diag(\cdot)$ operator,  with $\bar X^{\tran} \bar  X= I_k$, we have $2 \bar \Lambda - \bar \lambda \diag(VV^{\tran}) = \diag(\bar X^{\tran} \nabla f(\bar X)),$ 
 which again with \eqref{equ:theorem:prob:new:KKT:c1}  and \eqref{equ:rgrad:f} implies that $\bar Z = \rgrad f(\bar X) + \bar \lambda \bar X \off(VV^{\tran})$.  
 Recalling  $\bar X \in \Xcal_V$, it is easy to verify that  $[\bar X \off(VV^{\tran})]_{ij} = 0$  $\ \forall (i,j) \in \supp(\bar X)$, $[\bar X \off(VV^{\tran})]_{ij} > 0\ \forall (i,j) \in \Omega_0'(\bar X)$ and $[\rgrad f(\bar X)]_{ij}$$ =    [\nabla f(\bar X)]_{ij}\  \forall (i,j) \in \Omega'_0(\bar X) \cup \Omega''_0(\bar X)$. 
Hence,  we have 
\be\label{equ:Z:ulam}
\bar Z_{ij} =
\begin{cases}
 [\rgrad f(\bar X)]_{ij} &  (i,j) \in \supp(\bar X),\\ 
 [\nabla f(\bar X)]_{ij} + \bar \lambda [\bar X \off(VV^{\tran})]_{ij} & (i,j) \in \Omega'_0(\bar X), \\
  [\nabla f(\bar X)]_{ij} & (i,j)  \in \Omega''_0(\bar X).
 \end{cases}
\ee
For each $(i,j) \in \Omega'_0(\bar X)$, we can always choose    
\be \label{equ:lambda:inequality}
\bar \lambda \geq \bar \lambda(\bar X):= \max_{(i.j) \in \Omega'_0(\bar X)} \frac{-[\nabla  f(\bar X)]_{ij}}{[\bar X \off(VV^{\tran})]_{ij}}
\ee
 such that $\bar Z_{ij} \geq 0$. Thus we arrive at the equivalent formulation 
\eqref{equ:strongStationary:1} and \eqref{equ:strongStationary:2}.
\end{proof}


\rev{
We borrow the idea from mathematical programs with complementarity constraints, see \cite{scheel2000mathematical} for instance,  to define a   weakly stationary point $\bar X$ of  problem \eqref{equ:prob:orth+:new}.  
\begin{definition} We call $\bar X \in \Xcal_V$  a weakly stationary point of problem \eqref{equ:prob:orth+:new} if 
\eqref{equ:strongStationary:1} holds at $\bar X$. 
\end{definition}

Note that a weakly stationary point $\bar X$ has no requirements on the sign of  the Lagrange multiplier $[\nabla f(\bar X)]_{ij}$ with $(i,j) \in \Omega_0''(\bar X)$ and it  is actually a stationary point of problem \eqref{equ:prob:orth+:new} with additional constraints $X_{ij} = 0\ \forall (i,j) \in \Omega_0''(\bar X)$.
In the case when $\Omega_0''(\bar X) = \emptyset$, namely, $\|\bar X\|_0 = n$ or  $\nabla f(\bar X)\geq 0$ always holds, then the weakly stationary point $\bar X$ becomes a stationary point of problem  \eqref{equ:prob:orth+:new}.}

 We now assume that $f$ in  problem \eqref{equ:prob:orth+:new} is twice continuously differentiable.
The set of all sequential null constraint directions at a stationary point $\bar X$ (see Definition  8.3.1 in  \cite{sun2006optimization}) of problem \eqref{equ:prob:orth+:new} is given as 
\be
\SNCD_{\Xcal_V}(\bar X, \bar Z) = \left\{\!D \in \Rbb^{n \times k}: \!
\begin{array}{l}
X^l:= \bar X +  \alpha^{l} D^{l} \in \Xcal_V, \alpha^{l} > 0, \alpha^{l} \rightarrow 0, D^{l} \rightarrow D,\\
X^l_{ij} = 0\  \mbox{if}\ \bar Z_{ij} > 0, X^l_{ij}\geq 0\ \mbox{if}\ \bar Z_{ij} = 0
\end{array}
\!
\right\}\!.
\nn
\ee
Notice that $\SNCD_{\Xcal_V}(\bar X, \bar Z) \subseteq \TC_{\Xcal_V}(\bar X)$, 
with  \eqref{equ:Z:ulam} and   \eqref{equ:tangent:cone},  we have  
\be \label{equ:SNCD:Xcal}
\SNCD_{\Xcal_V}(\bar X, \bar Z) = \TC_{\Xcal_V}(\bar X) \cap \left\{\!D \in \Rbb^{n \times k}: 
D_{ij} = 0 \  \mbox{if}\  [\nabla f(\bar X)]_{ij} > 0\ \forall(i,j) \in  \Omega''_0(\bar X)
\right\}\!.
\ee
Since $\SNCD_{\Xcal_V}$ is independent of $Z$,  we write for $\SNCD_{\Xcal_V}(\bar X, \bar Z)$ as $\SNCD_{\Xcal_V}(\bar X)$ for short. Similarly, we have the set of all linearized null constraint directions at $\bar X$, also known as the critical  cone, 
$\LNCD_{\Xcal_V}(\bar X)  
=  \LC_{\Xcal_V}(\bar X)  \cap  \{D \in \Rbb^{n \times k}: D_{ij} = 0\ \mbox{if}\ \bar Z_{ij} > 0, (i,j) \in \Omega_0(\bar X)\}$.
Using \eqref{equ:tangentCone:Xcal:stiefplus}, \eqref{equ:Z:ulam}, \eqref{equ:strongStationary:1} and \eqref{equ:strongStationary:2},  we further have 
\be \label{equ:LNCD:Xcal}
\LNCD_{\Xcal_V}(\bar X)  
=    \LC_{\Xcal_V}(\bar X)  \cap \left\{D \in \Rbb^{n \times k}: 
D_{ij} = 0 \  \mbox{if}\  [\nabla f(\bar X)]_{ij} > 0\ \forall (i,j) \in  \Omega''_0(\bar X)
\right\}. 
\ee
Define the Riemannian Hessian \cite{absil2009optimization} with respect to the oblique manifold $\oblique$ as  
\be \label{equ:hess} 
\rhess f(X)[D]  \coloneqq \nabla^2 f(X)[D]  -  D \Diag\big(\bar X^{\tran} \nabla f(X)\big), 
\ee
where $D$ satisfies $\xbf_j^{\tran} \dbf_j = 0$ for  $j \in [k]$.
We are now ready to establish the second-order optimality conditions as follows.
\begin{theorem}[Second-order necessary conditions]\label{theorem:2nd:necessary}
If $\bar X \in \Xcal_V$ is a local minimizer of problem
\eqref{equ:prob:orth+:new}, then 
\be \label{equ:opt:SNC} 
\iprod{D}{\rhess f(\bar X)[D]} \geq 0, \quad \emph{for all} \ D \in \SNCD_{\Xcal_V}(\bar X).
\ee
\end{theorem}
\begin{proof} 
The proof of \cref{theorem:prob:new:KKT} tells that 
  $2 \bar \Lambda - \bar \lambda \diag(VV^{\tran}) = \diag(\bar X^{\tran} \nabla f(\bar X))$. With   \eqref{equ:Lagrangian:fun}, we have 
$\nabla_{XX}^2 L(\bar X, \bar \Lambda, \bar Z, \bar \lambda)[D] $  $= \rhess f(\bar X)[D] + \bar \lambda D \off(VV^{\tran})$.
By \cite[Theorem 8.3.3]{sun2006optimization} and the fact that $\bar X$ is a local minimizer of problem \eqref{equ:prob:orth+:new}, we have \be \label{equ:opt:SNC:proof:1} 
\iprod{D}{\rhess f(\bar X)[D] + \bar \lambda D \off(VV^{\tran})} \geq 0, \quad \mathrm{for\ all} \ D \in \SNCD_{\Xcal_V}(\bar X).
\ee
For  $D \in \SNCD_{\Xcal_V}(\bar X)$, we know from \eqref{equ:SNCD:Xcal} that $D^{\tran} D$ must be diagonal. Thus,  
 \be  \label{equ:opt:SNC:proof:2}  
 \iprod{D}{D \off(VV^{\tran})} = \mtr(D^{\tran}D \off(VV^{\tran})) = 0,
 \ee
which with \eqref{equ:opt:SNC:proof:1}  implies \eqref{equ:opt:SNC}. The proof is completed. 
\end{proof} 

\begin{theorem}[Second-order sufficient conditions]\label{theorem:2nd:sufficient}
Suppose that $\bar X \in \Xcal_V$ is a stationary point  of problem \eqref{equ:prob:orth+:new} and that there exists a Lagrange multiplier $\bar \lambda$ associated to $\|XV\|_{\Ftt} = 1$ with $\bar \lambda \geq \bar \lambda(\bar X)$ such that
\be  \label{equ:opt:SCC} 
\iprod{D}{\rhess f(\bar X)[D] + \bar \lambda D \off(VV^{\tran})} > 0, \quad \emph{for all} \ D \in \LNCD_{\Xcal_V}(\bar X)/\{0\}. 
\ee
Then $\bar X$ is a strict local minimizer of \eqref{equ:prob:orth+:new}. 
\end{theorem}

\begin{proof}
 It follows directly from, for instance \cite[Theorems 8.3.4]{sun2006optimization}.
\end{proof}

\begin{remark} \label{remark:2nd:opt:condition}
Consider the case when $\Omega''_0(\bar X) = \emptyset$, namely, $\|\bar X\|_0  = n$.  Following from \eqref{equ:SNCD:Xcal}, \eqref{equ:LNCD:Xcal} and $\TC_{\Xcal_V}(\bar X) = \LC_{\Xcal_V}(\bar X)$, 
we have 
$\SNCD_{\Xcal_V}(\bar X) = \LNCD_{\Xcal_V}(\bar X) = 
\LC_{\Xcal_V}(\bar X)$. Recalling \eqref{equ:opt:SNC:proof:2}, we thus know that 
 \eqref{equ:opt:SNC} and \eqref{equ:opt:SCC} become 
$\iprod{D}{\rhess f(\bar X)[D]} \geq 0$ $\forall D \in \LC_{\Xcal_V}(\bar X)$
and 
$\iprod{D}{\rhess f(\bar X)[D]} > 0$ $\forall D \in \LC_{\Xcal_V}(\bar X)/\{0\}$, respectively. 
\end{remark}

\rev{
\subsection{Relationship between problems \eqref{equ:prob:orth+} and \eqref{equ:prob:orth+:new}}
It is clear that formulations \eqref{equ:prob:orth+} and
\eqref{equ:prob:orth+:new} share the same minimizers. Moreover, the two problems share the same stationary points. 
\begin{lemma} \label{lemma:stiefplus:XcalV}
(i) The statements in \cref{lemma:CQ} hold for  problem \eqref{equ:prob:orth+}; (ii) Problems \eqref{equ:prob:orth+} and \eqref{equ:prob:orth+:new} share the same  minimizers and optimality conditions. 
%

\end{lemma}
\begin{proof} 
   We first claim that that problems \eqref{equ:prob:orth+} and \eqref{equ:prob:orth+:new} have the same tangent and linearized cones. Obviously, we know 
 $\TC_{\stiefplus}(X)  = \TC_{\Xcal_V}(X)$.
For the linearized cone, we have 
$ 
\LC_{\stiefplus}(X) = \{D \in \Rbb^{n \times k}: X^{\tran} D + D^{\tran} X = 0, D_{ij} \geq 0 \ \forall (i,j) \in \Omega_0(X)\}. \nn 
$
The linear equation above tells that $\xbf_l^{\tran} \dbf_j + \dbf_l^{\tran} \xbf_j = 0\ \forall l,j \in [n]$. With  $X \in \stiefplus$ and $D_{ij} \geq 0\ \forall (i,j) \in \Omega_0(X)$, we further know that $\xbf_l^{\tran} \dbf_j \geq 0$. Therefore, we have $\xbf_l^{\tran} \dbf_j =0\ \forall l,j \in [n]$ and thus $\xbf_j^{\tran} \dbf_j = 0\ \forall j \in [n]$ and $D_{ij} = 0\ \forall(i,j) \in \Omega_0'(X)$. This means that $\LC_{\stiefplus}(X) \subseteq \LC_{\Xcal_V}(X)$. On the other hand, it is easy to see that $D \in \LC_{\Xcal_V}(X)$ must imply that $D \in \LC_{\stiefplus}(X)$. Hence, we have $\LC_{\stiefplus}(X) = \LC_{\Xcal_V}(X)$. 
Besides, by some easy calculations,  the cones
$\Kcal_{\stiefplus}(X)$ and $\Kcal_{\Xcal_V}(X)$ coincide, see
\eqref{equ:Kcal:cone} for the definition.  This completes the proof of (i). 

The proof of (ii) can be verified  since  $\TC_{\Xcal_V}(\bar X) =
\TC_{\stiefplus}(\bar X)$  and  $\Ncal_{\stiefplus}(\bar X) =
\Ncal_{\Xcal_V}(\bar X)$ and $\Ccal_{\stiefplus}(\bar X) = \Ccal_{\Xcal_V}(\bar
X)$. The details are omitted to save space. 
\end{proof}

Based on the above lemma, it is safe to  rewrite problem \eqref{equ:prob:orth+}
as the equivalent problem \eqref{equ:prob:orth+:new}. 
The reason that we prefer the latter one is that it can better motivate us to
design the exact penalty approach. 
 Simply speaking, we can afford to preserve the simpler
constraints $\obliqueplus$ in our exact penalty algorithm but the nonnegative orthogonality constraint
cannot be kept together in algorithms for \eqref{equ:prob:orth+}.
 Moreover, only one simple constraint $\|XV\|_{\Ftt} = 1$ has to be penalized in
 our approach. This kind of framework is quite different
 from traditional exact penalty approaches applied to (1.1) directly. 
}


\section{An exact penalty approach} \label{section:exact:penalty}
We now present the exact penalty properties.
 Let
 $X_{\sigma, p,q,\epsilon}$ be  a global minimizer of
 \eqref{equ:prob:orth+:new:exact:penalty} and denote  $X^{\R}_{\sigma, p, q,
 \epsilon}$ as the matrix returned by Procedure \ref{alg:round}  in
 \cref{subsection:rounding} with an input  $X_{\sigma, p,q,\epsilon}$. The
 solution quality can be further improved by solving an  auxiliary problem
 constructed from  $X^{\R}_{\sigma, p, q, \epsilon}$ as 
   \be \label{equ:prob:refinement}
 X^{\diamondsuit}_{\sigma, p, q, \epsilon} =\arg \min_{X \in \obliqueplus} f(X) \quad \st \quad X_{ij}   = 0 \ \mathrm{if}\  (i,j) \not \in\supp(X_{\sigma, p, q, \epsilon}^{\R}).
\ee

Let  $L_f \geq 0$ be the Lipschitz constant of $f$, namely,  
\be \label{equ:f:Lips}
|f(X_1) - f(X_2)|\leq L_f \|X_1 - X_2\|_{\Ftt}, \quad \forall X_1, X_2 \in \obliqueplus.
\ee 
Such $L_f$ exists since the convex hull of $\obliqueplus$ is compact. Let $\kappa_f = \chi_f/L_f$. We define
  \[ 
 \underline{\nu} = 
  \begin{cases} 
      (\sqrt{2k})^{1 - 2p}  &  \mbox{if}\ 0 < p \leq 1/2\ \mbox{and}\ \epsilon = 0,\\[4pt]
  (\kappa_f)^{1 - 2p}& \mbox{if}\ p > 1/2\ \mbox{and}\ \epsilon = 0, \\[4pt]
      \frac{\sqrt{2k}}{(\kappa_f)^{2p} - (\varrho_q \sqrt{\epsilon})^{2p}}  &  \mbox{if}\ p >0\ \mbox{and}\ 0  < \epsilon < \kappa_f^2/\varrho_q^2,
   \end{cases}
 \]
 where the constant $\varrho_q$ is defined later in \cref{lemma:rounding}.

The next theorem shows that if $\sigma$ is chosen sufficiently large, the
optimal sign matrix can be obtained from $X^{\R}_{\sigma, p,q, \epsilon}$, thus  $X^{\diamondsuit}_{\sigma, p, q, \epsilon}$ is also a solution of  \eqref{equ:prob:orth+:new}.
 
 \begin{theorem}\label{theorem:lp:exact}
 Under \cref{assump:blank}, if we choose  
 \be \label{equ:sigma:threshold}
 \sigma >  \underline{\sigma}:=  \varrho_q^{2p}L_f \underline{\nu},
 \ee
then it holds that \emph{(i)} $\sgn(X_{\sigma, p, q, \epsilon}^{\R}) \in \sgn(\Xcal^*)$;  \emph{(ii)} $X^{\diamondsuit}_{\sigma, p, q, \epsilon}$ is a   global  minimizer of problem \eqref{equ:prob:orth+:new}.
   \end{theorem}
   
We next investigate the error bound for $\stiefplus$   in
section \ref{subsection:rounding}, then give the proof of
\cref{theorem:lp:exact} for a class of  general exact
penalty model in section \ref{subsection:exact:penaly:theory}. 



\subsection{Error bound for $\stiefplus$} \label{subsection:rounding}
\rev{It is well known that the error bound plays a key role in establishing the exact penalty results, see \cite{luo1996mathematical} for more discussion. By \cite[Theorem 16.7]{luo2000error}, we know that there exist positive scalars $\rho$ and $\gamma$ such that  
$\dist(X, \stiefplus) = \|\proj_{\stiefplus}(X) - X\|_{\Ftt} \leq \rho
(\zeta_2(X))^{\gamma}, \forall X \in \obliqueplus$. However, the exponent
$\gamma$ is not immediately  clear for our case. We next show that the exponent
is $\gamma = 1/2$.}
Our key step is based on rounding Procedure \ref{alg:round}.
 The basic idea for rounding is  simply keeping one largest element in each row
 and setting the remaining elements to be zeros, and then doing normalization such that each column takes the unit norm.  

 \begin{algorithm}[!htbp]   
 \floatname{algorithm}{Procedure}
 \SetKwFor{mypost}{Do postprocessing}{}{endw}
 \SetKwFor{mywhile}{For}{do}{}
\caption{A  procedure for rounding $X \in \obliqueplus$ to be $X^{\R} \in \stiefplus$.}\label{alg:round}
\textsf{Initialization}:   Set   $H \in \Rbb^{n \times k} $ as a zero matrix. \\
\nlset{S1} For $i \in [n]$,  set 
 $H_{ij^*} = 1\ \mbox{with}\  j^*\ \mbox{is the smallest index in the set} \argmax\nolimits_{j \in [k]}  X_{ij}.
$
\\
\nlset{S2} Set  the $j$-th column of $X^{\R}$ as $\xbf^{\R}_j = \frac{\xbf_j \circ \hbf_j}{\|\xbf_j \circ \hbf_j\|},\ \ j \in [k].$\\
\nlset{S3} Reset $X^{\R} = I_{n,k}$ if $X^{\R} \not \in \stiefplus$.
%
\end{algorithm}

\begin{lemma} \label{lemma:rounding}
For any  $X \in \obliqueplus$, 
we have $X^{\R} \in \stiefplus$ and 
\be \label{equ:bound:round}
\dist(X, \stiefplus) \leq \|X^{\R} - X\|_{\Ftt}  \leq \varrho_q  \sqrt{\zeta_q(X)}, 
\ee where $\varrho_q =  \left({2k}\tilde \varrho_q/\underline \omega\right)^{\frac12}$ with $\underline \omega = {\min_{i,j \in [k]} [VV^{\tran}]_{ij}}$. Here, $\tilde \varrho_q$ is 1 if $q \geq 2$, and is $\frac{\sqrt{k} + 1}{q}$ if $1\leq q < 2$,  and is $\frac{2\sqrt{k}(\sqrt{k} + 1)}{q(q+1)}$ if $0 < q \leq 1$.
\end{lemma}

\begin{proof}
We first focus on $q = 2$. Recalling  $\|V\|_{\Ftt} = 1$ and $\underline \omega > 0$, we have  
\be \label{lemma:dist:error:bound:theta} 
\zeta_2(X)  = \sum_{j \in [k]} \xbf_j^{\tran} \Bigg(\sum_{l \in [k]\setminus\{j\}} (VV^{\tran})_{jl} \xbf_l\Bigg) \geq \underline \omega  \sum_{j \in [k]} \xbf_j^{\tran} \Bigg( \sum_{l \in [k]\setminus\{j\}}  \xbf_l \Bigg).
\ee 

The proof of \eqref{equ:bound:round} is split to two  cases.

Case I. $\zeta_2(X) \geq \underline \omega$.  Since $X^{\R} \in \stiefplus$ and $X \in \obliqueplus$, we  obtain $\| X^{\R}\|_{\Ftt}^2 = \|X\|_{\Ftt}^2 = 2k$ and thus 
$\|X -  X^{\R}\|_{\Ftt}^2  \leq 2k.$  Hence, there holds 
 $\|X -  X^{\R}\|_{\Ftt}   \leq  \sqrt{2k} \leq \varrho  \sqrt{\zeta_2(X)}$.

Case II.  $\zeta_2(X) < \underline \omega$.  
First, we prove that $X^{\R}$ generated by S2  lies in $\stiefplus$. 
Clearly, it follows from S1 that  each row of $H$ has at most one element being $1$. 
We now claim that each column of $H$ has at least one element being 1. Otherwise, without loss of generality,  we assume 
$\hbf_{1} = 0$.  This together with  S1  implies that  $X_{i1}\leq  \max_{l \in [k]\setminus\{1\}}  X_{il}$, $\forall  i \in [n]$, which with  \eqref{lemma:dist:error:bound:theta} tells that 
$\zeta_2(X) \geq \underline \omega \sum_{i \in[n]} X_{i1}\max_{l \in [k]\setminus\{j\}}  X_{il} \geq \underline \omega \sum_{i \in[n]} X_{i1}^{2} =  \underline \omega\|\xbf_{1}\|^2 = \underline \omega$. 
This gives a contradiction to $\zeta_2(X) <  \underline \omega$. 
  In summary, we know that 
 $ \|\hbf_j\|_0 \geq 1, \forall j \in [k]\  \mbox{and}\  \hbf_i^{\tran} \hbf_j = 0, \forall i, j\in [k]\  \mbox{and}\  i \neq j
 $ 
  and
  \be \label{equ:proof:lemma:rounding:x:z}
  \xbf_j \circ \hbf_j \neq 0, \quad (\xbf_j \circ \hbf_j)^{\tran} (\xbf_j \circ (\1 - \hbf_j)) = 0, \quad  \forall j \in [k].
  \ee
  Therefore, using the construction  of $X^{\R}$ in S2, we must have $X^{\R} \in \stiefplus$. 
  Using S2, \eqref{equ:proof:lemma:rounding:x:z}, and the decomposition $\xbf_j
  = \xbf_j \circ \hbf_j + \xbf_j \circ (\1 - \hbf_j)$, we  obtain  
$\|\xbf_j - \xbf^{\R}_j\|^2 \leq 2 \|\xbf_j \circ (\1 - \hbf_j)\|^2$. 
With S1, we have 
\[ \|\xbf_j \circ (\1 - \hbf_j)\|^2  =  \sum_{i \in  [n],  H_{ij} = 0} X_{ij}^2 \leq  \sum_{i \in  [n]} X_{ij} \max_{l \in [k]\setminus\{j\}} X_{il} \leq \xbf_j^{\tran} \sum_{l \in [k]\setminus\{j\}} \xbf_l,
\]
which with \eqref{lemma:dist:error:bound:theta}  implies 
$ \|X  - X^{\R}\|_{\Ftt}^2 =  \sum_{j\in [k]} \|\xbf_j - \xbf_j^{\R}\|^2 \leq  2\sum_{j \in [k]} \|\xbf_j \circ (\1 - \hbf_j)\|^2 \leq {2k\zeta_2(X)}/{\underline \omega} \leq \varrho^2  \zeta_2(X)$.
Combining the above two cases gives  \eqref{equ:bound:round} for $q = 2$. 

It is ready to prove \eqref{equ:bound:round} for general $q$.   For $X \in \obliqueplus$,  there holds that  $1 \leq \|XV\|_{\Ftt} \leq \|X\|_2 \|V\|_{\Ftt} \leq \sqrt{k}$.  We consider three cases. 
Case I.  $q \in [2, +\infty)$. It is easy to have $\zeta_q(X) \geq \zeta_2(X)$.  
Case II.  $q \in [1, 2)$.  
  We first have $\zeta_1(X) = \frac{\zeta_2(X)}{\|XV\|_{\Ftt} + 1} \geq \frac{\zeta_2(X)}{\sqrt{k} + 1}$.  Then we have 
$\zeta_q(X) = (1 + \zeta_1(X))^q - 1 \geq  q \zeta_1(X) \geq \frac{q}{\sqrt{k} + 1} \zeta_2(X)$,
where the first inequality uses the fact that $(1 + a)^q - 1 > q a$ for  $a \in (0, +\infty)$ and $q \in [1,2)$.
Case III. $q \in (0,1)$.  Since $\|XV\|_{\Ftt} = 1 +  \zeta_1(X)  \geq 1 + \frac{\zeta_1(X) }{\sqrt{k}}$, we have 
\[
\zeta_q(X) \geq \left(1 + \frac{\zeta_1(X) }{\sqrt{k}}\right)^q - 1 \geq \frac{q(q+1)}{2\sqrt{k}} \zeta_1(X) \geq \frac{q(q+1)}{2\sqrt{k}(\sqrt{k} + 1)} \zeta_2(X), 
\]
where the second inequality uses the fact that $(1 + a)^q - 1 \geq \frac{q(q+1)}{2} a$ for $a \in (0,1)$, $q \in (0,1)$. Combining the above three cases, we have 
$\zeta_2(X) \leq \tilde \varrho_q \zeta_q(X)$,
which with  \eqref{equ:bound:round} for $q = 2$   implies that  \eqref{equ:bound:round} holds for general $q$. 
\end{proof}

\rev{
We remark that the order $1/2$  in the local error bound 
\eqref{equ:bound:round} is the best.
\begin{example}
 Take $q = 2$
and $V = {1}/{\sqrt{2}} \begin{bmatrix}1 &  1\end{bmatrix}^{\tran}$.  Let $0< \epsilon \ll 1$. Consider $X
= \begin{bmatrix} \sqrt{1 - \epsilon^2 - 2 \epsilon} & \epsilon \\ \epsilon &
\sqrt{1 - \epsilon^2 - \epsilon} \\ \sqrt{2 \epsilon} &
\sqrt{\epsilon}\end{bmatrix}$. We have $\proj_{\stiefplus}(X) =   \begin{bmatrix} \sqrt{1
- \epsilon^2 - 2\epsilon}/{\sqrt{1 - \epsilon^2}}  &  0  \\ 0  & 1 \\ {\sqrt{2
\epsilon}}/{\sqrt{1 - \epsilon^2}} & 0 \end{bmatrix}$ and $\dist(X, \stiefplus) =  \|\proj_{\stiefplus}(X)  - X\|_{\Ftt}  \approx \sqrt{\epsilon}$ while $\zeta_2(X) =  x_1^{\tran} x_2  \approx 4 \epsilon$. 
\end{example}
}


\subsection{A general exact penalty model} \label{subsection:exact:penaly:theory}
%

Let $0\leq Q_0 < \chi_f/L_f$ be a constant and  $\Psi: [Q_0,
+\infty) \rightarrow \Rbb_+$  be  strictly increasing. Choose 
$Q: \obliqueplus \rightarrow \Rbb_+$ such that 
\begin{subequations}
\begin{align}
& Q(X) \geq  \varrho_q \sqrt{\zeta_q(X)} \quad \forall X \in \obliqueplus, \label{equ:QX:1}\\
 & Q(X) \equiv Q_0 \quad \forall X \in \stiefplus, \quad Q(X) \geq Q_0\quad \forall X \in \obliqueplus. \label{equ:QX:2}
\end{align}
\end{subequations}

Note that \eqref{equ:QX:1} and \eqref{equ:bound:round} imply that 
\be
Q(X) \geq \|X^{\R} - X\|_{\Ftt} \geq  \dist(X, \stiefplus),\quad \forall X \in \obliqueplus. \label{equ:QX:0}  
\ee
Our general penalty model, including \eqref{equ:prob:orth+:new:exact:penalty} as a special case,  is given as
\be \label{prob:exact:penalty:Phi}
\min_{X \in \obliqueplus} f(X) + \sigma \Psi(Q(X)).
\ee
Let $X_{\sigma, \Psi}$ be a global minimizer of \eqref{prob:exact:penalty:Phi}, 
 and $X^{\R}_{\sigma, \Psi}$ be the matrix returned by Procedure
\ref{alg:round} with an input  $X_{\sigma, \Psi}$.

\begin{lemma}\label{lemma:genera:exact:penalty}
For the penalty model \eqref{prob:exact:penalty:Phi}, 
it holds 
\be\label{equ:lemma:genera:exact:penalty:00}
f(X^*) \leq f(X^{\R}_{\sigma, \Psi}) \leq f(X^*)  + L_f \Upsilon_{\sigma, Q_0, \Psi},
\ee
where $X^*$ is a global minimizer of problem \eqref{equ:prob:orth+:new} and
\be\label{equ:upsilon}
\Upsilon_{\sigma, Q_0, \Psi}:= \max_{z \in \Rbb}\,\,  z  \ 
 \st \   \Psi(z) \leq \Psi(Q_0) + \frac{L_f}{\sigma} z, 0\leq z \leq \Psi^{-1}\big(\Psi(Q_0)  + {\sqrt{2k}L_f}/{\sigma}\big). 
\ee
\end{lemma}
\begin{proof}
    Using  the Lipschitz continuity of $f$ in \eqref{equ:f:Lips}, we have 
\be\label{equ:genera:exact:penalty:a00}
f(X^{\R}_{\sigma, \Psi}) \leq f(X_{\sigma, \Psi})  + L_f \|X^{\R}_{\sigma, \Psi} - X_{\sigma, \Psi}\|_{\Ftt} \leq f(X_{\sigma, \Psi}) + L_f Q(X_{\sigma, \Psi}),
\ee
where the second inequality is due to \eqref{equ:QX:0}. 
By the optimality of $X_{\sigma, \Psi}$, we obtain 
\be\label{equ:genera:exact:penalty:b00}
f(X_{\sigma, \Psi}) + \sigma \Psi\left(Q(X_{\sigma, \Psi})\right) \leq f(X) + \sigma \Psi\left(Q(X)\right) = f(X) + \sigma \Psi(Q_0) \quad \forall X \in \Xcal_V. 
\ee
Taking $X = X^*$ in \eqref{equ:genera:exact:penalty:b00} and using
the strictly increasing property of $\Psi$, we have 
$f(X_{\sigma, \Psi}) \leq f(X^*)$. Hence, we know from \eqref{equ:genera:exact:penalty:a00} that 
\be \label{equ:genera:exact:penalty:f00}
f(X^*) \leq f(X^{\R}_{\sigma, \Psi})  \leq f(X^*) + L_f Q(X_{\sigma, \Psi}).
\ee

The remaining is to estimate $Q(X_{\sigma, \Psi})$. Taking $X$ to be
$X^{\R}_{\sigma, \Psi}$ in \eqref{equ:genera:exact:penalty:b00}, we get 
\be\label{equ:genera:exact:penalty:c00} 
\Psi\left(Q(X_{\sigma, \Psi})\right) \leq \Psi(Q_0) + \frac{f(X^{\R}_{\sigma, \Psi}) - f(X_{\sigma, \Psi}) }{\sigma} \leq \Psi(Q_0) + \frac{L_f \|X^{\R}_{\sigma, \Psi} - X_{\sigma, \Psi}\|_{\Ftt}}{\sigma}, 
\ee
where the second inequality is due to \eqref{equ:f:Lips}. Since $X_{\sigma,
\Psi} \in \obliqueplus$, it is easy to see that $\|X^{\R}_{\sigma, \Psi} -
X_{\sigma, \Psi}\|_{\Ftt} \leq  \sqrt{2k}$.  Thus, we have from \eqref{equ:genera:exact:penalty:c00} that 
$\Psi\left(Q(X_{\sigma, \Psi})\right) \leq \Psi(Q_0)  +
{\sqrt{2k}L_f}/{\sigma}$. Since $\Psi$ is  strictly increasing, we obtain 
\be \label{equ:genera:exact:penalty:c02} 
Q(X_{\sigma, \Psi}) \leq \Psi^{-1}\left(\Psi(Q_0)  + {\sqrt{2k}L_f}/{\sigma}\right).
\ee
On the other hand,   recalling \eqref{equ:QX:0}, we have from \eqref{equ:genera:exact:penalty:c00}  that 
$\Psi\left(Q(X_{\sigma, \Psi})\right) \leq    \Psi(Q_0)  +  \frac{L_f}{\sigma}  Q(X_{\sigma, \Psi})$,
which together with \eqref{equ:genera:exact:penalty:c02} and  \eqref{equ:genera:exact:penalty:f00} establishes 
\eqref{equ:lemma:genera:exact:penalty:00}.  
\end{proof}

  Let $X^{\diamondsuit}_{\sigma, \Psi}$ be a global minimizer of the
  problem  \eqref{equ:prob:refinement} with $X^{\R}_{\sigma, p, q, \epsilon}=X_{\sigma, \Psi}^{\R}$. We now have  the following  exact penalty
  property.
\begin{theorem}\label{thm:exact:penalty:new}
Suppose \cref{assump:blank}  holds  and 
 $\sigma > 0$ is chosen such that 
\be \label{equ:exact:penaly:Upsilon:condition}
 \Upsilon_{\sigma, Q_0, \Psi} < \kappa_f. 
\ee
Then it holds that \emph{(i)} $\sgn(X^{\R}_{\sigma, \Psi}) \in \sgn(\Xcal^{*})$; \emph{(ii)} $X^{\diamondsuit}_{\sigma, \Psi}$ is a global minimizer of problem \eqref{equ:prob:orth+:new}, namely, $f(X^{\diamondsuit}_{\sigma, \Psi}) = f(X^*)$.
\end{theorem}
\begin{proof}

We first claim that  $\sgn(X^{\R}_{\sigma, \Psi}) \in \sgn(\Xcal^{*})$. Otherwise, 
it follows from \cref{assump:blank}   that $f(X^{\R}_{\sigma, \Psi}) \geq f(X^*) + \chi_f$.  
By using \eqref{equ:lemma:genera:exact:penalty:00} and $\kappa_f = \chi_f/L_f$,
we thus have $\Upsilon_{\sigma, Q_0, \Psi} \geq \kappa_f$,  which makes a
contradiction to  
\eqref{equ:exact:penaly:Upsilon:condition}.   Using  $\sgn(X^{\R}_{\sigma, \Psi}) \in \sgn(\Xcal^{*})$ and the definition of   $X^{\diamondsuit}_{\sigma, \Psi}$, see problem \eqref{equ:prob:refinement} with $X^{\R}_{\sigma, p, q, \epsilon}$  being $X_{\sigma, \Psi}^{\R}$, we know that $X^{\diamondsuit}_{\sigma, \Psi}$ is a global minimizer of problem \eqref{equ:prob:orth+:new}. The proof is completed. 
\end{proof}


It follows from \eqref{equ:upsilon} that $\Upsilon_{\sigma, Q_0, \Psi} \leq \Psi^{-1}\big(\Psi(Q_0)  + {\sqrt{2k}L_f}/{\sigma}\big)$. To make  \eqref{equ:exact:penaly:Upsilon:condition} hold, we can choose 
 $0 \leq Q_0 < \kappa_f$ and  $\sigma >  \sqrt{2k} L_f \big(\Psi(\kappa_f) - \Psi(Q_0)\big)^{-1}$.
For  some particular $\Psi(\cdot)$, we next show that  this lower bound   can be improved.


\begin{myproof}
Let us choose $Q(X) = \varrho_q  \sqrt{\zeta_q(X) + \epsilon}$ and $\Psi(z) = \left({z}/{\varrho_q}\right)^{2p}$ with 
 $0\leq \epsilon <  \kappa_f^2/\varrho_q^2$ and $Q_0 =\varrho_q \sqrt{\epsilon}$. By \cref{thm:exact:penalty:new}, we only need to prove $\Upsilon_{\sigma, Q_0, \Psi} < \kappa_f$ if $\sigma > \varrho_q^{2p}L_f \underline{\nu}$. We consider three cases. 

Case I.  $\epsilon = 0$ and $0< p \leq 1/2$.  Since  $\sigma > \varrho_q^{2p}L_f \underline{\nu} =(\sqrt{2k})^{1 - 2p} \varrho_q^{2p}L_f$, we have from $\Psi(z) \leq \Psi(Q_0) + \frac{L_f}{\sigma} z$ that $z =  0$ or $z > \sqrt{2k}$  and have from $0\leq z \leq \Psi^{-1}(\Psi(Q_0)  + {\sqrt{2k}L_f}/{\sigma})$ that $0 \leq z < \sqrt{2k}$. By definition \eqref{equ:upsilon}, we have $\Upsilon_{\sigma, Q_0, \Psi} = 0$. 

Case II. $\epsilon = 0$ and $p > 1/2$.  Using the definition of $\chi_f$ in \cref{assump:blank}, \eqref{equ:f:Lips} and $\|X - Y\| \leq \sqrt{2k}$ for $X, Y \in \obliqueplus$, we have  $\chi_f \leq \sqrt{2k}L_f$, i.e., $\kappa_f \leq \sqrt{2k}$.  By    $\sigma > \varrho_q^{2p}L_f \underline{\nu} =(\kappa_f)^{1 - 2p} \varrho_q^{2p}L_f$,  it is easy to obtain from \eqref{equ:upsilon} that $\Upsilon_{\sigma, Q_0, \Psi} < \kappa_f$. 

Case III.  $0< \epsilon < \kappa_f^2/\varrho_q^2$, $p > 0$. In this case, we have $\varrho_q^{2p}L_f \underline{\nu} = \sqrt{2k} L_f \big(\Psi(\kappa_f) - \Psi(Q_0)\big)^{-1}$. Thus,  we have $\Upsilon_{\sigma, Q_0, \Psi} < \kappa_f$ by \eqref{equ:upsilon}. 
\end{myproof}

\begin{remark}
 The threshold $\underline{\sigma}$ of the  exact penalty parameter  
 depends on  the parameter $\kappa_f = \chi_f/L_f$, except that for model
 \eqref{equ:prob:orth+:new:exact:penalty} with $p \in (0,1/2]$ and $\epsilon =
 0$, $\underline{\sigma}$ is independent of $\kappa_f$ but the corresponding
 penalty term is nonsmooth.  Usually, estimating $L_f$ is possible for the
 instances such as the orthogonal nonnegative matrix factorization models 
 \eqref{equ:prob:onmf} and \eqref{equ:prob:onmf:XXt}.  However, computing
 $\chi_f$ is hard since it needs to know $f^*$ in advance and solves an
 optimization problem.  In fact, calculating  a threshold of the exact penalty
 parameter is always not easy, see \cite{bertsekas1996constrained, luo1996exact,
 friedlander2008exact, chen2016penalty} for some convex and nonconvex examples.
 In practice, we simply solve approximately a series of problems
 \eqref{equ:prob:orth+:new:exact:penalty}  with  an increasing   $\sigma$; see
 section \ref{section:penaly:alg} for a detailed description.
\end{remark}

A few more remarks on the exact penalty model \eqref{equ:prob:orth+:new:exact:penalty}  are listed in order.  First,  to make the objective function in problem \eqref{equ:prob:orth+:new:exact:penalty} smooth, we need to choose $\epsilon \in (0, \kappa_f^2/\varrho_q^2)$ for $p \in (0,1)$. As for $p \in [1, +\infty)$, we can simply choose $\epsilon = 0$.  Second,  by directly using the results in \cite[Lemma 3.1]{chen2016penalty}, we can show that a global  minimizer of  
\eqref{equ:prob:orth+:new:exact:penalty} with $p = 1/2$ and $\epsilon = 0$ is
also a global minimizer of \eqref{equ:prob:orth+} under the condition that
$\sigma > \varrho_q L_f$.  However, the results therein do not apply to the
general $\Psi(\cdot)$ and $Q_0$. By contrast, our results in
\cref{thm:exact:penalty:new} or \cref{theorem:lp:exact} allow more flexible
choices of $\Psi(\cdot)$ and $Q_0$ or $p$ and $\epsilon$. Third, the multiple
spherical constraints   in  model \eqref{equ:prob:orth+:new:exact:penalty} are
not only important to establish the exact penalty property but also make model
\eqref{equ:prob:orth+:new:exact:penalty}  working over a compact set.  It should
be mentioned that for a variant of  ONMF problem \eqref{equ:prob:onmf},
\cite{wang2019clustering} proposed an exact penalty model without keeping the
multiple spherical constraints. However, their results only hold on this special
formulation  \eqref{equ:prob:onmf} rather than the general problem \eqref{equ:prob:orth+}.  Besides, Gao  et al. \cite{gao2018parallelizable} used a customized augmented Lagrangian type method  to solve optimization with orthogonality constraints  but without the nonnegative constraints. The multiple spherical constraints are also kept therein to make their method more robust.  However, their method  can not be directly used to  solve problem \eqref{equ:prob:orth+} or \eqref{equ:prob:orth+:new} since the nonnegative constraints, which were not considered therein, make the problem totally different.

\section{A practical exact penalty algorithm} \label{section:penaly:alg}
We now  focus on the exact penalty model \eqref{equ:prob:orth+:new:exact:penalty}. A practical exact penalty method, named as EP4Orth+, is  presented in  \cref{alg:feasiblePenalty}. We adopt the way in \cite{chen2016penalty} to choose a feasible initial point $X^{\feas}$. In each iteration, the penalty parameter $\sigma$ is dynamically increased and we find an approximate stationary point $X^t$
satisfying the approximate first-order optimality condition \eqref{equ:palg:kkt:condition}  and the sufficient descent condition \eqref{equ:palg:decrease}. Such $X^t$ can be found in a finite number of iterations; see  section \ref{section:retr:based} for more discussion.
In practice, we utilize the  nonconvex projection gradient method \eqref{equ:gp:h} if $\zeta(X^{t,0}) = \|X^{t,0}V\|_{\Ftt}^2 - 1 > \bar{\zeta}$,
otherwise switch to the  second-order method, namely, \cref{alg:obliqueplus:alg:2nd}, developed in section
\ref{section:retr:based}. 
 To obtain an exact orthogonal nonnegative matrix and improve the solution quality, we perform a postprocessing procedure at the end of the algorithm. 

 \begin{algorithm}[!htbp]  
 \SetKwFor{mypost}{Postprocessing}{}{endw}
 \SetKwFor{mywhile}{For}{do}{}
\caption{EP4Orth+: A practical exact penalty method for solving \eqref{equ:prob:orth+:new}}\label{alg:feasiblePenalty}
\textsf{Initialization:} Choose $X^0 \in \obliqueplus$, $X^{\feas} \in \stiefplus$, $\sigma_0 > 0$,   $p,q \in (0, +\infty)$, $ \tol^{\feas}, \varepsilon_{0}^{\grad}, \varepsilon_{\min}^{\grad} \in [0,1)$, a positive integer $t_{\max}$. Choose $\gamma_2 > 1$, set $\epsilon_0 >0$, $\gamma_1 \in (0,\gamma_2^{-1/p})$  if $p \in (0,1)$ and $\epsilon_0 = 0$, $\gamma_1 = 1$ if $p \ge 1$. Let $\eta \in (0,\gamma_2^{-1} \gamma_1^{1-p})$.   Set $X^{0,0} = X^{0}$.\\
\mywhile{$t = 0,1,2, \ldots, t_{\max}$}{
If $P_{\sigma_{t},p,q,\epsilon_t}(X^{t,0}) > P_{\sigma_{t},p,q,\epsilon_t}(X^{\feas})$, set $X^{t,0} = X^{\feas}$.

Starting from $X^{t, 0}$, we  find an approximate stationary point $X^{t}$ of  \eqref{equ:prob:orth+:new:exact:penalty}  with $\sigma = \sigma_t$ and $\epsilon = \epsilon_t$ such that 
\begin{align} \label{equ:palg:kkt:condition}
&\|\min(X^{t}, \rgrad P_{\sigma_t, p, q,\epsilon_t}(X^{t}))\|_{\Ftt} \leq \varepsilon_{t}^{\grad},\\
 &P_{\sigma_{t},p,q,\epsilon_t}(X^{t}) \leq  P_{\sigma_{t},p,q,\epsilon_t}(X^{t, 0}). \label{equ:palg:decrease}
 \end{align}

 \If{$\|X^tV\|_{\Ftt}^2 - 1 \leq \tol^{\feas}$}{
\textbf{break}}
 Set $
 \epsilon_{t+1} = \gamma_1 \epsilon_t$, $\sigma_{t + 1} = \gamma_2 \sigma_t$,  $\varepsilon_{t+1}^{\grad} = \max\{\eta  \varepsilon_{t}^{\grad}, \varepsilon_{\min}^{\grad}\}$ and $X^{t + 1,0} = X^t$.
}
\tcc{Postprocessing}
Set $X^{\R} = (X^t)^{\R}$ using Procedure \ref{alg:round} and solve \eqref{equ:prob:refinement}  approximately with $X^{\R}_{\sigma, p, q, \epsilon}=X^{\R}$ to get $X^{\diamondsuit}$ such that $f(X^{\diamondsuit}) \leq f(X^{\R})$.
\end{algorithm}
It should be mentioned that  the postprocessing procedure is always easy to perform.
Consider a separable function $f(X) = \sum_{j = 1}^k
f_j(\xbf_j)$, where $f_j: \Rbb^n \rightarrow \Rbb$. Let  $\xbf_j^{\R}$ be the $j$-th column of $X^{\R}$.
The corresponding problem
\eqref{equ:prob:refinement} is split to the form of 
 \be 
  \min_{\xbf_j \in \Rbb^n}\  f_j(\xbf_j) \quad \st \quad  \xbf_j^{\tran} \xbf_j  = 1,  \xbf_j \geq 0, (\xbf_j)_i = 0 \ \mathrm{if}\  i \not \in\supp(\xbf_j^{\R}).\nn
\ee
$(\mathrm{i})$  If $f(X)$ is $-\langle C, X\rangle$ as in the   
K-indicators model \eqref{equ:prob:k:indicator} with fixed $Y$, the $j$-th  column of the
global minimizer $X^{\Diamond}$ is $\xbf_j^{\Diamond} = \proj_{\obliqueplus}(\cbf_j \circ \hbf_j^{\R})$, where $\hbf_j^{\R}$ is the $j$-th column of $\sgn(X^{\R})$. 
$(\mathrm{ii})$  Consider $f(X) = - \mtr (X^{\tran} M X)$ with $M = M^{\tran}
\geq 0$ in the ONMF models
\eqref{equ:prob:onmf} and \eqref{equ:prob:onmf:XXt}.  Then 
$(\xbf_j^{\Diamond})_i = 0$ if $i \not \in
\supp(\xbf_j^{\R})$ and $(\xbf_j^{\Diamond})_{\supp(\xbf_j^{\R})}$ is the
dominant eigenvector of  $M_j$,  which  is a principal submatrix of $M$ whose
rows and columns indices are both $\supp(\xbf_j^{\R})$.  Since $M_j \geq 0$,
there always exists a nonnegative dominant eigenvector due to Perron-Frobenius Theorem, see \cite[Theorem 1.1]{chang2008perron}. 
When $f$ is a  general smooth function, we can use the  nonconvex gradient projection method to get an approximate stationary point $X^{\Diamond}$ with $f(X^{\Diamond}) \leq f(
X^{\R})$. 

Next, we study the asymptotic convergence of \cref{alg:feasiblePenalty} without postprocessing.

\rev{
\begin{theorem}\label{thm:convergence}
Let $\{X^t\}$ be the sequence generated by \cref{alg:feasiblePenalty} with $t_{\max}=\infty$, $\tol^{\feas} = -1$ and $\varepsilon^{\grad}_{\min}=0$. If $X^{\infty}$ is a limit point  of $\{X^t\}$, then $X^{\infty}$ is   a weakly stationary point of problem \eqref{equ:prob:orth+:new}.
\end{theorem}
\begin{proof}
Note that $t_{\max}=\infty$ and $\tol^{\feas} = -1$, the algorithm does not stop within a finite number of iterations. Since  $\{X^t\}$ is bounded, without loss of generality, throughout the proof we assume that $\{X^t\}$ converges to $X^{\infty}$.  By \eqref{equ:palg:decrease} and $P_{\sigma_{t},p,q,\epsilon_t}(X^{t, 0})
 \leq  P_{\sigma_{t},p,q,\epsilon_t}(X^{\feas})$,   we obtain
\be\label{equ:theorem:kkt:alg:00}
f(X^t) + \sigma_t (\zeta_q(X^t) + \epsilon_t)^p \leq f(X^{\feas}) + \sigma_t \epsilon_t^p. 
\ee
If $p \geq 1$, it holds $\epsilon_t = 0$ and thus $\zeta_p(X^t) \rightarrow 0$.
If $p \in (0,1)$,  using   $(a + b)^p - a^p \geq  (1 - 2^{-p}) b^p$ for $b > a >
0$ with $a=\epsilon_t$ and $b = \zeta_q(X)$, we also have  $\zeta_p(X^t)
\rightarrow 0$ from \eqref{equ:theorem:kkt:alg:00}.   It follows from the proof
of \cref{lemma:rounding} that $\zeta_2(X^t) \leq \tilde \varrho_q \zeta_q(X^t)$ and thus  $\zeta_2(X^t) \rightarrow 0$. 


For simplicity of notation, we  below omit the $p$ and $q$ in the subscripts of $P_{\sigma_t, p, q, \epsilon_t}$. Denote $c_t :=  pq(\zeta_q(X^t)+\epsilon_t)^{p-1}\|XV\|_{\Ftt}^{q-2}$.
 Some easy calculations yield
\be \label{equ:rgrad:P}
\rgrad P_{\sigma_t, \epsilon_t}(X^t) = \rgrad f(X^t) + \sigma_t  \rgrad (\zeta_q(X^t) + \epsilon_t)^p
\ee 
 with 
\be\label{equ:thm:rgrad:zeta:epsilon}
\rgrad (\zeta_q(X^t)+ \epsilon_t)^p = c_t X^t \left(\off(VV^{\tran}) -  \Diag\left( \left((X^t)^{\tran}X^t - I_k\right)VV^{\tran}\right) \right).  
\ee
Denote $\bar \omega := \max_{i,j \in [k]} [VV^{\tran}]_{ij}$.  For each $X^t$, we have  
\begin{align}
&\underline \omega \max_{l \in [k]\setminus\{j\}} X_{il}^t \leq [X^t \off(VV^{\tran})]_{ij} \leq (k-1)\bar \omega  \max_{l \in [k]\setminus\{j\}} X_{il}^t \quad \forall (i,j)\in [n] \times [k], \label{equ:thm:bound1}  \\
&0 \leq [X^t \Diag\left( \left((X^t)^{\tran}X^t - I_k\right)VV^{\tran}\right)]_{ij}\le X^t_{ij}\zeta_2(X^t)  \quad \forall (i,j)\in [n] \times [k].
\label{equ:thm:bound2}
\end{align}

We consider the following two  cases. 

Case I. The sequence $\{\sigma_t c_t\}$ is unbounded.
Since $\{X^t\}$ converges to $X^{\infty}$,  there exists sufficiently large integer $T_1$ such that for every $t>T_1$, there holds that 
\be\label{equ:thm:Xbound1}
\zeta_2(X^t)\le \frac{1}{4}\underline \omega X^{\infty}_{\min}\ \text{ and }\ X^t_{ij} \geq \frac12 X^{\infty}_{\min}\  \forall (i,j) \in \supp(X^{\infty}),
\ee
where $X^{\infty}_{\min} := \min_{(i,j) \in \supp(X^{\infty})} X^{\infty}$.
For any $(i,j) \in \Omega'_0(X^{\infty})$,  noting that $X^{\infty} \in  \stiefplus$, with \eqref{equ:thm:Xbound1}, for $t > T_1$, we have that
$\max_{l \in [k]\setminus\{j\}} X_{il}^t\geq X_{\min}^{\infty}/2$. With the first assertion in \eqref{equ:thm:bound1},   for $t > T_1$ and $(i,j) \in \Omega'_0(X^{\infty})$, there holds that 
\be\label{equ:thm:Xbound3}
\left[X^t\off(VV^{\tran})\right]_{ij}\ge \frac12 \underline \omega X^{\infty}_{\min}.
\ee
With \eqref{equ:thm:bound2} and the first assertion in \eqref{equ:thm:Xbound1},  and noting $X_{ij}^t \leq 1$, we derive  for $t > T_1$ 
 and $(i,j) \in \Omega'_0(X^{\infty})$ that 
\be\label{equ:thm:Xbound2}
[X^t \Diag\left( \left((X^t)^{\tran}X^t - I_k\right)VV^{\tran}\right)]_{ij}\le \frac14 \underline\omega X^{\infty}_{\min}.
\ee
Combining  \eqref{equ:thm:Xbound3}, \eqref{equ:thm:Xbound2} and \eqref{equ:thm:rgrad:zeta:epsilon}, for $t > T$, we have 
\be\label{equ:thm:Xbound4}
\sigma_t  [\rgrad (\zeta_q(X^t) + \epsilon_t)^p]_{ij}  \ge \frac14\underline \omega X^{\infty}_{\min} \sigma_t c_t \quad   \forall (i,j) \in \Omega'_0(X^{\infty}),
\ee
which with the unboundness of $\{\sigma_t c_{\epsilon_t}(X^t)\}$ implies
$\limsup\limits_{t\to\infty}\sigma_t  [\rgrad (\zeta_q(X^t) +
\epsilon_t)^p]_{ij} = \infty,\ \forall(i,j) \in \Omega'_0(X^{\infty}).$
Since $\rgrad f(X^t)$ is bounded and due to  \eqref{equ:rgrad:P}, we finally have
$\limsup_{t\to\infty}\ [\rgrad P_{\sigma_t,\epsilon_t}(X^t)]_{ij}=\infty\  \forall(i,j) \in \Omega'_0(X^{\infty}),$
which with  \eqref{equ:palg:kkt:condition} implies that there exists sufficiently large integer $T_2$ and subindices $\{t_l\}$ with $t_l > T_2$ such that 
\[X^{t_l}_{ij} \leq \varepsilon_{t_l}^{\grad}\quad \forall (i,j) \in  \Omega'_0(X^{\infty}).
\]
Using \eqref{equ:thm:rgrad:zeta:epsilon}, \eqref{equ:thm:bound1} and \eqref{equ:thm:bound2}, we obtain that for any $(i,j) \in \supp(X^\infty)$,
\be\label{equ:thm:boundcase1}
 -\sigma_{t_l} c_{t_l} \zeta_2(X^{t_l})\leq [\sigma_{t_l}\rgrad (\zeta_q(X^{t_l})+ \epsilon_{t_l})^p]_{ij}\leq (k-1) \bar \omega\sigma_{t_l} c_{t_l} \varepsilon_{t_l}^{\grad}.
\ee
With the choice of $\eta$, it is easy to verify that $\sigma_{t_l} c_{t_l}  \varepsilon_{t_l}^{\grad} \rightarrow 0$.
Since 
\[
\zeta_2(X^t)=\mtr(((X^t)^\top (X^t)-I_k)VV^\top)\le k^2n\bar\omega(\varepsilon^{\grad}_t+ \max_{\begin{subarray}{c} (i,j_1),(i,j_2)\in\Omega''_0(X^\infty) \\ j_1\neq j_2\end{subarray} } X^t_{ij_1}X^t_{ij_2}),
\]
we can show the leftmost term of \eqref{equ:thm:boundcase1} tends to $0$ by proving that for any $i\in[n]$, $j_1,j_2\in[k]$ with $j_1\neq j_2$ such that $\|X^\infty_{i,:}\|=0$, there holds that 
$\lim_{l\to\infty} \sigma_{t_l}c_{t_l} X^{t_l}_{ij_1}X^{t_l}_{ij_2}=0.$
Obviously, we can focus on the case in which $\min\{X^{t_l}_{ij_1},X^{t_l}_{ij_2}\}> \varepsilon^{\grad}_{t_l}$. The approximate optimality conditions \eqref{equ:palg:kkt:condition} together with \eqref{equ:rgrad:P} and \eqref{equ:thm:rgrad:zeta:epsilon} gives
\[
[\sigma_{t_l}c_{t_l}  X^{t_l} \off(VV^{\tran})]_{ij_1}\le \varepsilon^{\grad}_{t_l}+\big|[\rgrad f(X^{t_l})]_{ij_1}\big|+\sigma_{t_l}c_{t_l} \zeta_2(X^{t_l})X^{t_l}_{ij_1}.
\]
We have from \eqref{equ:theorem:kkt:alg:00} and the choice of $\gamma_1$ that $\sigma_{t_l}pq(\zeta_q(X^{t_l})+\epsilon_{t_l})^{p}\le \sqrt{2k}L_f+\sigma_t\epsilon_t^p$ is bounded. It follows from the proof of Lemma \ref{lemma:rounding} that $\sigma_{t_l}c_{t_l}\zeta_2(X^{t_l})\le \sigma_{t_l}pq(\zeta_q(X^{t_l})+\epsilon_{t_l})^{p}\|X^{t_l}V\|_{\Ftt}^{q-2}\tilde\varrho_q$
is also bounded. Thus  
\[
\lim_{l\to\infty}\underline\omega\sigma_{t_l}c_{t_l} X^{t_l}_{ij_1}X^{t_l}_{ij_2}\le \lim_{l\to\infty}  X^{t_l}_{ij_1}[\sigma_{t_l}c_{t_l} X^{t_l} \off(VV^{\tran})]_{ij_1}=0.
\]
Together with \eqref{equ:thm:boundcase1}, the previous equation implies 
\be \label{equ:thm:liminf:case1}
\lim_{l \to \infty}\ [\sigma_{t_l}\rgrad (\zeta_q(X^{t_l})+ \epsilon_{t_l})^p]_{ij} = 0\quad  \forall (i,j) \in \supp(X^\infty).
\ee
On the other hand, it follows from \eqref{equ:palg:kkt:condition} that  
\be\label{equ:thm:limsup:case1}
\lim_{l\to\infty}[\rgrad P_{\sigma_{t_l},\epsilon_{t_l}}(X^{t_l})]_{ij}=0 \quad  \forall (i,j) \in \supp(X^\infty).
\ee
Combining \eqref{equ:thm:liminf:case1} and \eqref{equ:thm:limsup:case1}, we have from \eqref{equ:rgrad:P} that  $\lim\limits_{l\to\infty}[\rgrad f(X^{t_l})]_{ij} = 0\ \forall (i,j) \in \supp(X^\infty)$.  Considering that $X^{t_l} \to X^{\infty}$, we arrive at the conclusion in this case.

Case II. The sequence $\{\sigma_t c_t\}$ is bounded by a constant, say, $\bar N$. Similar to Case I, for any  $\mu\in(0,1)$, there exists  sufficiently large integer $T_3$ such that for every $t>T_3$ and $(i,j) \in \supp(X^{\infty})$, there holds that
\be  \label{equ:thm:zeta2:bound:case2}
\zeta_2(X^t)\le  \mu \min\{\underline \omega X^{\infty}_{\min},1\}
 \text{ and }\ X^t_{ij} \geq \frac12 X^{\infty}_{\min}\ge \varepsilon^{\grad}_t\quad  \forall (i,j) \in \supp(X^{\infty}). 
 \ee
 The second assertion above together with \eqref{equ:palg:kkt:condition} tells that for $t > T_3$ there holds that 
 \be  \label{equ:thm:case2:00}
 |[\rgrad P_{\sigma_t,\epsilon_t}(X^t)]_{ij}| \le  \varepsilon^{\grad}_t\quad  \forall (i,j) \in \supp(X^{\infty}). 
 \ee
Similar to Case I, for $(i,j) \in \Omega'_0(X^{\infty})$ and $t > T_3$, we have 
$\max_{l \in [k]\setminus\{j\}} X_{il}^t\geq X_{\min}^{\infty}/2$ and thus  
\be \label{equ:thm:Xbound5}
X^t_{ij} \leq \frac{2}{X_{\min}^{\infty}} \max_{l \in [k]\setminus\{j\}} X_{il}^t X_{ij}^t\leq \frac{2}{X_{\min}^{\infty}\underline \omega} [VV^{\tran}]_{j'j}(\xbf_{j'}^t)^{\tran} \xbf_j^t\leq \frac{2 \zeta_2(X^t)}{X_{\min}^{\infty}\underline \omega},
\ee
where $j'=\argmax_{l \in [k]\setminus\{j\}}X_{il}^t$ and the last inequality uses the fact that  $\zeta_2(X^t) = \sum_{i,j \in [k], i \neq j}[VV^{\tran}]_{ij} ((\xbf_i^t)^{\tran} \xbf_j^t)$ which appears in the proof of  \cref{lemma:key:observation:V}.   By the first assertion in \eqref{equ:thm:zeta2:bound:case2} and \eqref{equ:thm:Xbound5}, we have  for $t > T_3$ that $X_{ij}^t \le  2 \mu\,\forall (i,j) \in \Omega'_0(X^{\infty})$.
Noting that $X^{\infty} \in \stiefplus$, this together with \eqref{equ:thm:bound1}  implies that 
\be\label{equ:thm:Xbound1:case2}
\left[X^t\off(VV^{\tran})\right]_{ij} \le 2(k-1)\bar \omega \mu \quad \forall (i,j) \in \supp(X^{\infty}).
\ee
Again using  \eqref{equ:thm:bound2} and noting that $X_{ij}^t \leq 1$ and \eqref{equ:thm:zeta2:bound:case2},  we have 
\be\label{equ:thm:Xbound2:case2}
[X^t \Diag\left( \left((X^t)^{\tran}X^t - I_k\right)VV^{\tran}\right)]_{ij}\le  \mu   \quad \forall (i,j) \in \supp(X^{\infty}).
\ee
Combining \eqref{equ:thm:Xbound1:case2} and \eqref{equ:thm:Xbound2:case2}, we have from \eqref{equ:thm:rgrad:zeta:epsilon} that for $t > T_3$ there holds that 
\be
\left|[\sigma_t  \rgrad (\zeta_q(X^t) + \epsilon_t)^p]_{ij}\right|
\leq \bar N  \left(2(k-1)\bar \omega + 1\right) \mu \quad \forall (i,j) \in \supp(X^{\infty}).
\ee
Consequently, by \eqref{equ:rgrad:P} and \eqref{equ:thm:case2:00}, for $t > T_3$,  we have  for each $(i,j)\in \supp(X^\infty)$ that 
$\left|[\rgrad f(X^t)]_{ij}\right|
\le \bar N\left(2(k-1)\bar \omega + 1\right)\mu+ \varepsilon^{\grad}_t$.
 Due to the arbitrariness of $\mu$, we conclude from $X^t\to X^{\infty}$ that $|[\rgrad f(X^\infty)]_{ij}|\le \lim_{t\to\infty} \varepsilon^{\grad}_t=0\,\forall  (i,j)\in \supp(X^\infty) $. The proof is completed.
 \end{proof}


%
%


\begin{example}
For  problem \eqref{equ:prob:orth+:new} with $f(X) = \langle C, X\rangle$ and  $C\in \Rbb^{3 \times 2}$ with $C_{11} = C_{12} = -1$ and $C_{21} = C_{22} = C_{31} = C_{32} =  0$.
Consider \cref{alg:feasiblePenalty} with $V = 1/\sqrt{k} \begin{bmatrix} 1 & 1 & 1\end{bmatrix}^{\tran}$, $p = 1$, $\epsilon_0 = 0$, $\tol^{\feas} = 0$ and  $\tol^{\sub}_t = 0$ for all $t \geq 0$.  
By some easy calculations, we see that $X^t$ with $X^t_{11} = X^t_{12} =  {1}/{\sigma_t}$, $X_{22} = X_{31} = 0$ and $X_{21} = X_{32} = \sqrt{1 - 1/\sigma_t^2}$ 
 is a stationary point of  \eqref{equ:prob:orth+:new:exact:penalty} with $\sigma
 = \sigma_t$ and it also satisfies \eqref{equ:palg:decrease}. Unfortunately, the
 limit point $X^{\infty}$ of $\{X^t\}$ is not a stationary point of problem
 \eqref{equ:prob:orth+:new} under the above settings. This counter-example tells
 that  $X^{\infty}$ is a stationary pointx only if some additional conditions
 are satisfied. 
 On the other hand, we observe numerically that the point $\bar X$
 found by our algorithm satisfies $\|\bar X\|_0=n$, namely, $\Omega''_0(\bar X)=\emptyset$. This implies that all the weakly stationary points encountered in numerical experiments are themselves stationary points.
\end{example}

\begin{theorem}\label{thm:convergence2}
    Under the same conditions as Theorem \ref{thm:convergence}, 
if additionally there holds that 
\be\label{equ:thm:Xcond2}
\lim_{t\to\infty} \sigma_{t}(\zeta_q(X^{t}) + \epsilon_{t})^{p-1}X^{t}_{ij}=0\quad \forall (i,j)\in\Omega''_0(X^\infty), 
\ee
then $X^{\infty}$ is  
a stationary point of problem \eqref{equ:prob:orth+:new}. 
\end{theorem}

\begin{proof}
By Theorem \ref{thm:convergence}, it remains to verify the correctness of the
equation $[\rgrad f(X^\infty)]_{ij}\ge 0\ \forall (i,j) \in \Omega''(X^\infty)$.
For any $(i,j)\in\Omega''_0(X^\infty)$,  it is easy to see that
$(i,l)\in\Omega''_0(X^\infty)$ for each $l \in [k]$. Together with  \eqref{equ:thm:rgrad:zeta:epsilon} and \eqref{equ:thm:bound1}, we have
\be
\sigma_t[\rgrad (\zeta_q(X^t) + \epsilon_t)^p]_{ij}
\le  pq(k-1)\bar\omega\|X^tV\|_{\Ftt}^{q-2} \left(\!\sigma_t (\zeta_q(X^t)+\epsilon_t)^{p-1}\max_{(i,j)\in\Omega''_0(X^\infty)}X^t_{ij}\right)\!,
\nn
\ee
which with  \eqref{equ:thm:Xcond2} and $\lim_{t\to\infty} \|X^tV\|_{\Ftt}=1$ gives $
\lim_{t\to\infty}\sigma_t[\rgrad (\zeta_q(X^t) + \epsilon_t)^p]_{ij}\le0\ \forall (i,j)\in\Omega''_0(X^\infty).$
By \eqref{equ:palg:kkt:condition}, we have $[\rgrad P_{\sigma_t, p,
q,\epsilon_t}(X^t)]_{ij}\ge - \varepsilon^{\grad}_t,\ \forall (i,j) \in [n]
\times [k]$.  Consequently, it follows from \eqref{equ:rgrad:P} that  for any $(i,j)\in\Omega''_0(X^\infty)$ there holds  
$
[\rgrad f(X^\infty)]_{ij} = \lim_{t\to\infty}[\rgrad f(X^t)]_{ij} \geq \lim_{t\to\infty}[\rgrad P_{\sigma_t, p, q,\epsilon_t}(X^t)]_{ij}\ge0. \nn 
$
\end{proof}

Futhermore, \cref{thm:convergence} gives the following corollary.
\begin{corollary}
Consider the same conditions as in Theorem \ref{thm:convergence}. If additionally $\|X^{\infty}\|_0 = n$, then $X^\infty$ is a stationary point of problem \eqref{equ:prob:orth+:new}. 
\end{corollary}
\begin{remark}
    Here, we  present  a different understanding  of the above corollary. Lemma
    \ref{lemma:CQ} tells that CCP holds  at $X^{\infty}$ if $\|X^{\infty}\|_0 =
    n$. By the approximate optimality conditions \eqref{equ:palg:kkt:condition}
    and the choice of the penalty term,   $X^\infty$ is  an approximate KKT point (see \cite{andreani2016cone} for its definition). Recalling that CCP is a strict CQ,  a CQ which  guarantees an  approximate KKT point as a KKT point (see \cite{andreani2016cone} for details), we thus know that $X^{\infty}$ mush be a stationary point.
\end{remark}

Theorem \ref{thm:convergence2} implies that the limit point $X^\infty$ is stationary as long as $X^t_{ij}$ decays to $0$ sufficiently fast on $(i,j)\in\Omega''_0(X^\infty)$, while the following theorem improves the convergence result by requiring a better solution of the subproblem.

\begin{theorem}\label{thm:convergence:second}
Let $X^t$ be the sequence generated by \cref{alg:feasiblePenalty} with $t_{\max}=\infty$, $\tol^{\feas} = 0$, $p\le 1$ and $\varepsilon_{0}^{\grad}=\varepsilon_{\min}^{\grad}=0$. If $X^t$ satisfies the WSOC conditions (refer to Definition \ref{def:prob:orth:WSOC}) of the  subproblem  \eqref{equ:prob:orth+:new:exact:penalty}, then the algorithm stops at some $\tilde t$ iteration and $X^{\tilde t}$ is a stationary point of problem \eqref{equ:prob:orth+:new}.
\end{theorem}

\begin{proof}
We prove it by contradiction. Suppose that  $X^t\notin \stiefplus$ for every $t$.
Without loss of generality, we assume $X^t \rightarrow X^{\infty}$. Since $p\le
1$, the sequence $\{\sigma_t c_t\}$ tends to infinity. Following Case I in the
proof of Theorem \ref{thm:convergence}, we have $X^t_{ij}=0, \,\forall (i,j)\in\Omega'_0(X^\infty)$ for large enough $t$ since $X^t$ satisfies conditions \eqref{equ:palg:kkt:condition} with $\varepsilon_t^\grad=0$.
 
 Since $X^t\notin \stiefplus$ for any $t$, there must exist $i_1\in[n]$,
 $j_1,j_2\in[k]$ and a subsequence $\{t_l\}$ such that $j_1\neq j_2$,
 $\|X^\infty_{i_1,:}\|=0$, $X^{t_l}_{i_1,j_1},X^{t_l}_{i_1,j_2}\neq 0$ and
 $\supp(X^\infty)\subset \supp(X^{t_l})$ for all $l$.   Here, $k_1,k_2$ are
 chosen such that $(k_1,j_1),(k_2,j_2)\in\supp(X^\infty)$. Since $X^{t_l}$
 satisfies the WSOC conditions of    \eqref{equ:prob:orth+:new:exact:penalty},
 we have $\langle D, \rhess P_{\sigma_{t_l}, p,
 q,\epsilon_{t_l}}(X^{t_l})[D]\rangle  \geq 0, \ \forall D \in \tilde \Ccal_{\obliqueplus} (X^{t_l})$.  Choosing 
$D^{t_l}$ with $D^{t_l}_{i_1j_1}=1,D^{t_l}_{i_1j_2}=-1, D^{t_l}_{k_1j_1}=-X^{t_l}_{i_1j_1}/X^{t_l}_{k_1j_1}$ and $D^{t_l}_{k_2j_2}=X^{t_l}_{i_1j_2}/X^{t_l}_{k_2j_2}$ and the remaining  elements of $D^{t_l}$ being zeros. One can check via direct calculation that $D^{t_l}\in \tilde \Ccal_{\obliqueplus} ( X^{t_l})$ and 
\begin{subequations}
\begin{align}
&\iprod{D^{t_l}}{D^{t_l}\off (VV^\tran)}\le -\underline \omega,\label{eq:thm:Dprop2}\\ 
&\big|\iprod{ D^{t_l}}{X^{t_l}VV^\tran}\big|\le \bar\omega \max_{(i,j)\in\Omega''_0(X^\infty)} X^{t_l}_{ij}.\label{eq:thm:Dprop1}
\end{align}
\end{subequations}
Substituting $D^{t_l}$ into the WSOC conditions yields
\be\label{eq:thm:secondcond}
\iprod{D^{t_l}}{\rhess f(X^{t_l})[D^{t_l}]}+\sigma_{t_l}\iprod{D^{t_l}}{\rhess (\zeta_q(X^{t_l})+\epsilon_{t_l})^p)[D^{t_l}]}\ge 0,
\ee
where \begin{equation}\label{eq:thm:hess} 
\begin{aligned}
& \iprod{D^{t_l}}{\rhess (\zeta_q(X^{t_l})+\epsilon)^p)[D^{t_l}]}\\
={}& c_{t_l}\Big(a |\langle D^{t_l},X^{t_l} VV^\tran\rangle|^2+  
\big\langle D^{t_l},D^{t_l}\big(\off(VV^{\tran}) -  \Diag\big(((X^{t_l})^{\tran}X^{t_l} - I_k)VV^{\tran}\big)\big)\big\rangle\Big)
\end{aligned}
\end{equation}
with $a =  \frac{q-2}{\|X^{t_l}V\|_{\Ftt}^{2}}+\frac{(p-1)\|X^{t_l}\|_{\Ftt}^{q-2}}{\zeta_q(X^{t_l})+\epsilon_{t_l}}$. 
Since $p\le 1$, we can drop the term with respect to $p-1$ when deriving a upper bound for the right-hand side of \eqref{eq:thm:hess}. A closer check then reveals that the term $\langle D^{t_l},D^{t_l}\off(VV^{\tran})\rangle$ dominates the equation since others tend to $0$, according to $X^{t_l}\to X^\infty\in \stiefplus$,  \eqref{eq:thm:Dprop2} and  \eqref{eq:thm:Dprop1}. Hence, we obtain
\[
\lim_{l\to\infty} \sigma_{t_l}\iprod{D^{t_l}}{\rhess (\zeta_q(X^{t_l})+\epsilon_{t_l})^p)[D^{t_l}]} \le \sigma_{t_l}c_{t_l}\langle D^{t_l},D^{t_l}\off(VV^{\tran})\rangle =-\infty.
\]
Together with \eqref{eq:thm:secondcond} and the fact that $\iprod{D^{t_l}}{\rhess f(X^{t_l})[D^{t_l}]}$ is bounded, we reach a contradiction. Therefore, the algorithm stops at some $\tilde t$ iteration. Since \eqref{equ:palg:kkt:condition} holds at $X^{\tilde t}$ with $\varepsilon_t^{\rgrad} = 0$ we know that $X^{\tilde t}$ must be a stationary point of probelm \eqref{equ:prob:orth+:new}.
\end{proof}

\begin{remark}
Notice that $\SNCD_{\Xcal_V}(X)\subset \LNCD_{\obliqueplus}(X)$. When $p=1$ and $q=2$, we have  
$\rhess f(X)[D]=\rhess P_{\sigma,p,q,\epsilon}(X)[D]\ \forall D\in \SNCD_{\Xcal_V}(X).$
Thus,  when $p=1,\,q=2$, $X^{\tilde t}$ in Theorem \ref{thm:convergence:second} satisfies the second order necessary conditions \eqref{equ:opt:SNC} of problem \eqref{equ:prob:orth+:new} as long as WSOC conditions therein replaced by the second order necessary conditions.
\end{remark}

Moreover, our exact penalty approach can be applied to the general problem \eqref{equ:prob:orth+:general}. The  subproblem  \eqref{equ:prob:orth+:new:exact:penalty} becomes 
 \be \label{equ:prob:orth+:new:exact:penalty:general}
 \min_{X \in \obliqueplus, Y \in \Ycal}\, \left\{P_{\sigma,p,q,\epsilon}(X, Y):=  f(X, Y) + \sigma    \left(\|XV\|_{\Ftt}^q - 1 + \epsilon\right)^{p} \right\}.
\ee
The corresponding algorithm is almost the same as \cref{alg:feasiblePenalty} except that  approximate stationary point conditions \eqref{equ:palg:kkt:condition}  becomes  
$\big\|\min(X^{t}, \rgrad_X P_{\sigma_t, p, q,\epsilon_t}(X^{t}, Y^{t}))\big\|_{\Ftt}$ $\leq \varepsilon_{t}^{\grad}$ and $\dist\big(Y^t, \proj_{\Ycal}(Y^t - \nabla_Y P_{\sigma_t, p, q,\epsilon_t}(X^{t}, Y^{t})) \big)  \leq \varepsilon_{t}^{\grad}$
 and the descent condition \cref{equ:palg:decrease} is replaced by  $P_{\sigma_{t},p,q,\epsilon_t}(X^{t}, Y^{t}) \leq  P_{\sigma_{t},p,q,\epsilon_t}(X^{t, 0}, Y^{t, 0})$.
To obtain a point satisfying the above conditions,  we can employ
the proximal alternating linearized minimization (PALM) method in
\cite{bolte2014proximal}.
One can also use the proximal alternating minimization scheme \cite{attouch2010proximal}, wherein the $X$-subproblem can be approximately solved by the second-order method \cref{alg:obliqueplus:alg:2nd}. In this case, 
we can extend the convergence results  to this general model by following almost the same proof.
}

\section{Optimization over $\obliqueplus$}\label{section:opt:obliqueplus}The penalty subproblem \eqref{equ:prob:orth+:new:exact:penalty} with
    suitable choices of parameters  is a special instance of optimization over   $\obliqueplus$, namely,  
\be\label{equ:prob:obliqueplus}
\min_{X \in \obliqueplus}\,h(X), 
\ee
where $h \colon \Rbb^{n \times k} \rightarrow \Rbb$ is  continuously
differentiable. 
We next first investigate  the optimality conditions for  problem
\eqref{equ:prob:obliqueplus}  and then present a second-order method. 

\subsection{Optimality conditions}
Note that   LICQ holds  for  \eqref{equ:prob:obliqueplus}. Similar to the discussion in section \ref{subsection:1st2ndKKT}, we establish the optimality conditions for  \eqref{equ:prob:obliqueplus}. To save the space, we omit some details here. 

\begin{theorem}[First-order necessary conditions] \label{lemma:1st:opt:obplus}
Let $\bar X\in \obliqueplus$ be a local minimizer of problem \eqref{equ:prob:obliqueplus}, then $\bar X$ is a stationary point, namely, there holds that 
\be\label{equ:prob:orth+:KKT}
0 \leq \bar X \perp \rgrad h(\bar X) \geq \0,
\ee
which is further equivalent to 
$\min(\bar X, \rgrad h(\bar X)) = \0$ or $\proj_{\LC_{\obliqueplus}(\bar X)} (-\rgrad h(X)) = 0$, where $\LC_{\obliqueplus}(\bar X) = \{D \in \Rbb^{n \times k} : \bar \xbf_j^{\tran} \dbf_j = 0,\ D_{ij}\geq  0 \ \emph{if}\ \bar X_{ij} = 0\}$.
\end{theorem}

\begin{theorem}[Second-order  necessary and sufficient conditions] \label{lemma:2nd:opt:obplus}
If  $\bar X$ is a local minimizer of problem \eqref{equ:prob:obliqueplus} then 
$\iprod{D}{\rhess h(\bar X)[D]} \geq 0, \ \forall D \in \LNCD_{\obliqueplus} (\bar X)$,
where $\rhess h(\bar X)[D]$ is obtained by specializing \eqref{equ:hess} to $h(\bar X)$ and the critical cone 
\[
\LNCD_{\obliqueplus} (\bar X) =  \LC_{\obliqueplus}(\bar X) \cap  \left\{ D \in \Rbb^{n \times k}:  
D_{ij} = 0 \ \emph{if} \ [\nabla f(\bar X)]_{ij} > 0\ \mathrm{and}\ \bar X_{ij} = 0
 \right\}.
\]
If $\bar X$ is a stationary point of  \eqref{equ:prob:obliqueplus} and 
$\iprod{D}{\rhess h(\bar X)[D]}$ $> 0, \ \forall  D \in \LNCD_{\obliqueplus} (\bar X)/\{\0\}$, then $\bar X$ is a strict local minimizer of \eqref{equ:prob:orth+:new}. 
\end{theorem}
\begin{definition} \label{def:prob:orth:WSOC}
We say $X \in \obliqueplus$ satisfies the weak second-order optimality conditions (WSOC) if $\min\left(\bar X, \rgrad h(\bar X)\right) = \0$ and 
$\iprod{D}{\rhess h(\bar X)[D]} \geq 0$ for any $D \in \tilde \Ccal_{\obliqueplus}(\bar X) := \{D \in \Rbb^n: \bar \xbf_j^{\tran} \dbf_j = 0, j \in [k], D_{ij} = 0\ \mbox{if}\ \bar X_{ij} = 0\}$. 
\end{definition}




\subsection{Methods for optimization over $\obliqueplus$}
\label{section:retr:based}
We first
    introduce a first-order  nonconvex gradient projection method:
\be\label{equ:gp:h}
X^{l+1} \in \proj_{\obliqueplus} (X^l - \alpha^l \nabla h(X^l)),\quad  \alpha^l > 0,
\ee
where the stepsize $\alpha^l$ can be determined by either monotone or
non-monotone linesearch, see  \cite{zhang2004nonmonotone} and the references
therein.  Note that $\proj_{\obliqueplus}(\cdot)$ is explicitly available and can be computed in $O(nk)$ flops, one can refer \cite{bauschke2017projecting, xu2017globally, zhang2018sparse} for instance. Since $\obliqueplus$ is compact, by Theorem 3.1 in
\cite{yang2017proximal}, any limit point of the  sequence
$\{X^l\}$ generated by  \eqref{equ:gp:h} is a stationary point of
\eqref{equ:prob:obliqueplus}.  If further $h$ is a KL function, by Theorem 5.3
in \cite{attouch2013convergence}, the sequence $\{X^l\}$ converges to a
stationary point of \eqref{equ:prob:obliqueplus} under some conditions.

We next adopt the adaptive quadratically regularized Newton method
\cite{hu2019structured, hu2018adaptive} to solve \eqref{equ:prob:obliqueplus}
in order to accelerate the convergence. At the $l$-th iteration, we
construct a quadratically regularized subproblem as
\be \label{prob:quadreg:subprob}
\min_{X \in \obliqueplus} m_l(X),
\ee
where $m_l(X) \coloneqq \langle \nabla h(X^l), X - X^l \rangle + \frac12
\langle  X - X^l, \nabla^2 h(X^l)[X - X^l] \rangle + \frac{\tau_l}{2}
\|X - X^l\|_{\Ftt}^2$.   It holds $\rgrad m_l(X^l) = \rgrad h(X^l)$ and $\rhess m_l(X^l)
[D] = \rhess h(X^l) [D] + \tau_l D$.  Since $\obliqueplus$ is compact, there exist positive constants
$\kappa_g$ and $\kappa_H$ such that $\|\nabla h(X)\|_{\Ftt} \leq \kappa_g$,
$\|\nabla ^2 h(X)\|\leq \kappa_H, \ \forall X \in \obliqueplus$. They imply that
$\|\rhess m_l(X^l) \| \leq \kappa_H + \kappa_g + \tau_l$.
We compute an approximate solution $Y^l \in \obliqueplus$ of the subproblem \eqref{prob:quadreg:subprob} satisfying 
\be \label{eqi:trial:point:condition}
m_l(Y^l) \leq -\frac{a}{\kappa_H + \kappa_g + \tau_l} \|\proj_{\Tbf(X^l)} (-\rgrad h(X^l))\|_{\Ftt}^2,
\ee
where $\Tbf(X) = \{D \in \Rbb^{n \times k} :  \xbf_j^{\tran} \dbf_j = 0,\ \xbf_j+\dbf_j\ge 0\}$ and $a$ is a positive constant.  Then, we calculate the ratio $\rho_l$
between the predicted reduction and the actual reduction to determine whether the trial point $Y^l$ is accepted or not. The complete algorithm is presented in \cref{alg:obliqueplus:alg:2nd}.
\begin{algorithm}[!htbp] 
 \SetKwFor{mypost}{Phase II: postprocessing}{}{endw}
 \SetKwFor{mywhile}{Phase I: while}{do}{endw}
\caption{An adaptive quadratically regularized Newton method for \eqref{equ:prob:obliqueplus}} \label{alg:obliqueplus:alg:2nd}
\textsf{Initialization:} Choose $X^0 \in \obliqueplus$, a tolerance $\epsilon > 0$ and an initial regularization parameter $\tau^0 > 0$. Choose $0< \eta_1 \leq \eta_2< 1$, $0 < \beta_0 < 1 < \beta_1 < \beta_2$. Set $l:=0$. 
\\
\While{$\|\bar X - \proj_+\big(\bar X - \rgrad h(\bar X)\big)\|_{\Ftt}> \epsilon$}{
Solve the subproblem \eqref{prob:quadreg:subprob} to obtain a trial point $Y^l$ satisfying \eqref{eqi:trial:point:condition}.\\
Calculate $\rho_l = \big(h(Y^l) - h(X^l)\big)/m_l(Y^l)$.\\
Set $X^{l + 1} \coloneqq Y^l$ if $\rho_l \geq \eta_1$ and set $X^{l + 1} \coloneqq X^l$ otherwise. \\
Choose  $\tau_{l+1} \in  (0, \beta_0\tau_l]$ if $\rho_l \geq \eta_2$; choose $\tau_{l + 1} \in [\beta_0 \tau_l, \beta_1 \tau_l]$ if $\eta_1 \leq \rho_l \leq \eta_2$; 
choose $\tau_{l+1} \in [\beta_1 \tau_l, \beta_2 \tau_l]$ otherwise.\\
Set $l := l + 1$.}
\end{algorithm}
By following the proof of \cite[Theorem 4]{hu2019structured}, we can establish $\|\proj_{\Tbf(X^l)} \left(-\rgrad h(X^l)\right)\|_{\Ftt} =0$  for some $l>0$ or 
\[\lim_{l \rightarrow \infty} \|\proj_{\Tbf(X^l)} \left(-\rgrad h(X^l)\right)\|_{\Ftt} = \lim_{l \rightarrow \infty} \|  X^l - \proj_{\Rbb_{+}^{n\times k}}\big(X^l - \rgrad h(X^l)\big)\| =  0.
\] 

 We now show that  the inexact condition \eqref{eqi:trial:point:condition}  is well defined. 
 Let  $c_1$ be a positive constant. 
For a direction  $D^l \in \Tbf(X^l)$ satisfying 
\be \label{equ:D:condition}
\iprod{\rgrad h(X^l)}{D^l} \leq -c_1 \|\proj_{\Tbf(X^l)} (-\rgrad h(X^l))\|_{\Ftt} \|D^l\|_{\Ftt},
\ee
 we compute 
\be\label{equ:Yl:new:proj}
 \widetilde{Y}^l = \proj_{\obliqueplus}\left(X^l +  \alpha_l D^l\right)\quad \mbox{with}\quad  \alpha_l  = \frac{2c_1(1 - c_2)\|\proj_{\Tbf(X^l)} (-\rgrad h(X^l))\|_{\Ftt}}{(\kappa_g + \kappa_H + \tau_l) \|D\|_{\Ftt}}, 
\ee
where $c_2 \in (0,1)$ is a constant. 
 For any $D \in \Tbf(X)$, it is easy to verify that $\iprod{\nabla h(X)}{D} = \iprod{\rgrad h(X)}{D}$ and  
\be  \label{equ:retraction:property}
\|\proj_{\obliqueplus}(X + D)  - X \|_{\Ftt} \leq \|D\|_{\Ftt},  \quad  \|\proj_{\obliqueplus}(X + D)  - X - D \|_{\Ftt} \leq \frac12\|D\|_{\Ftt}^2.
 \ee
 It follows from the property \eqref{equ:retraction:property} and the arguments in
 \cite[Lemma 2.10]{boumal2019global} that any $Y^l$ with $m_l(Y^l) \leq m_l(\widetilde{Y}^l)$  satisfies \eqref{eqi:trial:point:condition} with $a =
 2c_1^2c_2(1 - c_2)$. For sake of saving space, we
 omit the tedious details here.  

We  give two particular choices of $D^l$ satisfying \eqref{equ:D:condition}.
The first one is a single projected gradient  step  $D^l = \proj_{\Tbf(X^l)} (-\rgrad h(X^l))$. Then
\eqref{equ:D:condition} holds with $c_1 = 1$. The second choice  $D^l$ is
from the Newton subproblem in \cite{hu2018adaptive}: 
\be\label{equ:subproblem:oblique+:qp}
\min\limits_{D\in \Tbf(X^l)}   \iprod{\rgrad m_l(X^l)}{D} + \frac12 \iprod{D}{\rhess m_l(X^l)[D]}.
\ee
Setting $D = Z - X^l$ reformulates problem \eqref{equ:subproblem:oblique+:qp} as 
\be\label{equ:subproblem:oblique+:qp:2}
\min_{Z \in \Delta(X^l)} \iprod{\rgrad m_l(X^l)}{Z- X^l} + \frac12 \iprod{Z- X^l}{\rhess m_l(X^l)[Z - X^l]},
\ee
 where $\Delta(X^l) := \{Z \in \Rbb^{n \times k}: (\xbf^l)_j^{\tran} \zbf_j = 1, \zbf_j\geq  0, j \in [k]\}$. 
 The first-order optimality condition is the following nonsmooth equation: 
\be \label{equ:nonsmoothFZ}
\Fcal(Z) :=  Z - \proj_{\Delta(X^l)} \big(Z - \alpha \big(\rgrad m_l(X^l) + \rhess m_l(X^l) [Z - X^l]\big)\big) = 0,
\ee
where $\alpha > 0$ is a constant. Denote $C :=  Z - \alpha \big(\rgrad m_l(X^l) + \rhess m_l(X^l) [Z - X^l] \big)$ for simplicity.
Thanks to   \cite{li2018on}, we can efficiently compute the the HS generalized Jacobian $\Pcal_{C}(\cdot)$ of $\proj_{\Delta(X)}(\cdot)$  efficiently. 
 Define a linear operator $\Xi: \Rbb^{n \times k} \rightarrow \Rbb^{n \times k}$ by $[\Xi(H)]_{ij}  = 0$ if $[\proj_{\Delta(X^l)}(C)]_{ij} = 0$ and $[\Xi(H)]_{ij}  = H_{ij}$ otherwise for any $H\in \Rbb^{n \times k}$. 
We simply denote $\Xi[\hbf_j] = (\Xi(H))_{j,:},\ \forall j \in [k]$. Following
Proposition 3 in \cite{li2018on} yields the HS-Jacobian of $\Pi_{\Delta(X)}(\cdot)$ at $C$ as $
\Pcal_{C}(H) = \Xi(H) - \Xi(X) M,$
where $M$ is diagonal with 
$M_{jj} = {\xbf_j^{\tran} \Xi[\hbf}_j]/{\xbf_j^{\tran} \Xi[\xbf_j]}$, $\forall j \in [k].$
Hence, we have the HS-Jacobian of $\Fcal$  at $Z$ as 
$\partial \Fcal(H) = H - \Pcal_{C}\left(H - \alpha\, \rhess m_l(X^l)[H]\right), \forall H \in \Rbb^{n \times k}$.
 We then apply the semi-smooth Newton method in
    \cite{xiao2018regularized,milzarek2019stochastic}    
    to generate a sequence $\{Z^{j}\}$ to solve \eqref{equ:subproblem:oblique+:qp}. 
    It  satisfies  $\lim_{j \rightarrow \infty}
\|\Fcal(Z^j)\| = 0$ by virtue of Theorem 3.10 in \cite{milzarek2019stochastic}
under some reasonable assumptions.   If $\tau_l > \kappa_H$ and $\alpha \in (0,
\frac{2}{\kappa_H + \kappa_g + \tau_l})$, we can show that the limit point of
$\{D^j := Z^j - X^l\}$ satisfies \eqref{equ:D:condition} with $c_1 =
\frac{\tau_l - \kappa_H}{\tau_l + \kappa_H + 1}$. However, since $\rhess m_l(X^l)$ may not be positive
definite in other cases, it is still not clear whether \eqref{equ:D:condition} holds or not although the numerical performance is well.

\section{Numerical experiments} \label{section:num}
In this section, we present a variety of numerical results to evaluate the performance of our proposed method. All  experiment are performed in  Windows 10 on an Intel Core 4 Quad CPU at 2.30 GHZ with 8 GB of RAM.  All codes are written in MATLAB R2018b.   The matrix $V$ is simply taken as $V = \1/\sqrt{k}$, and the choice of parameters for \cref{alg:feasiblePenalty} are set as follows: $p=1$, $q=2$, $\epsilon_0 = 0$, $\gamma_1 =0$,  $\varepsilon_{\min}^{\grad} = 10^{-7}$, $t_{\max} = 300$;  the choices of $\gamma_2$, $\sigma_0$, $\eta$, $\tol^{\mathrm{feas}}$ and $X^0$ are given in each subsection. In our implementation, instead of using \eqref{equ:palg:kkt:condition}, we use the stopping condition when the distance between  two consecutive iterations is small, namely, $\|X^{l+1} - X^l\|_{\Ftt} \leq \varepsilon_{l}^{\grad}$.

\subsection{Computing projection onto $\stiefplus$}\label{subsection:num:projection}
Given a matrix $C \in \Rbb^{n \times k}$, we consider to compute its projection  onto $\stiefplus$,  which is formulated as 
\be\label{equ:prob:stiefelplus:proj}
 \min_{X \in \stiefplus} \, \|X - C\|_{\Ftt}^2.
\ee
The exact penalty model \eqref{equ:prob:orth+:new:exact:penalty} with $p = 1$, $q = 2$, and $\epsilon = 0$ is  equivalent to 
\be\label{equ:penalty:model:projection}
\min_{X \in \stiefplusk}\, \left\{P_{\sigma}(X): =  - \frac{1}{\sigma}\langle C,   X\rangle  + \frac{1}{2} \|XV\|_{\Ftt}^2\right\}.
\ee 
 The Lipschitz constants of $\nabla P_{\sigma}(X)$ is  $1$ since $VV^{\tran} \preceq I_k$. Thus we know from Theorem 5.3 in \cite{attouch2013convergence} that the sequence $\{X^l\}$ generated by  the nonconvex gradient projection scheme 
$X^{l + 1}\ \in\ \proj_{\obliqueplus}\left(X^l - \alpha \left(X^l VV^{\tran} -  C/\sigma \right)\right)$ with $\alpha \in (0,1) \label{equ:kind:PALM:projection}
$
converges to a stationary point of \eqref{equ:penalty:model:projection}.  
In our  tests,  we simply choose $\alpha = 0.99$ without invoking \cref{alg:obliqueplus:alg:2nd} to solve  \eqref{equ:penalty:model:projection}, namely, the switch parameter $\bar{\zeta} \equiv 0$.

It is always difficult to seek the projection globally for a general matrix $C$. Thanks to \cref{prop:proj}, we can  construct  a family of matrices with unique and known projection. For a given $B \in \stiefplus$,  the MATLAB codes for generating $C$ is given as 

\begin{verse}
\texttt{X = (B>0).*(1+rand(n,k)); Xstar = X./sqrt(sum(X.*X));\\
d = 0.5+3*rand(k,1); L = xi*((d*d').\^\,0.5).*rand(k,k);\\
L(sub2ind([k,k],1:k,1:k))=d; C=Xtar*L;\\
}
\end{verse}
The parameter $\xi \in [0,1]$ controls the magnitude of noise level. Larger $\xi$ makes it more difficult to find the ground truth $X^* = \proj_{\stiefplus}(C)$. Let $X^{\diamondsuit}$ be the solution generated by \cref{alg:feasiblePenalty}. Note that the postprocessing problem \eqref{equ:prob:refinement} has closed form solution. Define  $\mathrm{gap} =  {\|X^{\diamondsuit}-C\|_{\Ftt}}/{\|X^*-C\|_{\Ftt}} - 1$ as a measure of the solution quality.  
For each  $\xi$, $n$ and $k$, we run 50 times of our algorithms, and the initial
point is generated by rounding $C$ through Procedure \ref{alg:round}. We choose
$\gamma_2 = 5$, $\tol^{\mathrm{feas}}=10^{-8}$, $\sigma_0=10^{-2}$, $\eta=0.8$.
The averaged results are reported in \cref{table:proj}, wherein the ``suc''
means the total number of instances for which the gap is zero.  From this table,
we can see that for small $\xi$,  \cref{alg:feasiblePenalty}  can solve all 50 instances to a zero gap, while for large $\xi$ it can only solve some instances to a zero gap. However, for all cases,  \cref{alg:feasiblePenalty}  always returns an orthogonal nonnegative matrix with satisfactory  quality. 

\begin{table}[!htbp]
\centering
\caption{Numerical results on computing projection onto $\stiefplus$, ``time'' is in seconds.}
	\setlength{\tabcolsep}{3pt}
\begin{tabular}{|c||rrrr|rrrr|rrrr|}
\hline
    & \multicolumn{4}{c|}{$n=2000,k=10$} & \multicolumn{4}{c|}{$n=2000,k=50$} & \multicolumn{4}{c|}{$n=2000,k=100$} \\ \hline
 $\xi$ & suc     & gap & time  & nproj &  
suc     & gap & time  & nproj &
suc     & gap & time  & nproj  \\ \hline
0.50&   50 & 0.0e0 & 0.01 & 20.5 &   50 & 0.0e0 & 0.04 & 38.3 &   50 & 0.0e0 & 0.32 & 53.9 \\ 
0.70&   50 & 0.0e0 & 0.01 & 22.9 &   50 & 0.0e0 & 0.05 & 50.9 &   50 & 0.0e0 & 0.43 & 76.5 \\ 
0.90&   50 & 0.0e0 & 0.01 & 28.7 &   50 & 0.0e0 & 0.07 & 82.1 &   50 & 0.0e0 & 0.66 & 134.6 \\ 
0.95&   49 & 7.2e-5 & 0.01 & 31.9 &  46 & 2.1e-4 & 0.09 & 112.2 &   49 & 6.6e-7 & 0.87 & 184.8 \\ 
0.98&   43 & 8.9e-4 & 0.01 & 33.8 &  22 & 5.0e-4 & 0.11 & 156.3 &   19 & 8.0e-4 & 1.23 & 268.2 \\ 
1.00&   37 & 1.2e-3 & 0.01 & 38.1 &   0 & 2.6e-3 & 0.12 & 170.3 &    0 & 2.6e-3 & 1.43 & 317.5 \\ 
\hline
\hline 
   & \multicolumn{4}{c|}{$n=2000,k=200$} & \multicolumn{4}{c|}{$n=2000,k=300$} & \multicolumn{4}{c|}{$n=2000,k=400$} \\ \hline
 $\xi$ & suc     & gap & time  & nproj &  
 suc     & gap & time  & nproj &
suc     & gap & time  & nproj  \\ \hline
0.50&   50 & 0.0e0 & 0.77 & 73.7  &   50 & 0.0e0 & 1.34 & 89.9 &   50 & 0.0e0 & 1.96 & 99.8 \\ 
0.70&   50 & 0.0e0 & 1.13 & 113.5 &   50 & 0.0e0 & 2.01 & 137.9 &   50 & 0.0e0 & 2.99 & 157.8 \\ 
0.90&   50 & 0.0e0 & 1.6 & 207.6  &   50 & 0.0e0 & 3.43 & 276.0 &   50 & 0.0e0 & 5.39 & 328.7 \\ 
0.95&   50 & 0.0e0 & 2.42 & 295.1 &   50 & 0.0e0 & 5.13 & 424.9 &   50 & 0.0e0 &7.74 & 483.0 \\ 
0.98&   23 & 4.5e-4 & 3.93 & 489.2 &  20 & 2.5e-4 & 8.60 & 718.6 &   24 & 1.7e-4 & 15.42 & 962.2 \\ 
1.00&    0 & 1.9e-3 & 5.07 & 636.9 &   0 & 1.8e-3 &11.31 & 951.3 &    0 & 1.6e-3 & 20.86 & 1324.0 \\ 
\hline
\end{tabular}
\label{table:proj}
\end{table}

\subsection{Orthogonal nonnegative matrix factorization} \label{section:onmf}

We compare our proposed method with uni-orthogonal NMF (U-onmf) \cite{ding2006orthogonal}, orthonormal projective
nonnegative matrix factorization (OPNMF) \cite{yang2010linear}, orthogonal nonnegatively penalized matrix factorization (ONP-MF) \cite{pompili2014two} and EM-like algorithm for ONMF (EM-onmf) \cite{pompili2014two}.   In addition to the above methods, we also compare our method with K-means, which is 
considered as a benchmark in clustering problems. We implement U-onmf by ourselves since the original code is not available. We adopt the implementation of OPNMF from \url{https://github.com/asotiras/brainparts}. 
The codes of ONP-MF and EM-onmf can be downloaded from \url{https://github.com/filippo-p/onmf}. As to K-means, we call the MATLAB function \texttt{kmeans} directly. Note that our proposed method and OPNMF solve the equivalent formulation \eqref{equ:prob:onmf:XXt} while the remaining methods solve directly \eqref{equ:prob:onmf}. Considering that the objective function in \eqref{equ:prob:onmf:XXt}  is quartic, to make the subproblem \eqref{equ:prob:orth+:new:exact:penalty} easier to solve, one can  consider the Gauss-Newton technique as 
\[
\|A - XX^{\tran}A \|_{\Ftt}^2 = \|A - X\tilde X^{\tran}A - \tilde XS^{\tran} A - SS^{\tran} A\|_{\Ftt}^2 \approx \|A - X\tilde X^{\tran}A - \tilde XS^{\tran} A \|_{\Ftt}^2,
\]
where $S = X - \tilde X$.   By neglecting the term $\tilde XS^{\tran} A$, we obtain 
a partial Gauss-Newton approximation, namely, $\|A - XX^{\tran} A \|_{\Ftt}^2 \approx \|A -   X \tilde X^{\tran}A\|_{\Ftt}^2$.  Moreover, if $X \in \stiefplus$,  we know that $\|A - XX^{\tran} A \|_{\Ftt}^2 = \|A - X(X^{\tran}X)^{-1}X^{\tran} A \|_{\Ftt}^2$.  Hence, to make the approximation robust, we consider $\|A - XX^{\tran} A \|_{\Ftt}^2 \approx \|A -   X (\tilde X^{\tran} \tilde X)^{-1}\tilde X^{\tran}A\|_{\Ftt}^2$. The subproblem \eqref{equ:prob:orth+:new:exact:penalty} at $t$-th iteration with $p = 1$, $q = 2$, and $\epsilon = 0$ becomes 
 \be \label{equ:prob:orth+:new:exact:penalty:onmf:pGN}
 \min_{X \in \obliqueplus}\,   \|A - X (Y^{t})^{\tran}\|_{\Ftt}^2 + \sigma_t \|XV\|_{\Ftt}^2\nn 
\ee
with $Y^{t} = \proj_+\big(A^{\tran} \tilde X^{t}((\tilde X^{t})^{\tran} \tilde X^{t})^{-1}\big).$

In some datasets, the matrix  $A$ maybe degenerated, namely,  there exists a row
(column) of $A$ with all zero entries. This  causes a division by zero error
when running the U-onmf method. Thus we will first remove such degenerate rows
and columns of $A$. For K-means and EN-onmf, the initial points are chosen
randomly. The other methods adopt the SVD-based initializations
\cite{Boutsidis2008SVD}.  In practice, the time cost of generating initial
points is relatively low compared to that of  the rest parts. We set
$\sigma_0=10^{-3}$, $\eta=0.98$,  and  choose $\tol^{\feas}=0.3$, $\bar{\zeta} =
0.6$ for hyperspectral datasets and $\tol^{\feas}=10^{-8}$, $\bar{\zeta} = 5$
for other datasets. The main parameters of \cref{alg:obliqueplus:alg:2nd} are
chosen as $\eta_1 = 0.01$, $\eta_2 = 0.9$, $\beta_0 = 0.98$, $\beta_1 = 1$, and
$\beta_2 = 1.3$. In sections \ref{subsection:onmf:1} and
    \ref{subsection:onmf:2}, we set $\gamma_2 = 1.05$ if $\|X^tV\|_{\Ftt}^2>2$
    and $\gamma_2 = 1.03$, otherwise.  In section \ref{subsection:onmf:3}, we set
    $\gamma_2= 1.1 \times 1.05$  if $\|X^tV\|_{\Ftt}^2>2$ and $\gamma_2 = 1.1
    \times 1.03$, otherwise.
 We adopt the Barzilai-Borwein stepsize \cite{barzilai1988two} and use the
nonmonotone line search \cite{zhang2004nonmonotone} in the gradient projection
iteration \cref{equ:gp:h}.  Define $S^{l-1} = X^l - X^{l-1}$ and $Z^{l-1} =
\nabla h(X^l) - \nabla h(X^{l-1})$.  We compute $\alpha^l = \max\{10^{-10},
\min\{\alpha_{\LBB}^l, 10^{10}\}\}$ with $\alpha_{\LBB}^{l} = \frac{\langle
S^{l-1}, S^{l-1} \rangle}{|\langle S^{l-1}, Z^{l-1}\rangle|}$.  For
\cref{alg:obliqueplus:alg:2nd}, we use the ASSN in \cite{xiao2018regularized} to
approximately solve \eqref{equ:subproblem:oblique+:qp}. Our numerical results
show that the point returned by
ASSN almost always satisfies \eqref{equ:D:condition}. Otherwise, $\tau_l$ is
increased until \eqref{equ:D:condition} is satisfied. Since we aim to show in sections \ref{subsection:onmf:1} and \ref{subsection:onmf:2} that our algorithm can generate a solution with high quality and  small feasibility violation, we remove the postprocessing in \cref{alg:feasiblePenalty} to give a fair comparison therein.

\subsubsection{Synthetic data} \label{subsection:onmf:1}
Our main aim in this part is to compare the performance of solving the ONMF problem itself,  so EN-onmf and K-means will be excluded in the comparison since they can only provide the results of clustering other than a meaningful orthogonal nonnegative matrix factorization.

Given a random generated matrix $B \in \stiefplus$, a positive integer $r$ and  a real number  $\xi$ which controls the magnitude of noise, we construct the matrix $A$ by  the following MATLAB codes: 
\begin{verse}
\texttt{C = rand(k,r); D = rand(n,r); A = B*C; \\ 
A = A/norm(A,'fro'); A = A + xi/norm(D,'fro')*D; 
}
\end{verse}
Let $\hat X$ be the solution generated by algorithms,  we calculate the feasibility violation as $\mathrm{feasi} \coloneqq||\hat X^{\tran}\hat X-I_k||_{\Ftt}+||\min(\hat X,0)||_{\Ftt}$.   Performing rounding Procedure \ref{alg:round} on $\hat X$ to obtain a feasible   $\hat X^{\R}$, we  take $\mathrm{resi} \coloneqq||A-\hat X^{\R} (\hat X^{\R})^{\tran}A||_{\Ftt}$  to measure the quality of the solution. The results are presented in \cref{tab:table3}, where $n=1000,r=3000,k=10$.  From this table, we can see  that the orthogonality and nonnegativity of the solutions given by our method are well kept, while the solutions generated by U-onmf, ONP-MF and OPNMF have relatively large violation.  Besides,  the solution quality of our proposed method is also  better than that of other methods. In summary,  our proposed method outperforms the other methods for the synthetic datasets.
\begin{table}[!htbp]
\centering
\caption{ ONMF results on synthetic data with different noise magnitude. 
}
	\setlength{\tabcolsep}{2pt}
\label{tab:table3}
\begin{tabular}{{|c||ccc|ccc|ccc|}}
\hline
 & \multicolumn{3}{c|}{$\xi=0$}  & \multicolumn{3}{c|}{$\xi=0.01$} & \multicolumn{3}{c|}{$\xi=0.1$} \\ 
 \cline{2-10}
method          & feasi & resi &time  & feasi & resi  & time & feasi & resi  & time \\ \hline
EP4Orth+      & 4.3e-16       & \textbf{3.8e-16}  &0.9      & 8.8e-16   & \textbf{5.4e-3} &3   & 9.2e-16 & \textbf{5.4e-2} &4   \\ 
OPNMF           & 2.6e-2        & 1.6e-15 &0.9       & 2.1e-2         & 5.8e-3 &2     & 5.4e-2        &5.9e-2   &3     \\ 
U-onmf          & 5.7e-2       &4.9e-16  &9        & 5.7e-2          & 5.6e-3  &9    & 7.7e-2        & 5.8e-2   &9  \\ 
ONP-MF          & 3.1e-3     &  5.0e-1    &18     &  3.2e-3       & 5.0e-1     &20     & 3.1e-3       &4.8e-1     &15  \\ \hline
 & \multicolumn{3}{c|}{$\xi=1$}   & \multicolumn{3}{c|}{$\xi=10$}   & \multicolumn{3}{c|}{$\xi=100$} \\ 
 \cline{2-10}
method          & feasi   & resi &time  & feasi & resi  & time & feasi & resi  & time\\ \hline
EP4Orth+     & 1.2e-15       & \textbf{5.1e-1}&5  &8.9e-16 & \textbf{5.0}&5  & 7.2e-16   & \textbf{49.7}  &5 \\ 
OPNMF         & 3.9e-1         & 5.9e-1  &14   & 7.4e-1        & 5.5  &67       & 7.1e-1        & 53.1  &85   \\ 
U-onmf          & 3.6e-1          & 5.8e-1 &23    & 1.2e0       & 5.6&49         & 1.2e0         & 54.4  &50  \\ 
ONP-MF         &  3.2e-3       & 7.1e-1   &19   & 3.2e-3       &5.2 &27        &  3.2e-3       & 50.3  &32       \\ 
\hline 
\end{tabular}
\end{table}
\subsubsection{Text and image clustering}\label{subsection:onmf:2} 
We evaluate algorithms on text and image datasets adopted from \cite{cai2008modeling}, they are available at \url{http://www.cad.zju.edu.cn/home/dengcai/Data/data.html}. Since the original text dataset is too huge and disproportionate, we extract some subsets from original data to make it suitable for testing clustering algorithms. The details of modification are provided as follows.

\begin{itemize}[leftmargin=18pt]
\item Reuters-t10(-t20): For the 10 (20) classes with the largest number of texts in the dataset Reusters, we collect 5 percent of texts from the 1st class with the most texts, 10 percent from the 2nd, and all the texts from 3rd-10th (3rd-20th) classes.
\item TDT2-l10(-l20): We use all texts in the 10 (20) classes with the smallest number of texts in the dataset TDT2.
\item TDT2-t10(-t20): We take 20 percent of texts of 10 (20) classes with the largest number of texts in the dataset TDT2.
\item NewsG-t5: We take 50 percent of texts of 5 classes with the largest number of texts in the dataset Newsgroup.
\end{itemize}

For text datasets, every article is assigned with a vector, which reflects the frequency of each word in the article. While for image datasets, a vector represents the gray level of each pixel in a picture. The data matrix $A$ is comprised of these vectors. Any solution $X^*\in\stiefplus$ of ONMF indicates a partition (clustering result) of the dataset. The scale of each dataset is given in \cref{tab:datasets}, in which  ``data'' denotes the number of rows of data matrix $A$ and ``features'' stands for the number of columns.
\begin{table}[!htbp]
\centering
\caption{
Description of each dataset. In the table, ``d'', ``f'', ``c'' mean ``data'',
``features'' and ``clusters'', respectively. Moreover, ``R'' and ``T'' in the column
``Name'' denotes ``Reuters'' and ``TDT2'' for short, respectively.  }\label{tab:datasets}
	\setlength{\tabcolsep}{3.0pt}
\begin{tabular}{|c|c|c|c||c|c|c|c||c|c|c|c| }
\hline
Name        & d  & f & c & Name        & d & f   & c  & Name        & d & f  & c    \\ \hline
R-t10 & 1897 & 12444  & 10 &  R-t20 & 2402 & 13568 & 20    & T-l10    &  653 & 13684    & 10     \\ 
T-t10    & 1477 & 22181    & 10     &T-l20    & 1938 & 20845  & 20      &T-t20    & 1721 & 23674   & 20       \\ 
NewsG-t5    & 2344 & 14475  & 5     & MNIST       & 4000 & 784    & 10    &Yale       &  165 & 1024    & 5        \\ \hline
\end{tabular}
\end{table}

We consider three criteria to compare the performance of clustering results: purity,  entropy and NMI. We denote $k$ as the number of clusters, and $n$ the total number of data points. Suppose that $\Ccal=\bigcup_{i=1}^k\Ccal_i$ and $\Ccal'=\bigcup_{j=1}^k\Ccal'_j$ are clustering results given by ground truth and  certain test algorithm. Let $n_i=|\Ccal_i|$, $n'_j=|\Ccal'_j|$ and $n_{ij}=|\Ccal_i\cap \Ccal'_j|$. 
The purity  \cite{ding2006orthogonal} is computed as $\text{purity}\coloneqq \sum_{i=1}^k{\max_j\{n_{ji}\}}/{n}$. 
Purity gives a measure of the predominance of the largest category per cluster, better clustering results leads to larger purity.  The entropy \cite{zhao2004empirical} and normalized mutual information (NMI) \cite{xu2003document} are computed as  $\text{entropy}\coloneqq-\frac{1}{n\log_2k}\sum_{j=1}^k\sum_{i=1}^kn_{ij}\log_2\frac{n_{ij}}{n_j'}$ and $\text{NMI}\coloneqq \frac{1}{\max(H(\Ccal),H(\Ccal'))}\sum_{i=1}^k\sum_{j =1}^k\frac{n_{ij}}{n}\log_2\frac{n n_{ij}}{n_in'_j}$,  
where $H(\Ccal)=-\sum_{i=1}^k\frac{n_i}{n}\log_2\frac{n_i}{n}$ and $H(\Ccal')$ was defined similarly. A better clustering result  has smaller entropy and larger NMI. 
 Note that we will not calculate ``feasi'' for K-means and EN-onmf, as they only generate the clustering results instead of solutions of ONMF problem.
For random algorithm, their results are averaged over 10 runs.

In \cref{tab:textclustering}, we report text and images clustering results. We can observe from this  table that our proposed method performs very well.  Specifically, the clustering results given by our proposed method has the highest purity and NMI in most of cases (being close for the rest dataset). As to the speed, our method is faster than U-onmf and ONP-MF for most of cases, and it is especially efficient on text dataset. Besides, the feasibility violation of the solution returned by our method is very small, while those returned by the other methods are always very large. On the other hand, K-means is the fastest among all algorithms and performs well on image datasets MNIST and Yale, but it results poorly when applying to text dataset; EM-onmf and OPNMF are efficient but their performance is  slightly worse than ours.

\begin{table}[!htbp]
\centering
\caption{
Text clustering results on real datasets. In the table,  ``c1'', ``c2'' and ``c3'' stand for ``purity'' (\%), ''NMI'' (\%) and ``entropy'' (\%), respectively; ``t'' means the time in seconds. Results marked in bold mean better performance in the corresponding index. }
	\setlength{\tabcolsep}{2pt}
\begin{tabular}{{|c||rrrrr|rrrrr|rrrr|}}
\hline
     & \multicolumn{5}{c|}{EP4Orth+} & \multicolumn{5}{c|}{U-onmf}   & \multicolumn{4}{c|}{K-means}  \\ 
     \cline{2-15}
datasets     & c1 & c2   & c3  & feasi& t  &c1 & c2 & c3 & feasi& t & c1 & c2  & c3 & t   \\ \hline 
Ret-t10  &\textbf{73.1}& \textbf{60.7}& \textbf{37.9}& 2e-15&9&72.7&59.2&39.4&0.6&55&36.9&22.2&75.1&4 \\ 
Ret-t20  &\textbf{65.5} & 56.3 & 38.4&2e-15&25&60.6& 52.7 &41.7 &0.9&149& 33.9& 17.4&79.8 &4 \\ 
TDT2-l10 &\textbf{84.5} & \textbf{79.9} & \textbf{20.1}& 9e-16& 4& 81.8 & 76.0 &24.0  &0.4&7 & 35.2 & 26.2 & 71.3&0.8 \\ 
TDT2-t10 &\textbf{85.7} & 70.0 &20.8 &2e-15&9& 80.9 & 65.7 & 22.8&0.5&115&41.1 & 17.8 &70.5 &4 \\ 
TDT2-l20 &83.1 & \textbf{84.2} & \textbf{15.5} &1e-15&17&81.9 & 82.0 & 17.7 &0.4&60&  23.8& 17.6 &80.7 &6 \\ 
TDT2-t20 &\textbf{82.3} & \textbf{69.6} & \textbf{18.1}&1e-15&18& 79.3 & 64.3 &21.2 &0.7&299& 39.1 & 18.6&65.8 &7 \\ 
NewsG-t5 &41.5 & \textbf{22.8} &\textbf{77.1} &2e-15&7& 39.3 & 14.9 &85.0 &0.2&18& 21.1 & 0.4&99.5 &2 \\ 
MNIST    & \textbf{60.1} &\textbf{48.9} &\textbf{51.0} &1e-15&26&50.0& 41.9 & 58.0  &1.0&39& 55.4 & 45.2& 54.7&0.9 \\ 
Yale     &\textbf{44.8} & \textbf{47.9} &\textbf{52.1}  &6e-16&2&  43.7 & 45.9 &54.0&1.2&2&40.8  & 44.1& 55.9&0.1 \\ 
\hline
\hline 

   & \multicolumn{5}{c|}{OPNMF} & \multicolumn{5}{c|}{ONP-MF}   & \multicolumn{4}{c|}{EM-onmf}  \\ 
     \cline{2-15}
datasets  &   c1 &   c2 &   c3 &feasi& t &   c1 &   c2 &  c3  &feasi& t &   c1 &  c2  &   c3 & t   \\ \hline 
Ret-t10   & 72.0 & 58.7 & 39.9 &1.1  &15 & 66.9 & 52.8 & 45.6 &3e-3 &82 & 71.3 & 58.6 & 39.9 &17 \\ 
Ret-t20   & 62.9 & 54.6 & 40.0 &1.8  &24 & 62.0 & 53.5 & 41.6 &4e-3 &386& 64.1 & \textbf{57.4} & \textbf{37.8} &30 \\ 
TDT2-l10  & 82.4 & 77.3 & 22.6 &1.1  &1  & 81.3 & 75.5 & 24.4 &3e-3 &77 & 78.0 & 78.5 & 21.4 & 4 \\ 
TDT2-t10  & 82.2 & 64.3 & 24.4 &0.9 &10 & 82.9 & 65.3 & 23.8 &3e-3 &133& 85.0 & \textbf{71.3} & \textbf{20.1} &21 \\ 
TDT2-l20  & \textbf{83.4} & 82.5 & 17.2 &1.5  &6  & 82.6 & 83.1 & 16.5 &4e-3 &450& 80.4 & 82.0 & 17.7 &27 \\ 
TDT2-t20  & 79.1 & 62.5 & 21.4 &1.1  &14 & 81.1 & 65.0 & 20.4 &4e-3 &542& 80.8 & 67.2 & 19.3 &25 \\ 
NewsG-t5  & 37.1 & 13.1 & 86.7 &0.4&11 & \textbf{42.9} & 22.6 & 77.2 &2e-3 &44 & 35.7 & 15.4 & 84.5 &14 \\ 
NMIST     & 55.1 & 44.1 & 55.9 &1.3  &218& 57.4 & 46.1 & 53.8 &5e-2 &61 & 56.3 & 47.8 & 52.2 & 4 \\ 
Yale      & 43.7 & 45.4 & 54.6 &1.4  &4  & 40.0 & 43.6 & 56.6 &1e-2 &10 & 38.1 & 41.7 & 58.3 &0.1\\  \hline\end{tabular}
\label{tab:textclustering}
\end{table}

\subsubsection{Hyperspectral unmixing} \label{subsection:onmf:3} 
A set of images taken on the same object at different wave lengths is called a hyperspectral image. At a given wavelength, images are generated by surveying reflectance on each single pixel. Hyperspectral unmixing plays an essential role in hyperspectral image analysis \cite{Bioucasdias2012Hyperspectral,Keshava2002Spectral}. It assumes that each pixel spectrum $\abf\in\Rbb_+^{r}$ is a composite of $k$ spectral bases $\{\ybf_i\}_{i=1}^k\in\Rbb_+^{r}$. Each spectral base is  denoted as an endmember, which represents the pure spectrum. For example, a spectral base could be the spectrum of ``rock", ``tree" etc.

Linear mixture model \cite{Keshava2002Spectral} approximates the pixel spectrum $\abf$  by a linear combination of endmembers as $
\abf= Y\xbf+\mathbf{r}$,
where $\xbf\in\Rbb_+^{k}$ is called the abundance vector corresponding to pixel $\abf$, $\mathbf{r}\in\Rbb^{r}$ is a residual term and $Y=[\ybf_1,\dots,\ybf_k ]\in\Rbb_+^{r\times k}$ is the endmember matrix. When ONMF is applied to hyperspectral unmixing, we assume that both endmember and abundances remain unknown. In addition, each pixel only corresponds to one material. That is to say, $\xbf$ only has one non-zero element. For all the pixels combined together, the ONMF formulation of hyperspectral image unmixing becomes \eqref{equ:prob:onmf},
where $A=[\abf_1,\dots,\abf_n]^{\tran}\in\Rbb_+^{n\times r}$ is a hyperspectral image matrix with row vectors correspond to its pixels and $X\in \stiefplus$ is the abundance matrix with $X_{i,:}$ representing the $i$-th abundance vector for $i\in[n]$.

We test algorithms on three hyperspectral image datasets, Samson, Jasper Ridge and Urban \cite{zhu2014spectral}. 
 They are widely used datasets in the hyperspectral unmixing study and can be downloaded at \url{http://www.escience.cn/people/feiyunZHU/Dataset\_GT.html}.  
Since the sizes of the first two images are huge, we choose a region in each image. This process is common in the context of hyperspectral unmixing. For Samson, a region which contains $95\times 95$ pixels is chosen,  starting from the $(252,332)$-th pixel in original image. We choose a subimage of Jasper Ridge with $100\times 100$ pixels, whose first pixel corresponds to the $(105,269)$-th pixel in the original image. The size of refined Samson is $156\times 95\times 95$, which contains three endmembers: water, tree and rock. The size of refined Jasper Ridge is  $198\times 100\times 100$, and its endmembers include water, tree, dirt and road. Urban is the largest hyperspectral data with $307\times 307$ pixels observed at $162$ wavelengths, and there are four endmembers:  asphal, grass, tree and roof. \cref{num:city} gives an illustration of these datasets.

\begin{figure}[!htbp]
\centering
\subfigure[Samson]{
\begin{minipage}[b]{0.22\linewidth}
\includegraphics[height=0.85\linewidth]{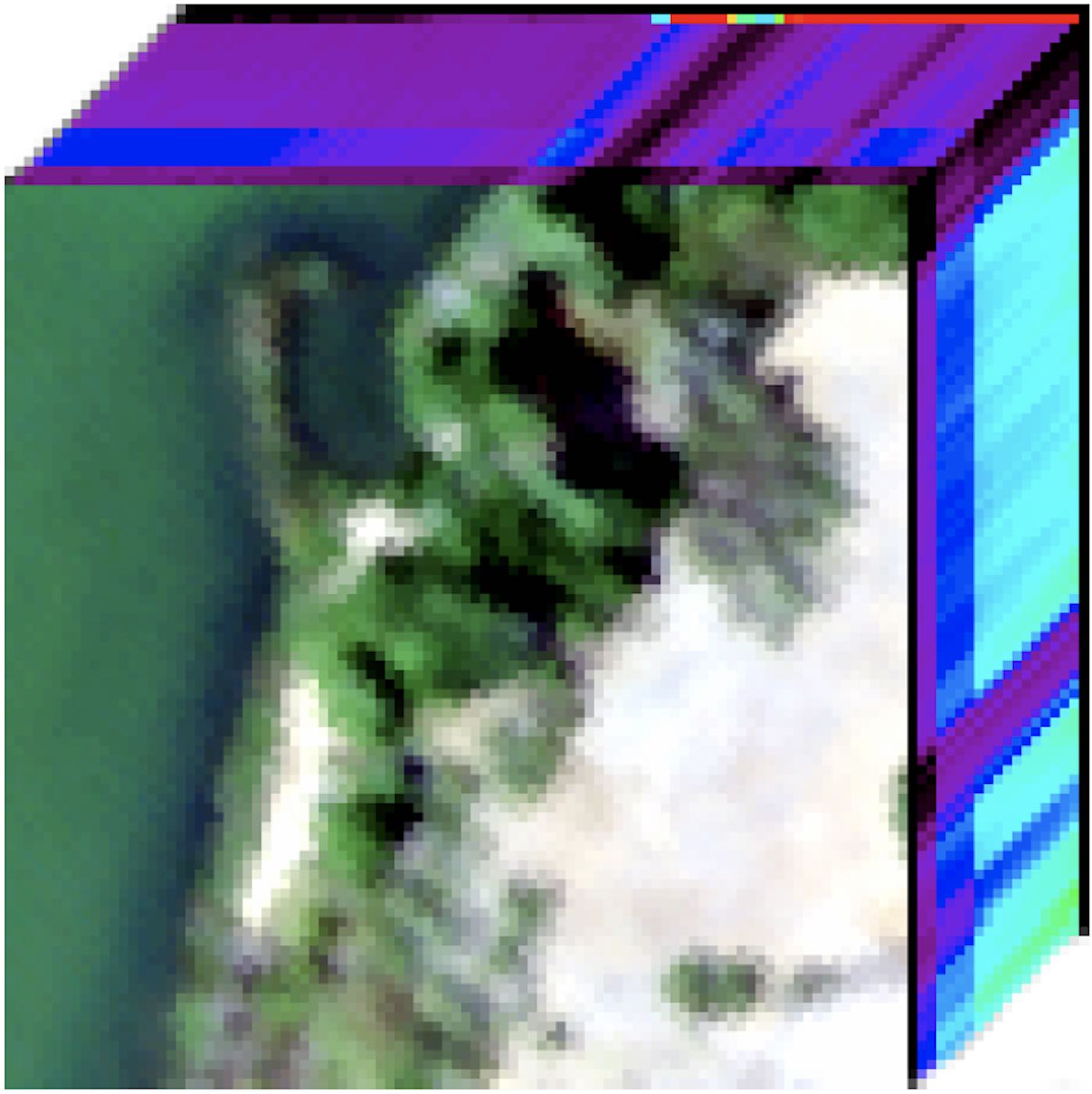}
\end{minipage}
}
\hspace{0.5cm}
\subfigure[Jasper Ridge]{
\begin{minipage}[b]{0.22\linewidth}
\includegraphics[height=0.85\linewidth]{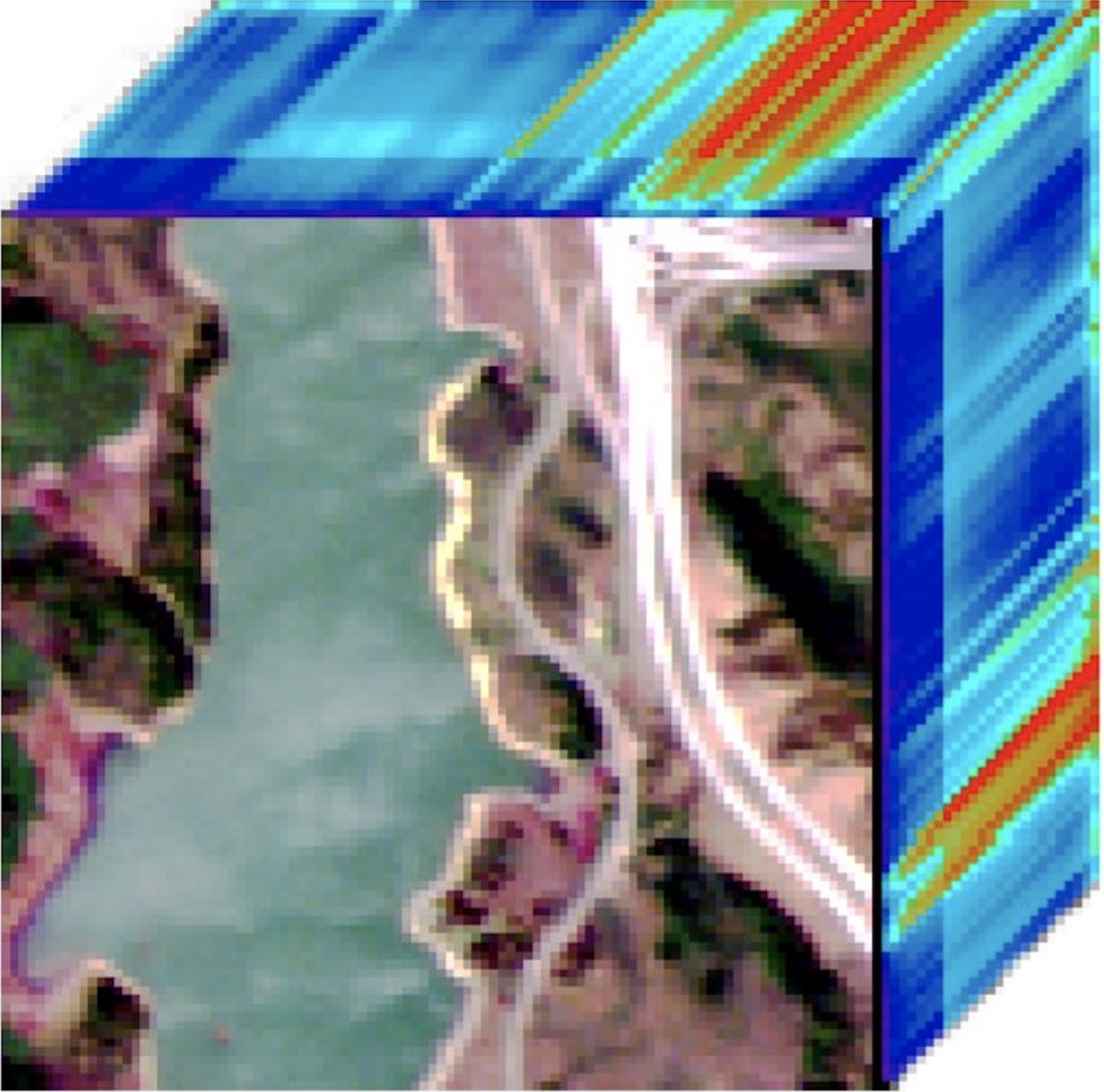}
\end{minipage}
}
\hspace{0.5cm}
\subfigure[Urban]{
\begin{minipage}[b]{0.22\linewidth}
\includegraphics[height=0.85\linewidth]{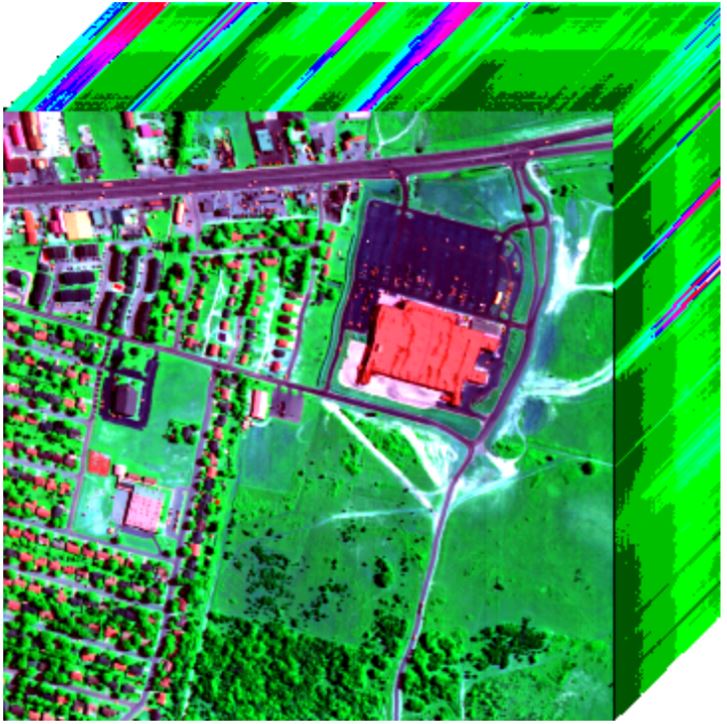}
\end{minipage}
}
\caption{
Three real hyperspectral images}\label{num:city}
\end{figure}

Since the groundtruth of abundance matrix $X$ does not satisfy the orthogonality constraints, the criteria utilized in the preceding subsection are not appropriate to measure the quality of hyperspectral unmixing. Here we consider spectral angle distance (SAD) (see for instance \cite{zhu2014spectral}) to evaluate the performance of  algorithms. SAD uses an angle distance between groundtruth and estimated endmembers to measure the accuracy of endmember estimation. It is defined as
$
\text{SAD}\coloneqq\frac{1}{k}\sum_{i=1}^k\arccos\left(\frac{y_i^{\tran}\hat{y_i}}{\|y_i\| \|\hat{y_i}\|}\right),
$
where $\hat{y_i}$ and $y_i$ are estimation of $i$-th endmember and its corresponding groundtruth. Smaller SAD corresponds to better performance. Since other algorithms cannot generate a solution of problem \eqref{equ:prob:onmf} with small feasibility violation,  in order to keep a fair comparison, we perform the rounding procedure and postprocessing on the solution generated by each method. Note that the postprocessing problem \eqref{equ:prob:refinement} is easy to solve, it mainly needs to find the maximum singular value and corresponding singular vector for $k$ small scale matrices.

The unmixing results of Samson, Jasper Ridge and Urban are illustrated in \cref{num:samson}, \cref{num:Japser} and \cref{num:Urban}, respectively. For  Samson image, our method and ONP-MF are able to separate three  endmembers, while the rest methods mix them together. For Jasper Ridge image, none of the methods can identify the road endmember, while our method and K-means can split water from other endmembers completely. All of algorithms perform relatively well on Urban dataset except for K-means, being able to separate four endmembers.


\begin{figure}[!htbp]
\centering
\subfigure[\scriptsize \!\!ground truth]{
\begin{minipage}[b]{0.137\linewidth}
\includegraphics[width=\linewidth]{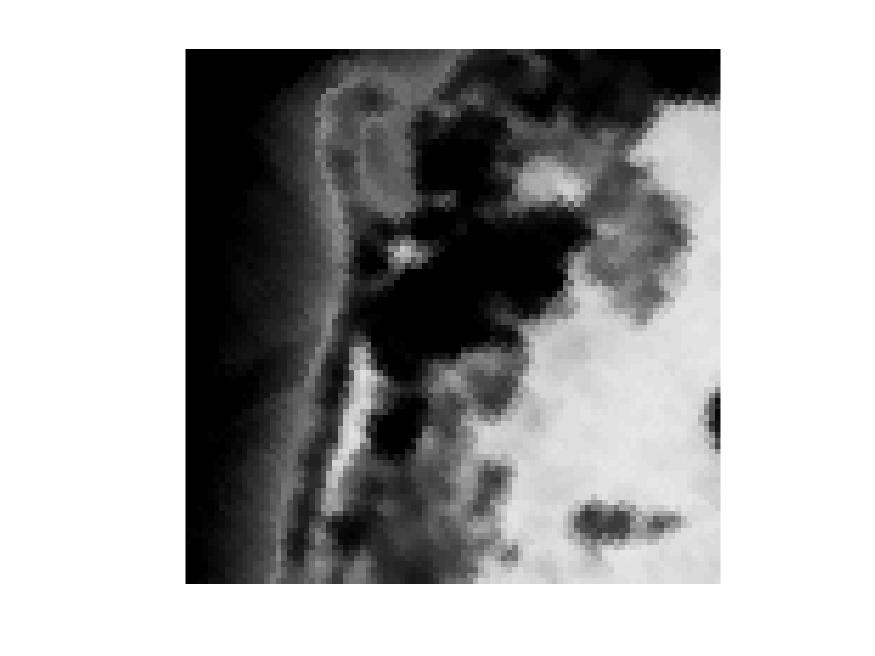}
\includegraphics[width=\linewidth]{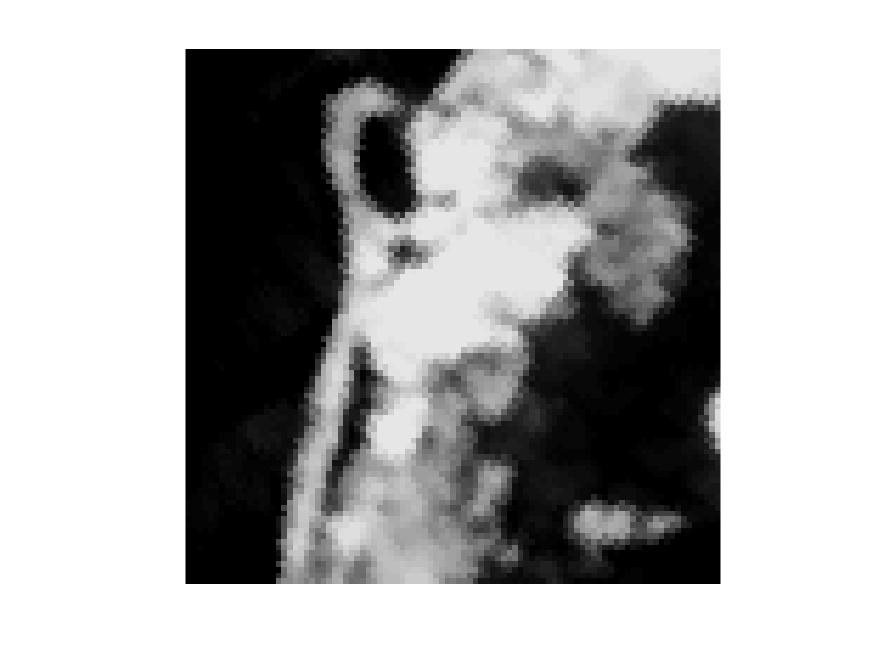}
\includegraphics[width=\linewidth]{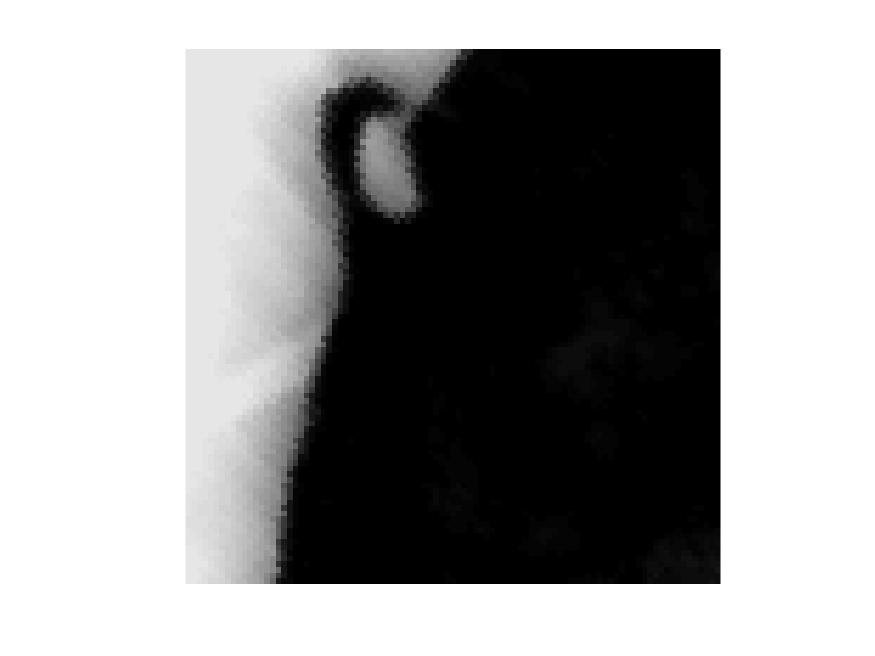}
\end{minipage}
}
\subfigure[\scriptsize \!EP4Orth+]{
\begin{minipage}[b]{0.137\linewidth}
\includegraphics[width=\linewidth]{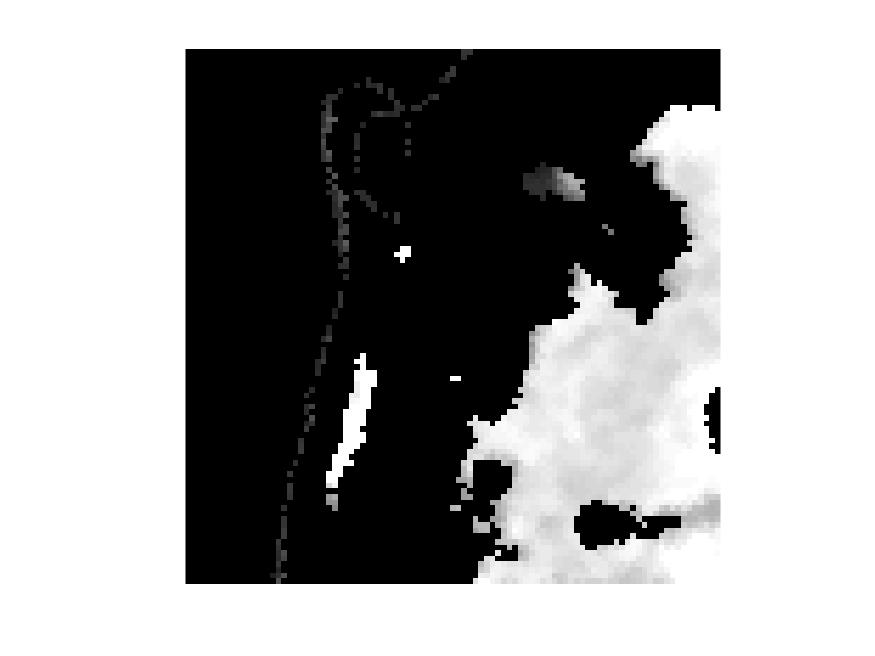}
\includegraphics[width=\linewidth]{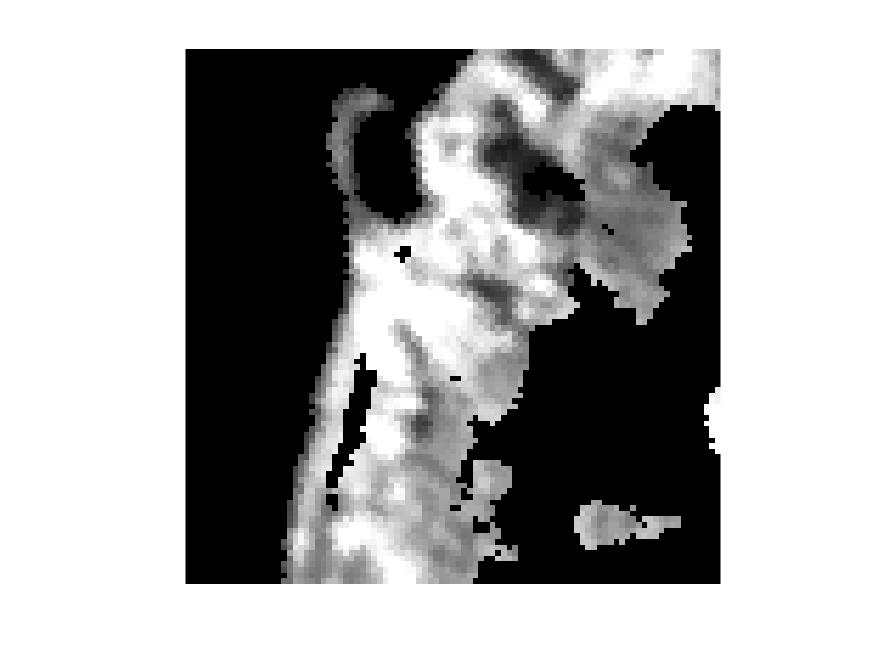}
\includegraphics[width=\linewidth]{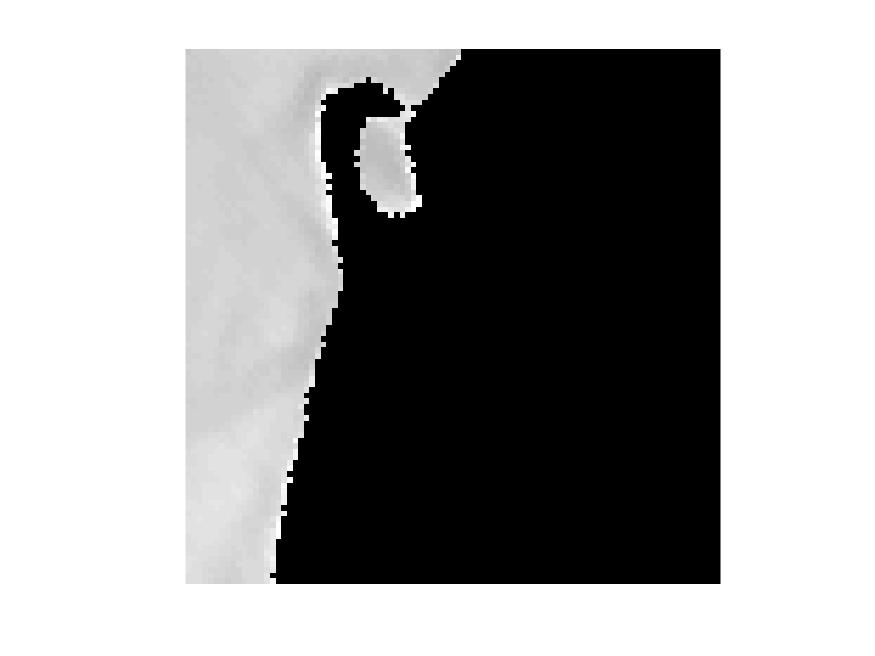}
\end{minipage}
}\hspace{-0.33cm}
\subfigure[U-onmf]{
\begin{minipage}[b]{0.137\linewidth}
\includegraphics[width=\linewidth]{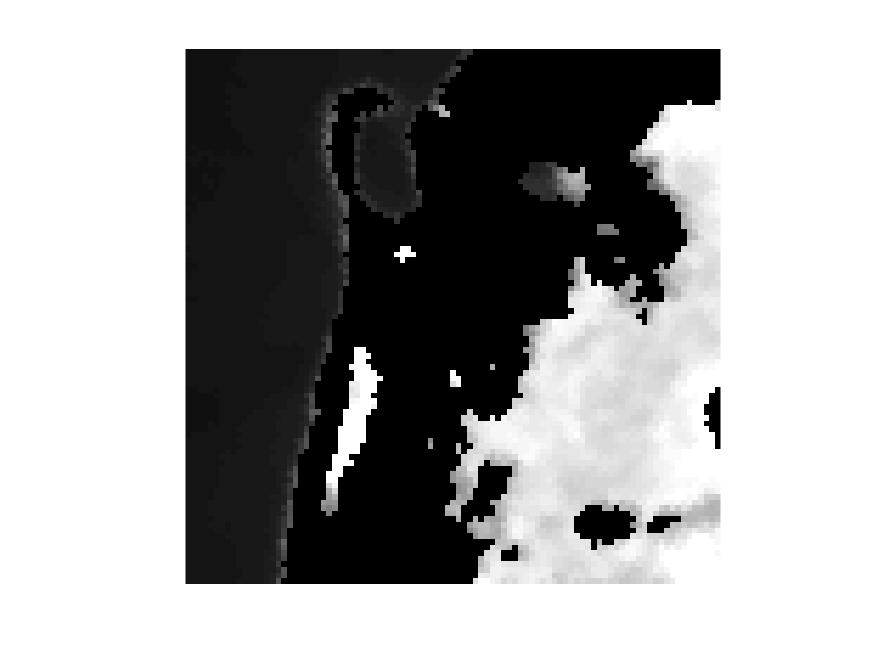}
\includegraphics[width=\linewidth]{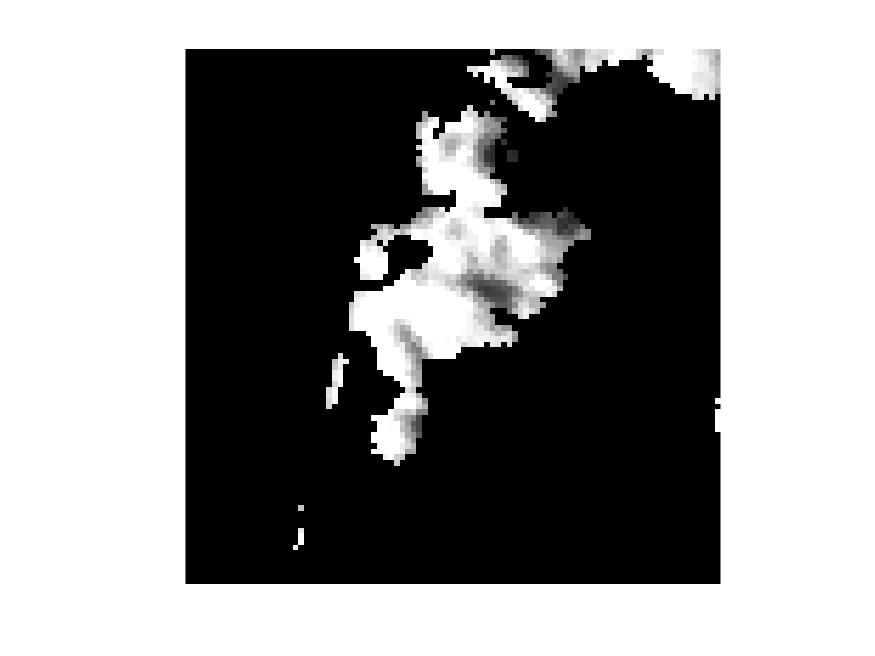}
\includegraphics[width=\linewidth]{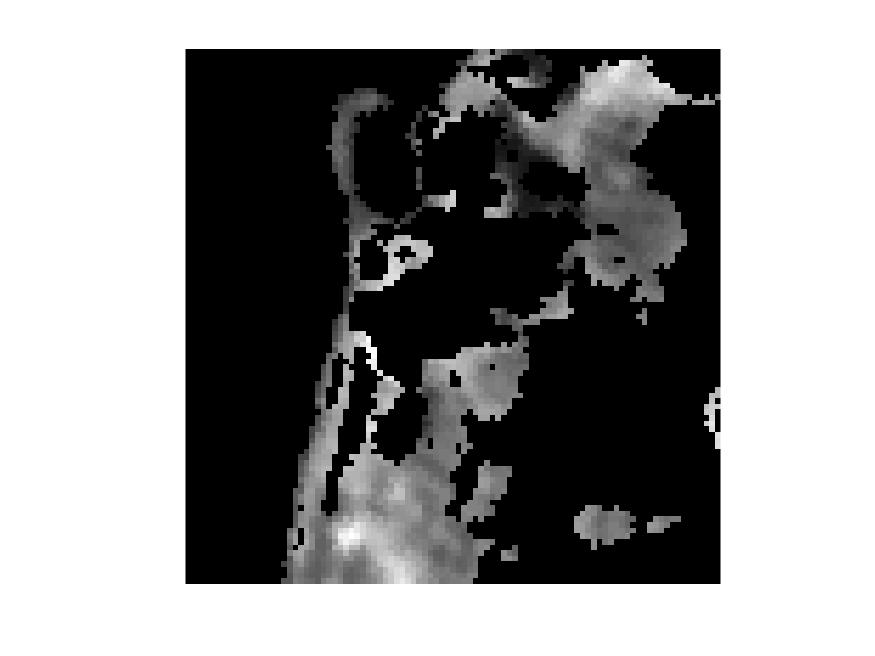}
\end{minipage}
}\hspace{-0.33cm}
\subfigure[OPNMF]{
\begin{minipage}[b]{0.137\linewidth}
\includegraphics[width=\linewidth]{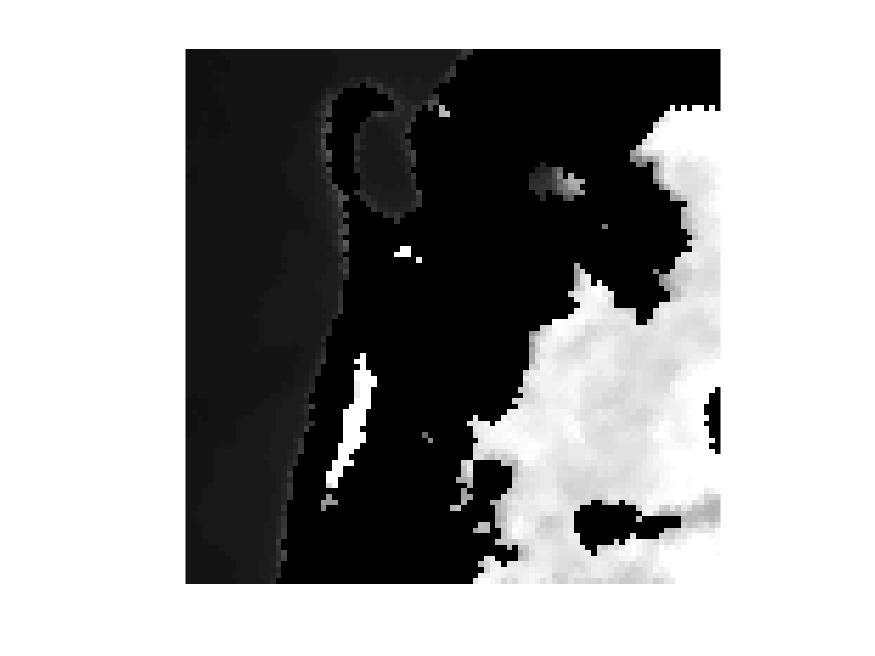}
\includegraphics[width=\linewidth]{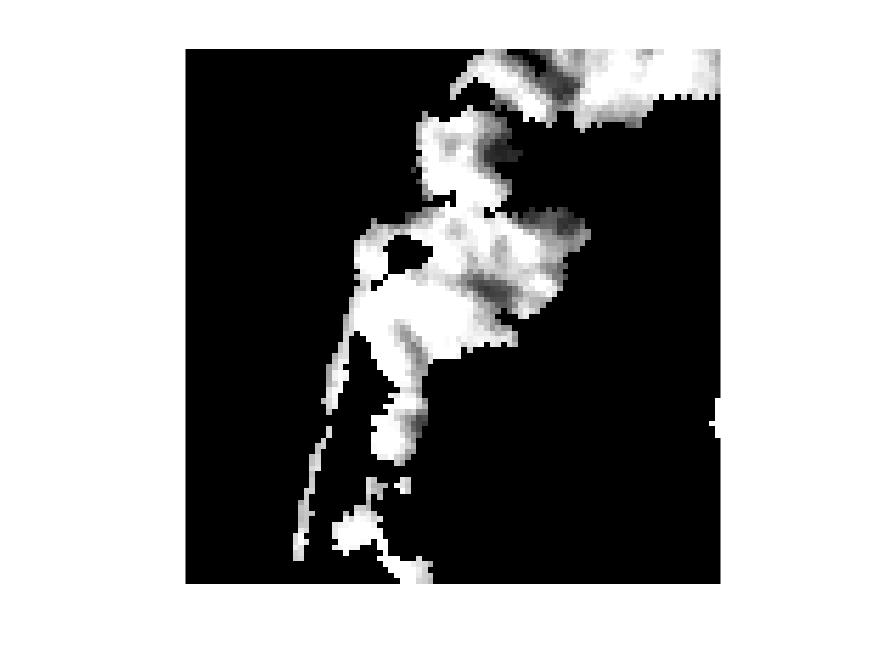}
\includegraphics[width=\linewidth]{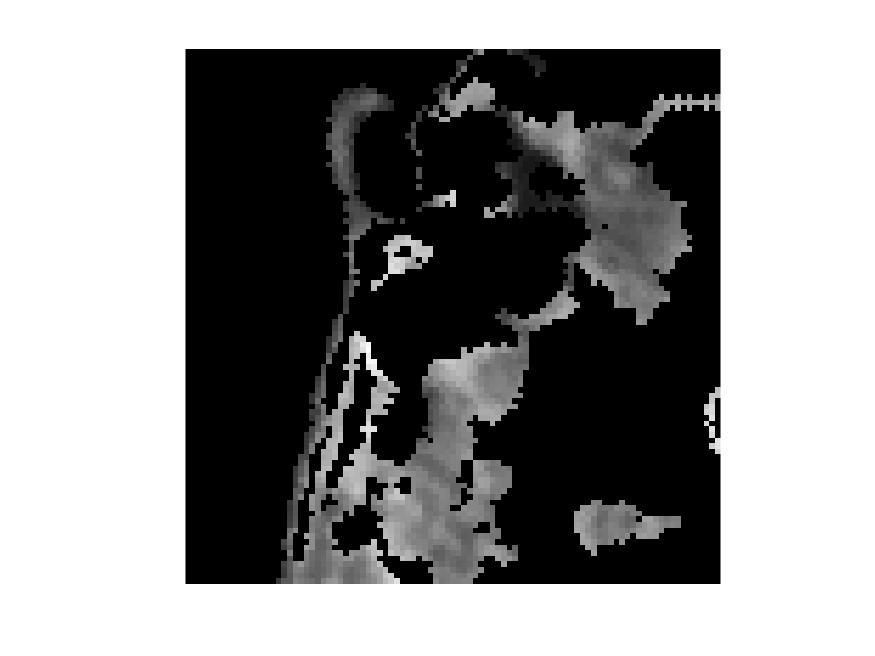}
\end{minipage}
}\hspace{-0.33cm}
\subfigure[K-means]{
\begin{minipage}[b]{0.137\linewidth}
\includegraphics[width=\linewidth]{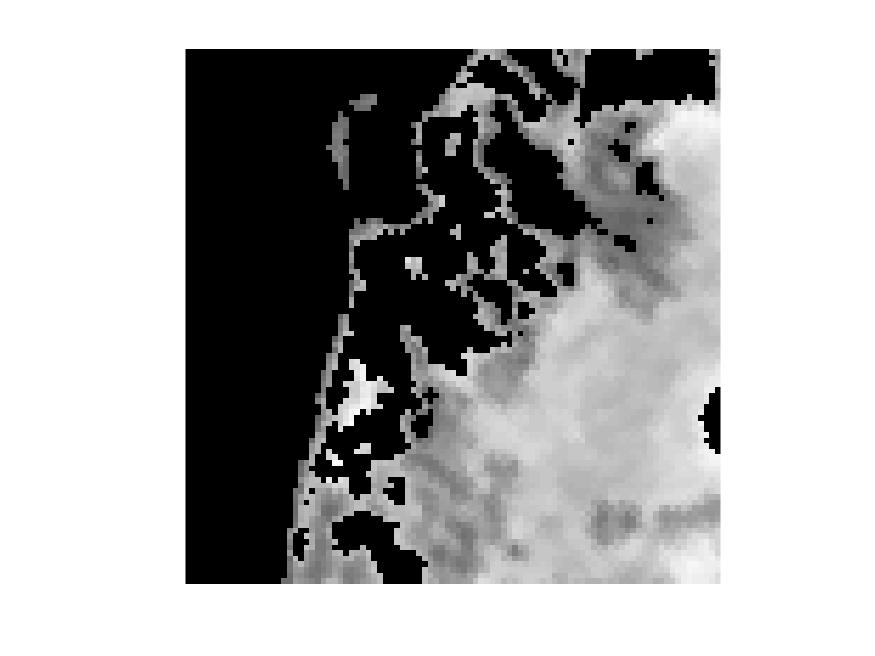}
\includegraphics[width=\linewidth]{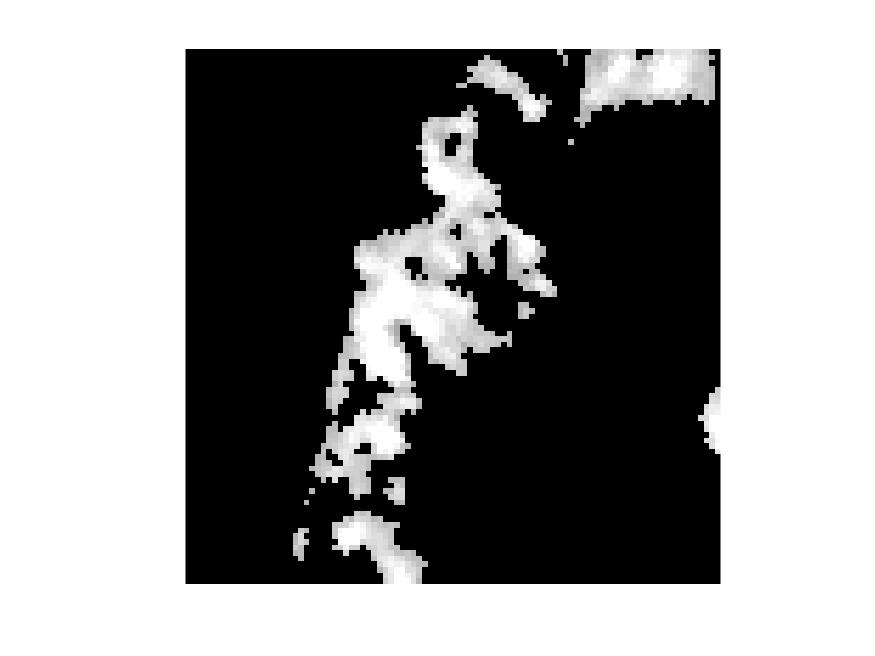}
\includegraphics[width=\linewidth]{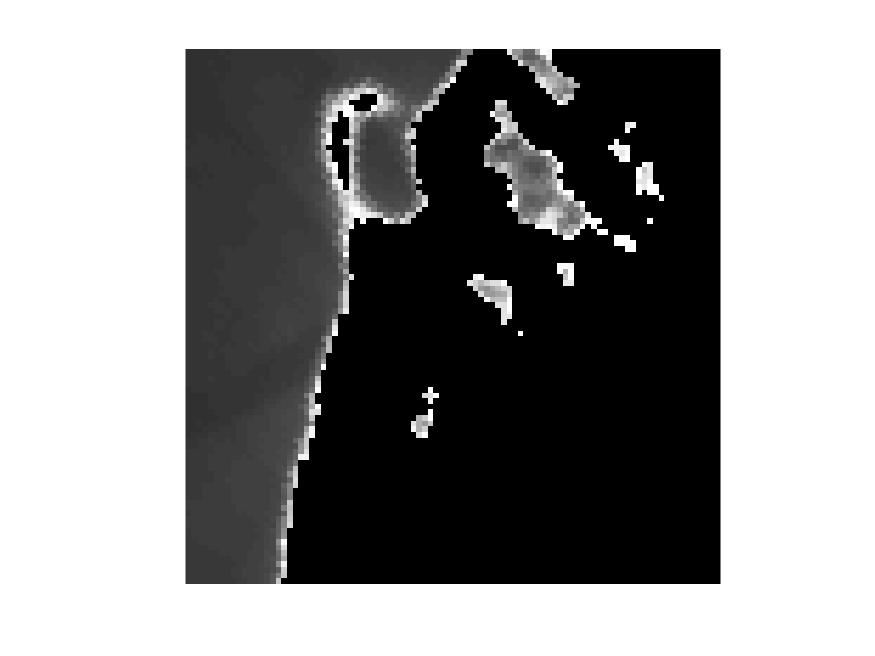}
\end{minipage}
}\hspace{-0.33cm}
\subfigure[ONP-MF]{
\begin{minipage}[b]{0.137\linewidth}
\includegraphics[width=\linewidth]{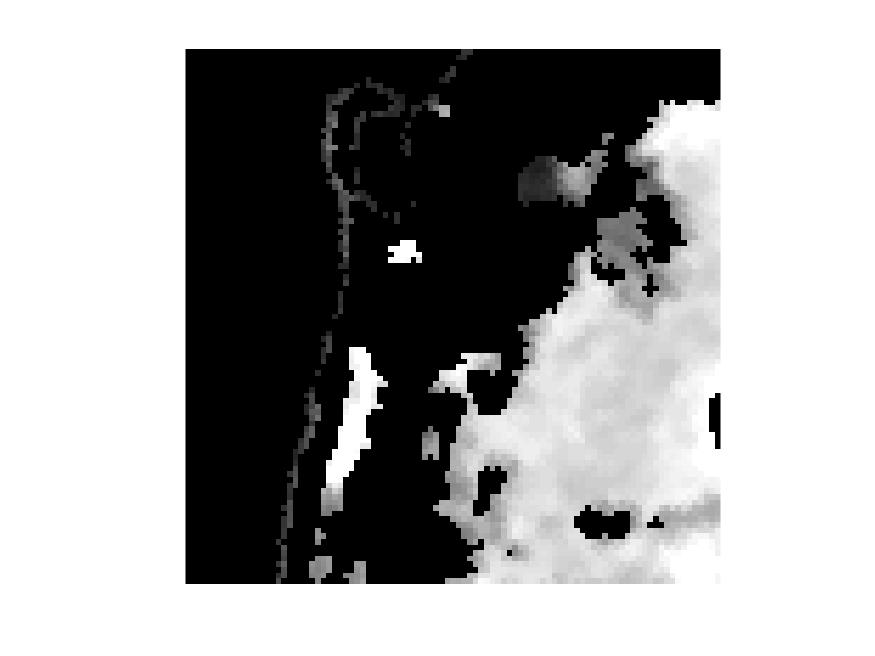}
\includegraphics[width=\linewidth]{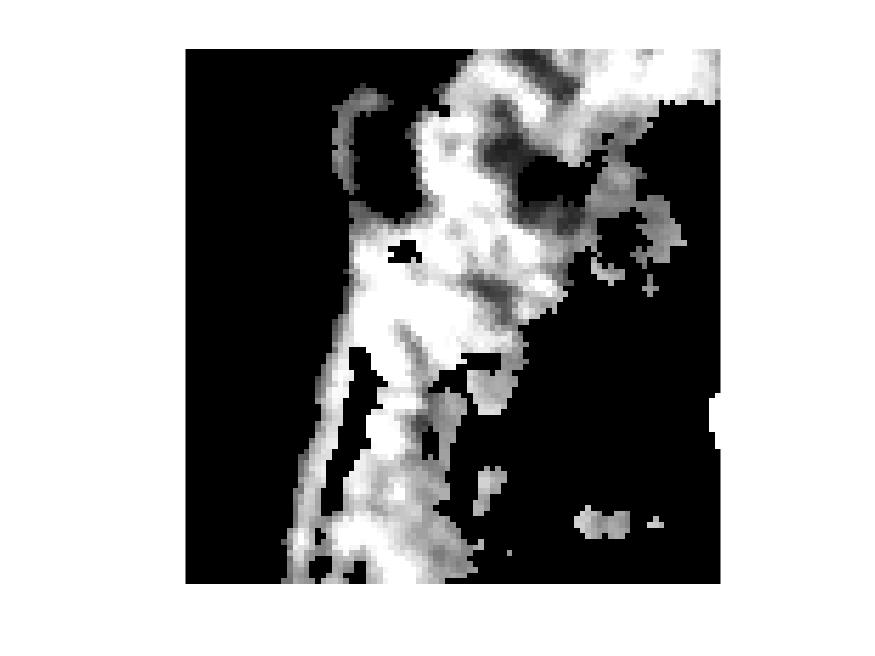}
\includegraphics[width=\linewidth]{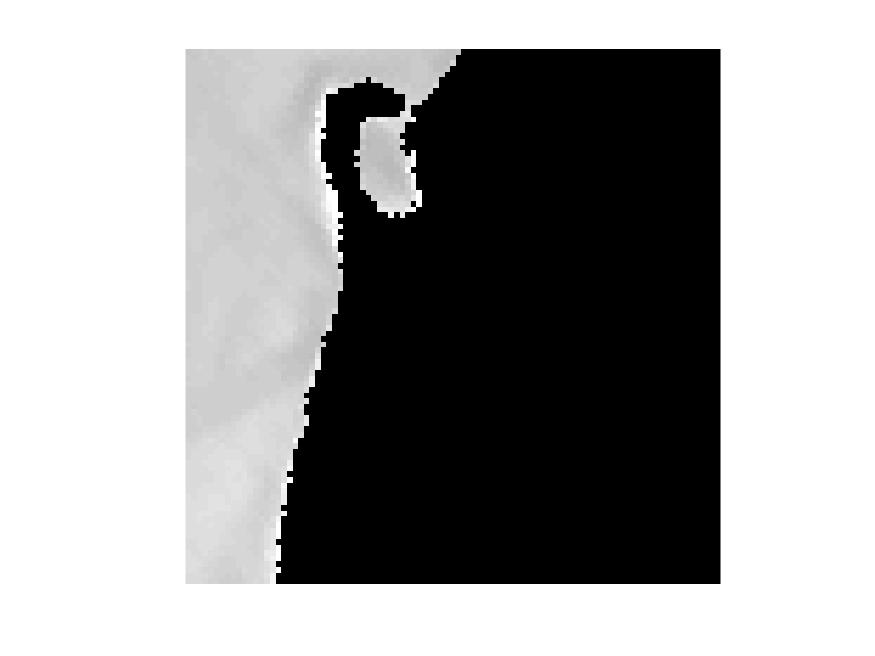}
\end{minipage}
}\hspace{-0.33cm}
\subfigure[EM-onmf]{
\begin{minipage}[b]{0.137\linewidth}
\includegraphics[width=\linewidth]{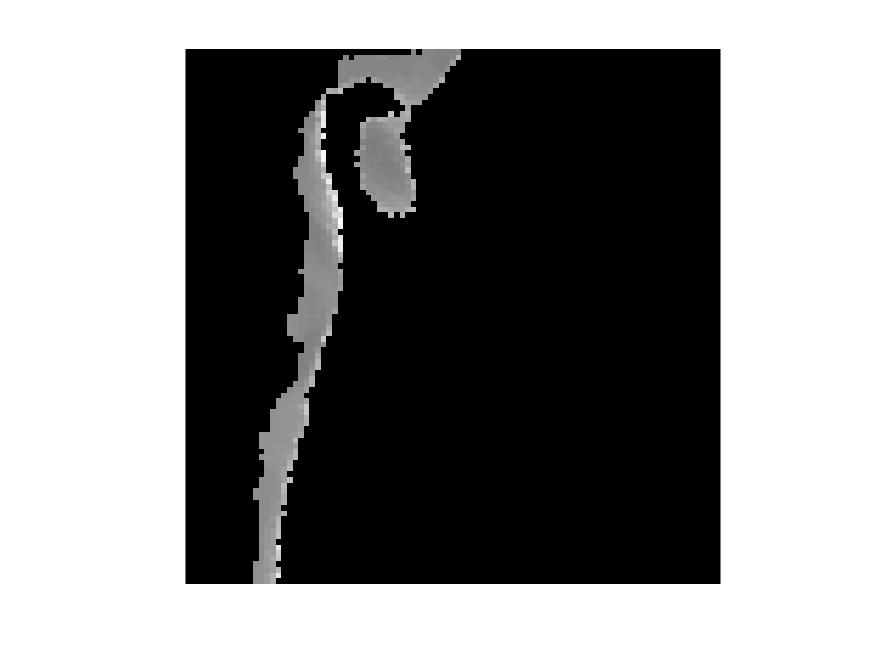}
\includegraphics[width=\linewidth]{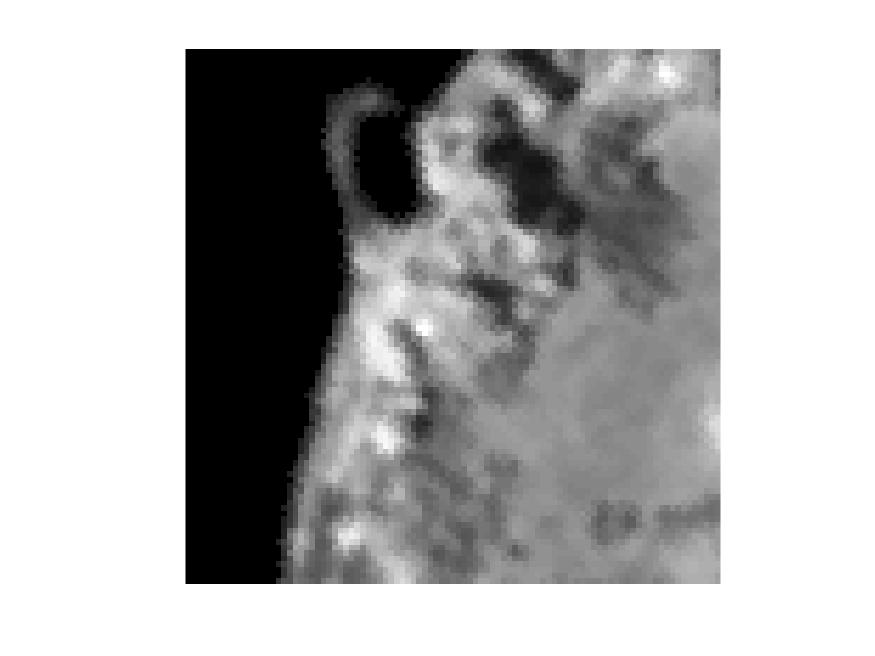}
\includegraphics[width=\linewidth]{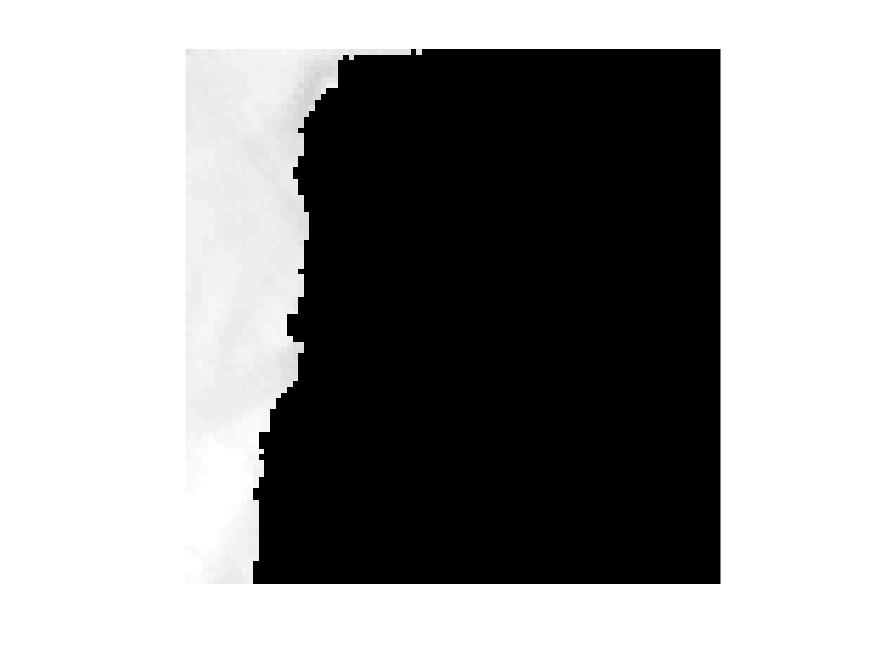}
\end{minipage}
}
\caption{
Unmixing results of Samson, from top to bottom: rock, tree, water.}
\label{num:samson}
\end{figure}

\begin{figure}[!htbp]
\centering
\subfigure[\scriptsize \!\!ground truth]{
\begin{minipage}[b]{0.137\linewidth}
\includegraphics[width=\textwidth]{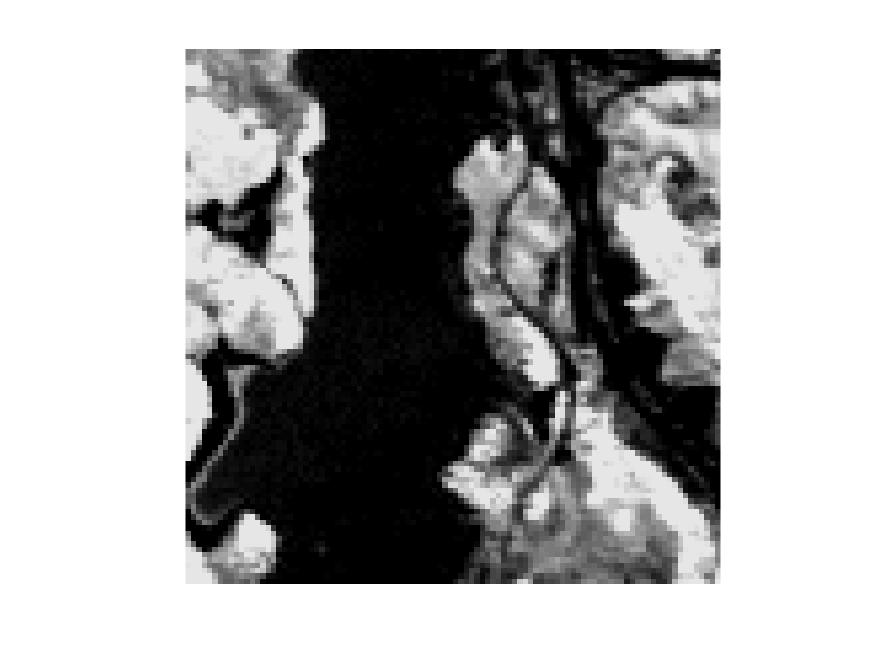}
\includegraphics[width=\textwidth]{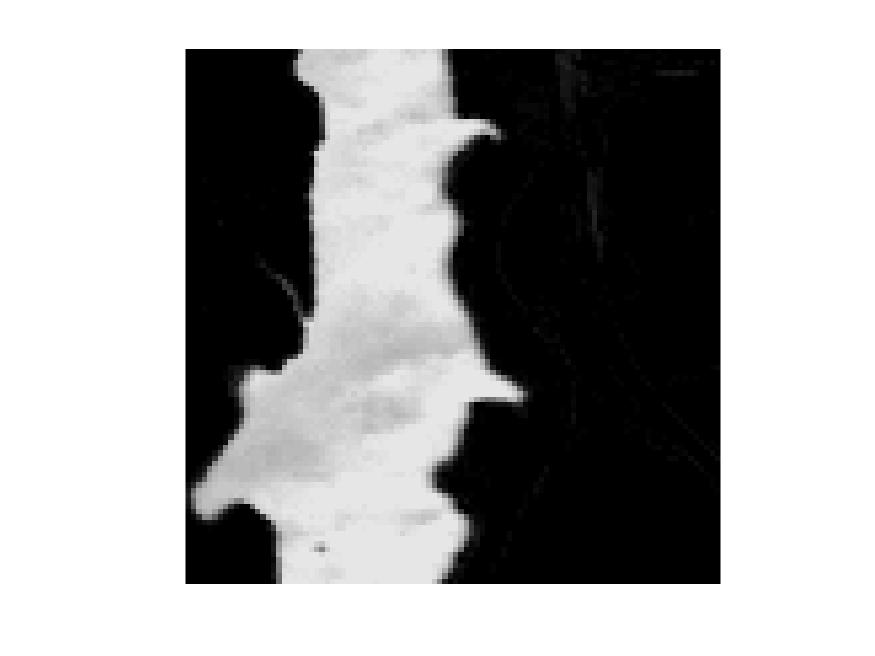}
\includegraphics[width=\textwidth]{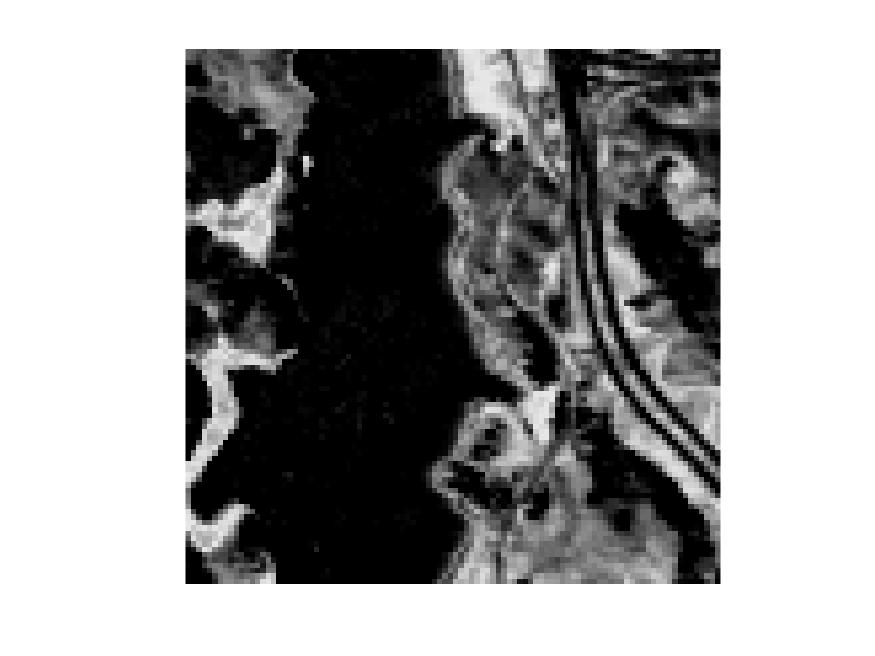}
\includegraphics[width=\textwidth]{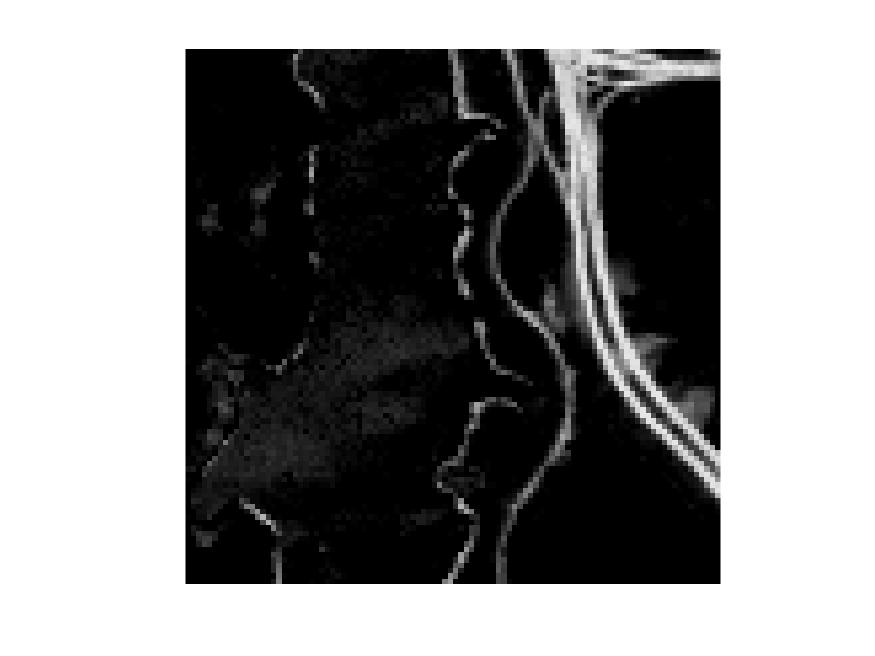}
\end{minipage}
}
\subfigure[\scriptsize \!EP4Orth+]{
\begin{minipage}[b]{0.137\linewidth}
\includegraphics[width=\textwidth]{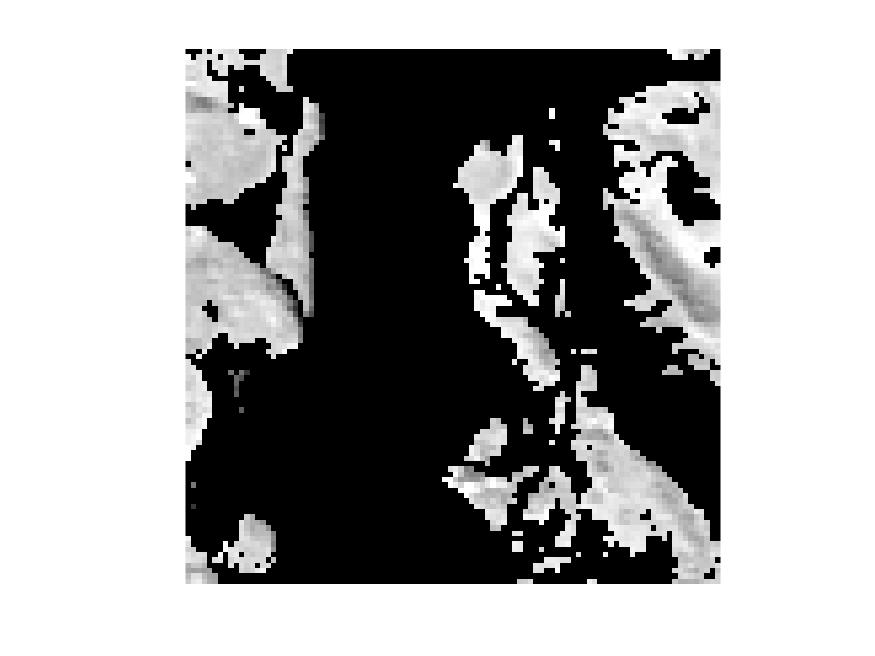}
\includegraphics[width=\textwidth]{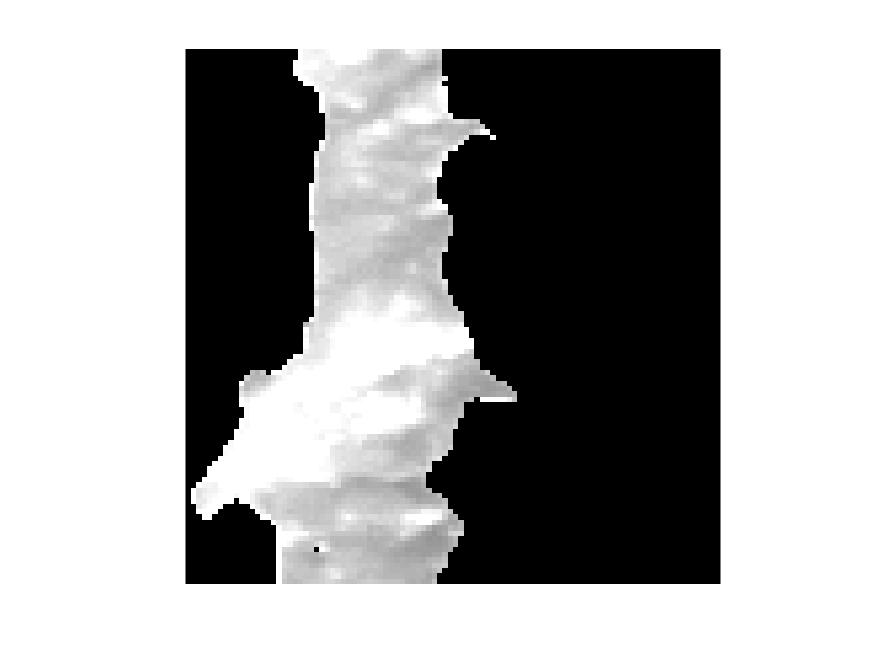}
\includegraphics[width=\textwidth]{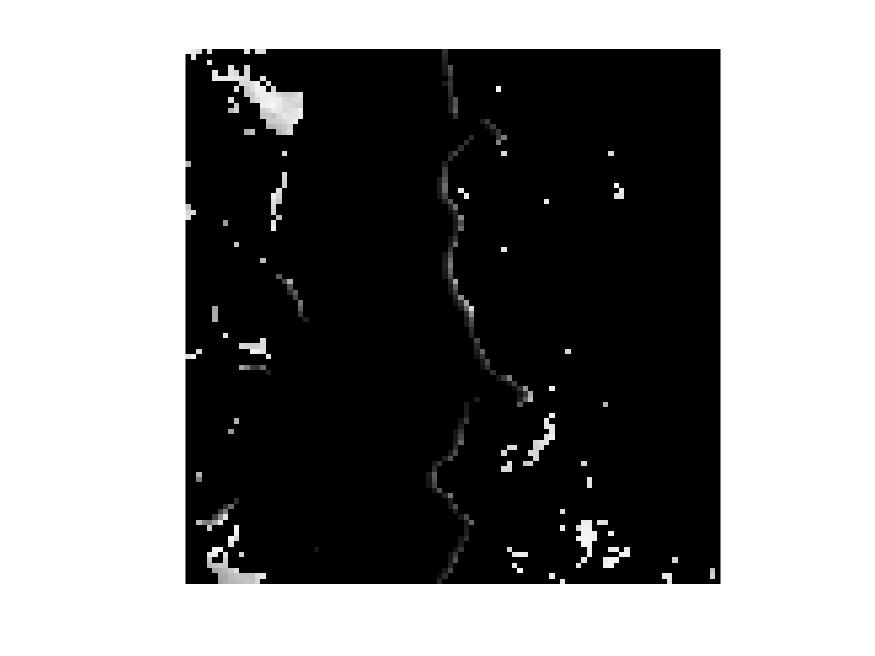}
\includegraphics[width=\textwidth]{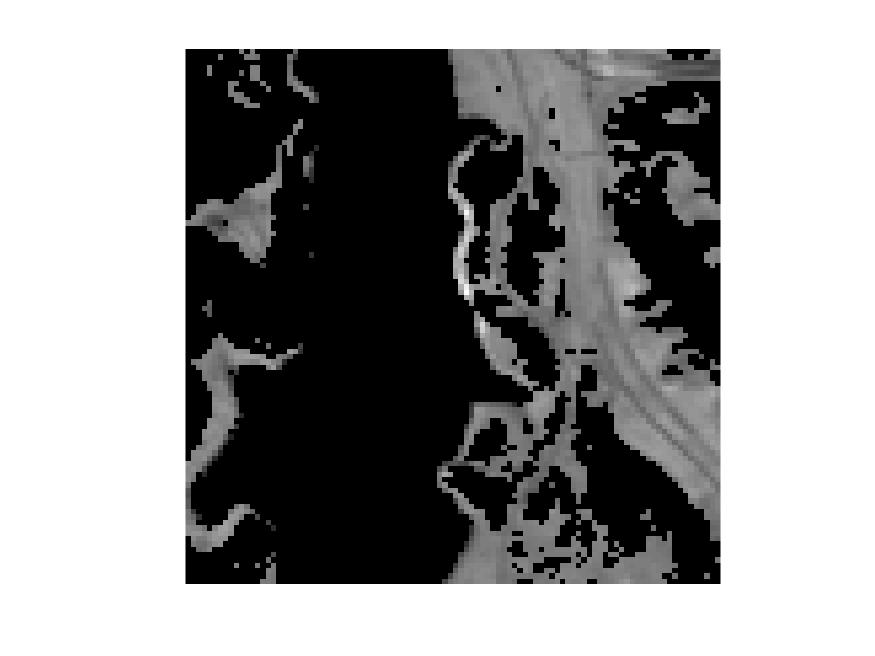}
\end{minipage}
}\hspace{-0.33cm}
\subfigure[U-onmf]{
\begin{minipage}[b]{0.137\linewidth}
\includegraphics[width=\textwidth]{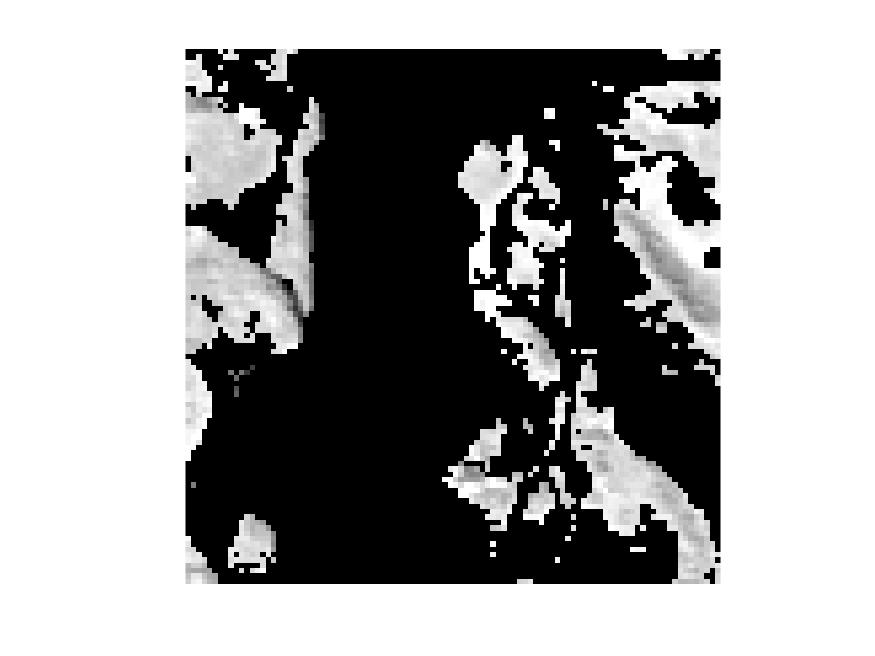}
\includegraphics[width=\textwidth]{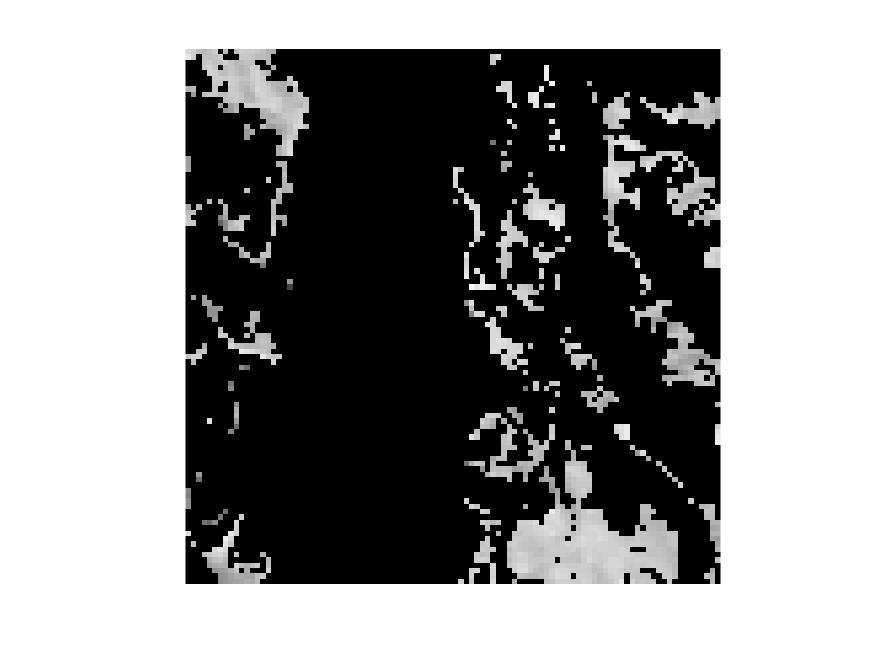}
\includegraphics[width=\textwidth]{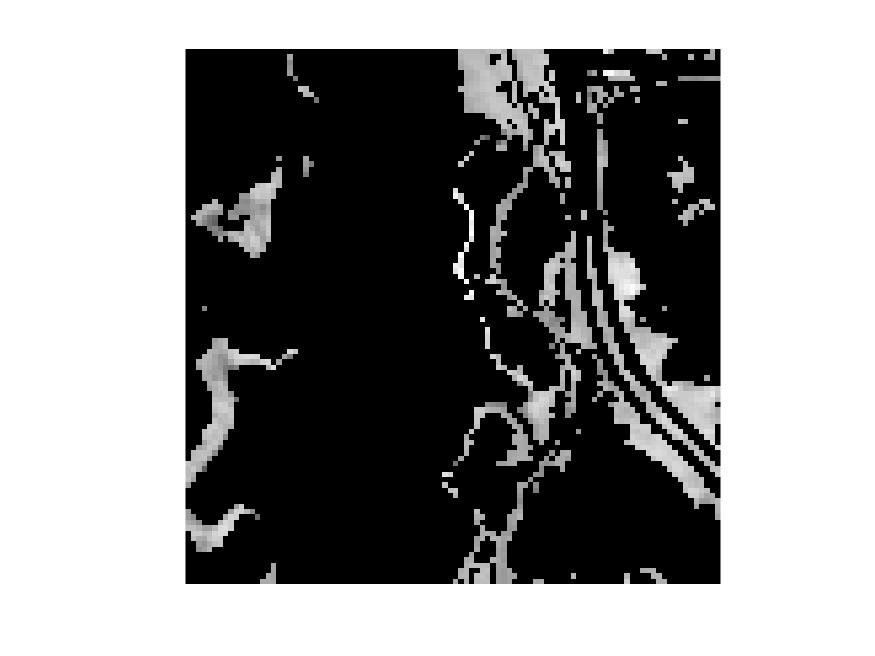}
\includegraphics[width=\textwidth]{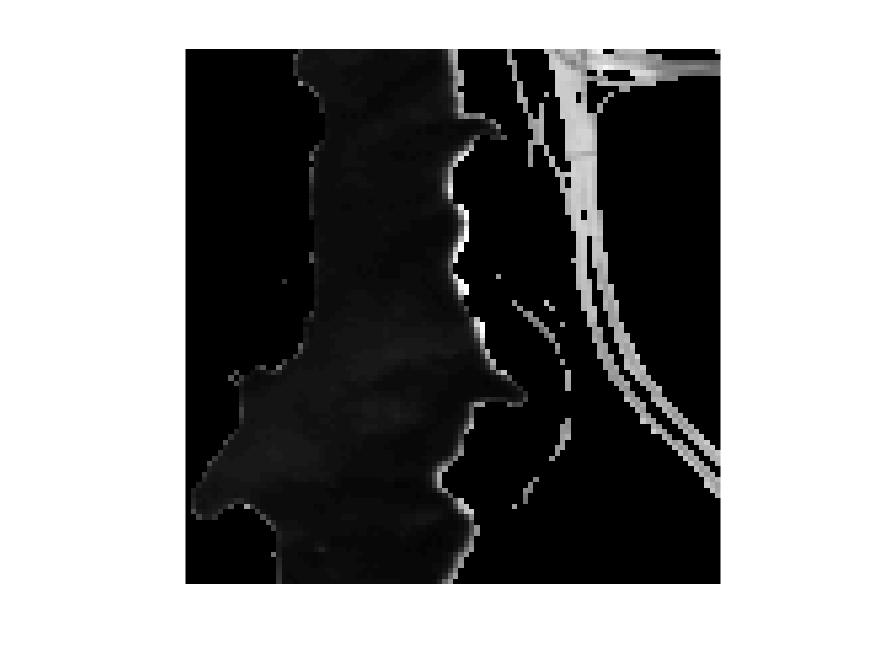}
\end{minipage}
}\hspace{-0.33cm}
\subfigure[OPNMF]{
\begin{minipage}[b]{0.137\linewidth}
\includegraphics[width=\textwidth]{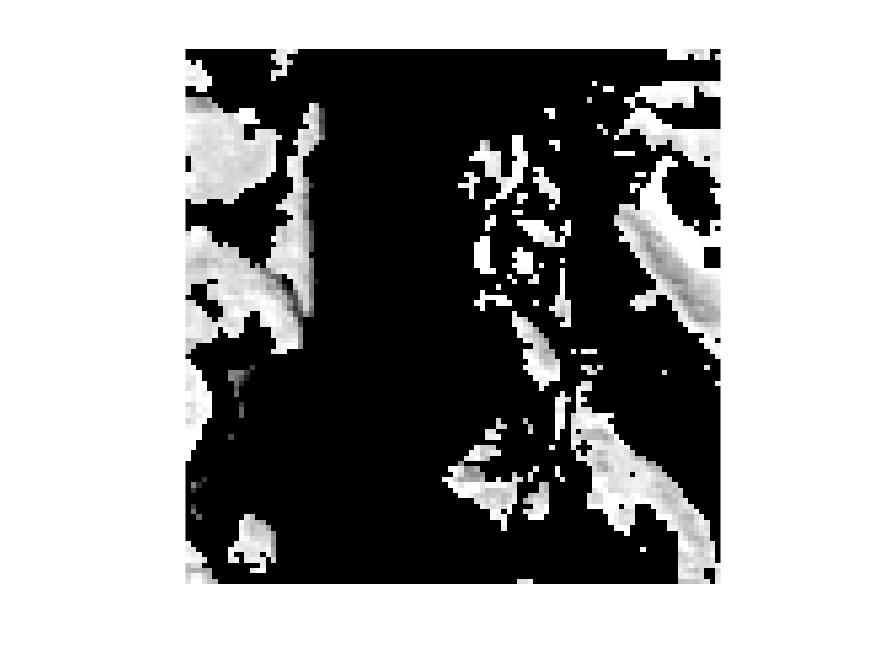}
\includegraphics[width=\textwidth]{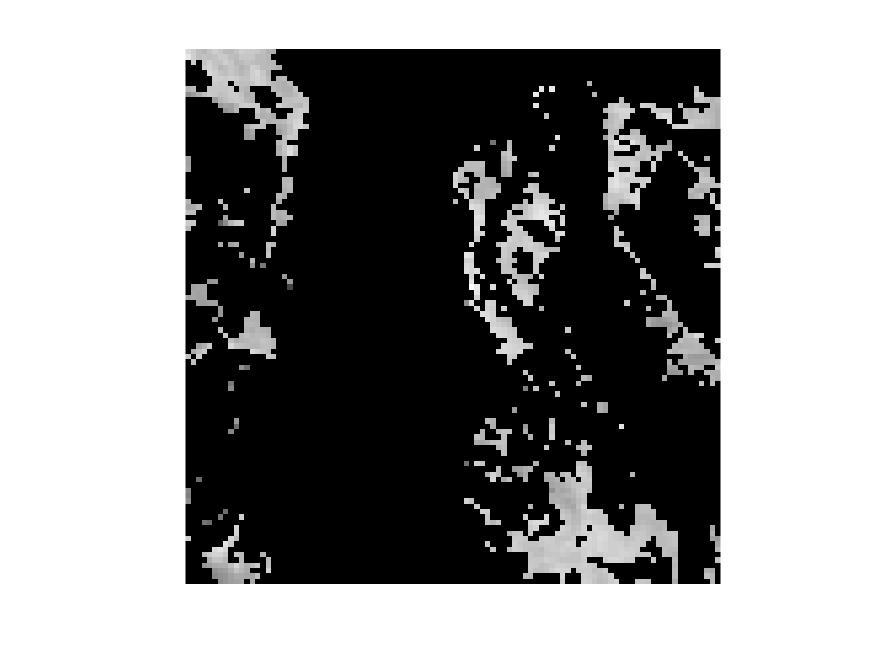}
\includegraphics[width=\textwidth]{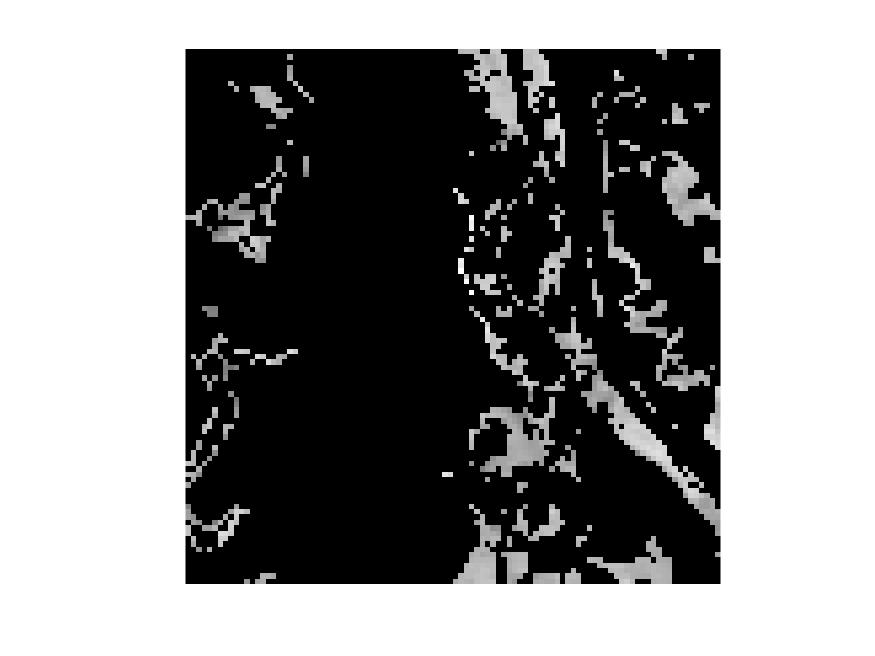}
\includegraphics[width=\textwidth]{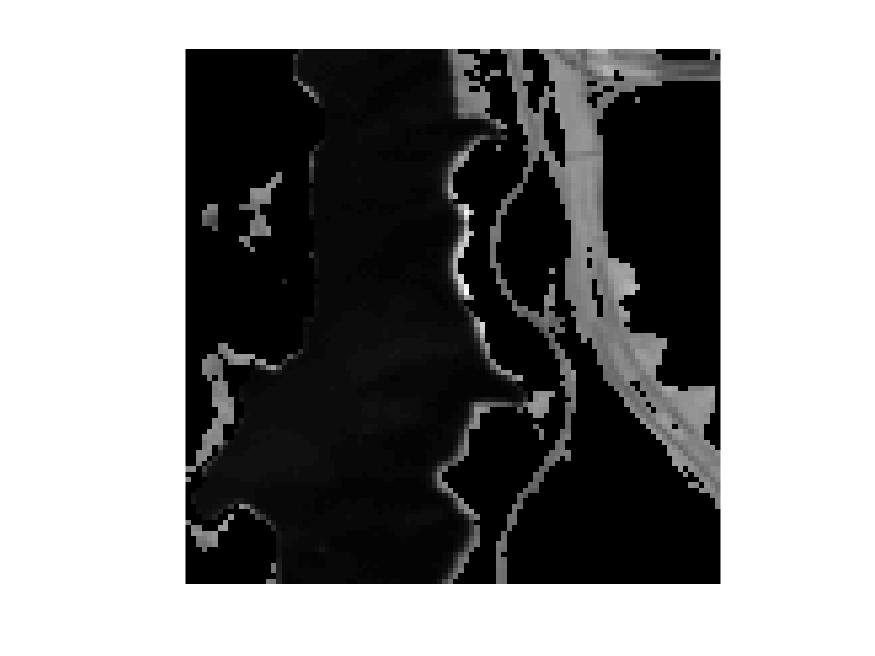}
\end{minipage}
}\hspace{-0.33cm}
\subfigure[K-means]{
\begin{minipage}[b]{0.137\linewidth}
\includegraphics[width=\textwidth]{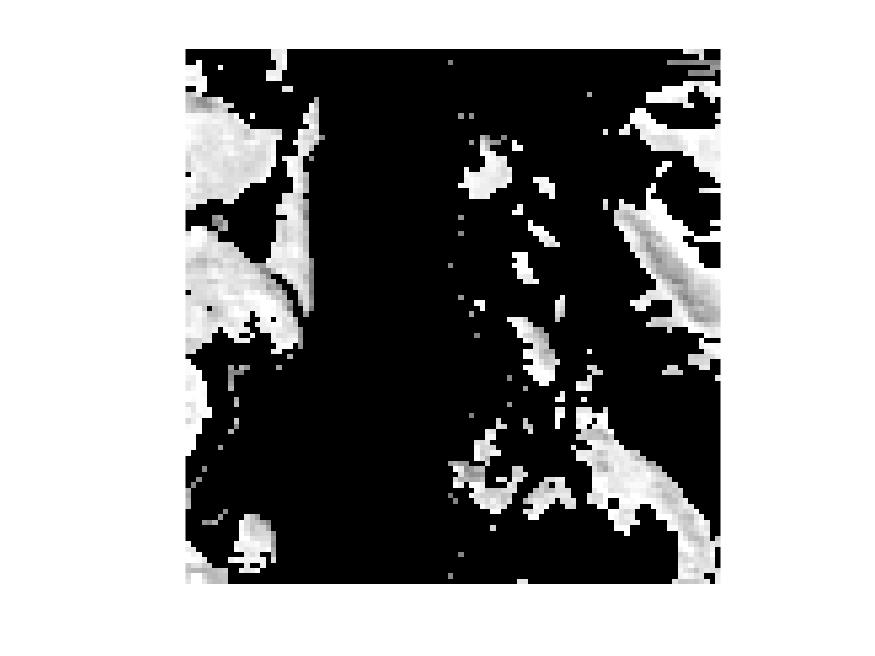}
\includegraphics[width=\textwidth]{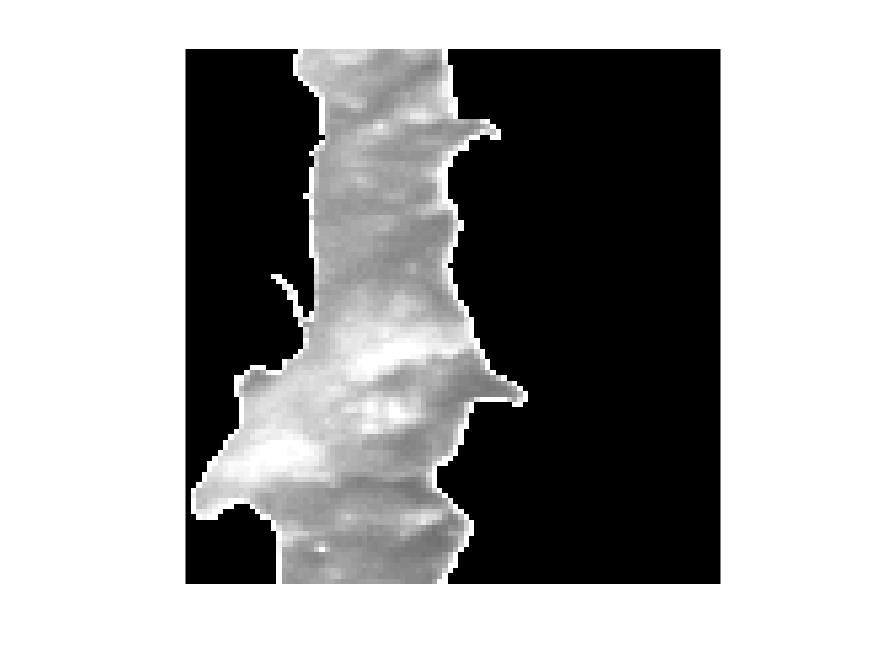}
\includegraphics[width=\textwidth]{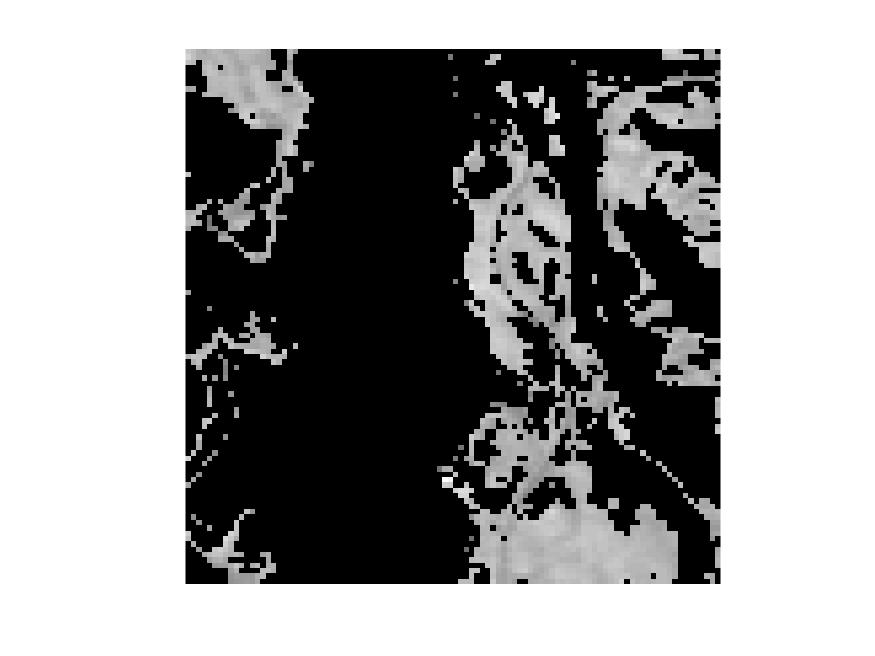}
\includegraphics[width=\textwidth]{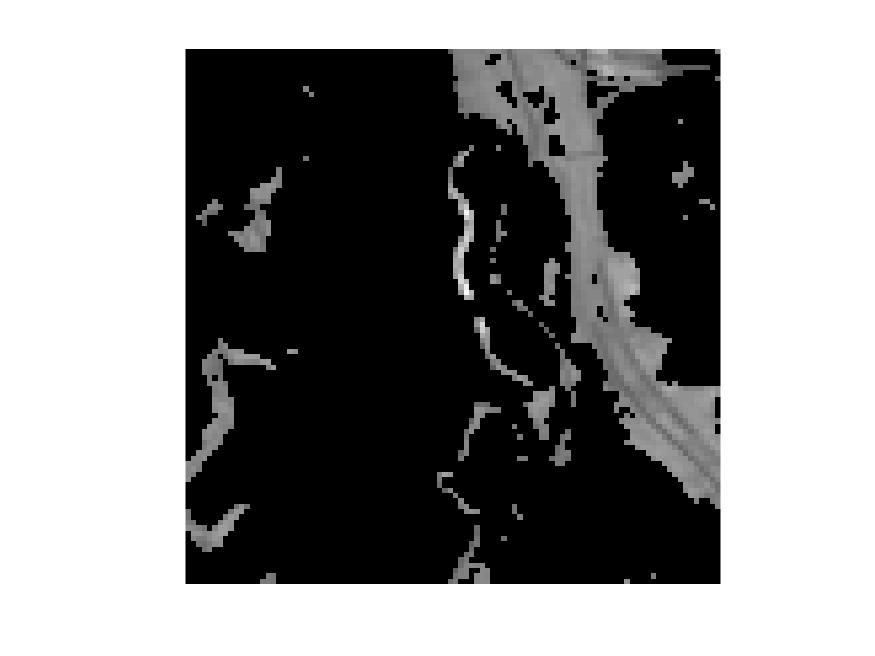}
\end{minipage}
}\hspace{-0.33cm}
\subfigure[ONP-MF]{
\begin{minipage}[b]{0.137\linewidth}
\includegraphics[width=\textwidth]{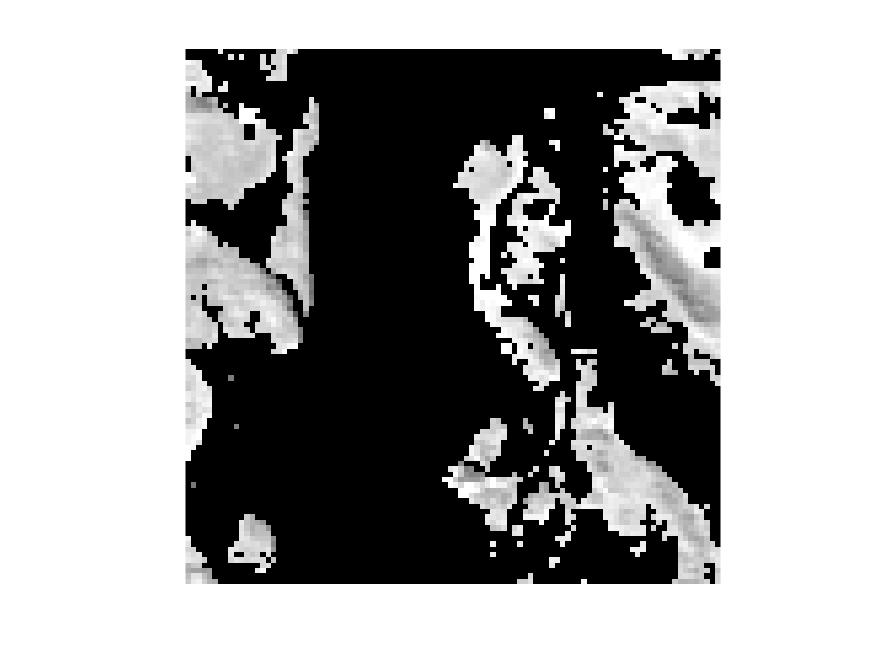}
\includegraphics[width=\textwidth]{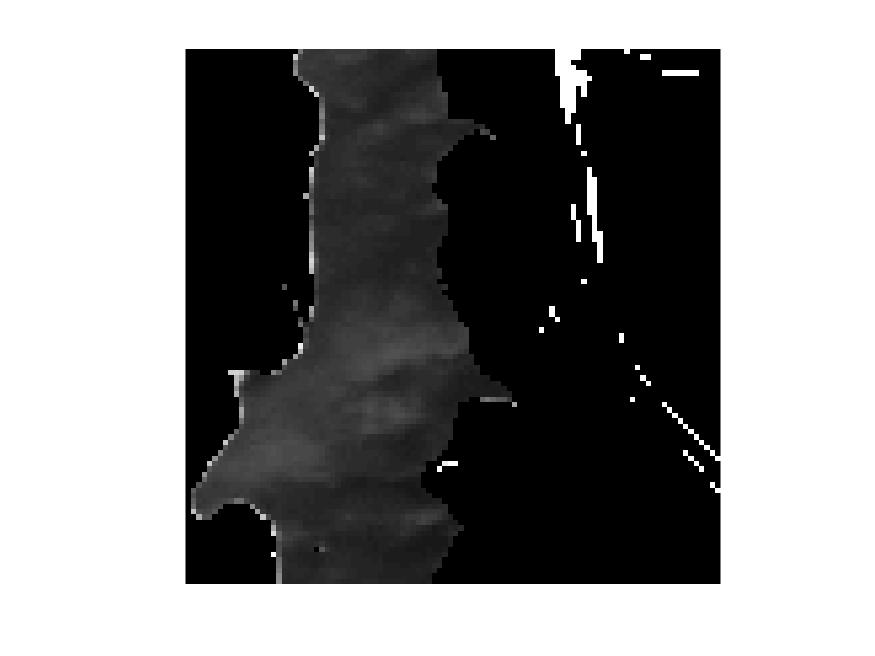}
\includegraphics[width=\textwidth]{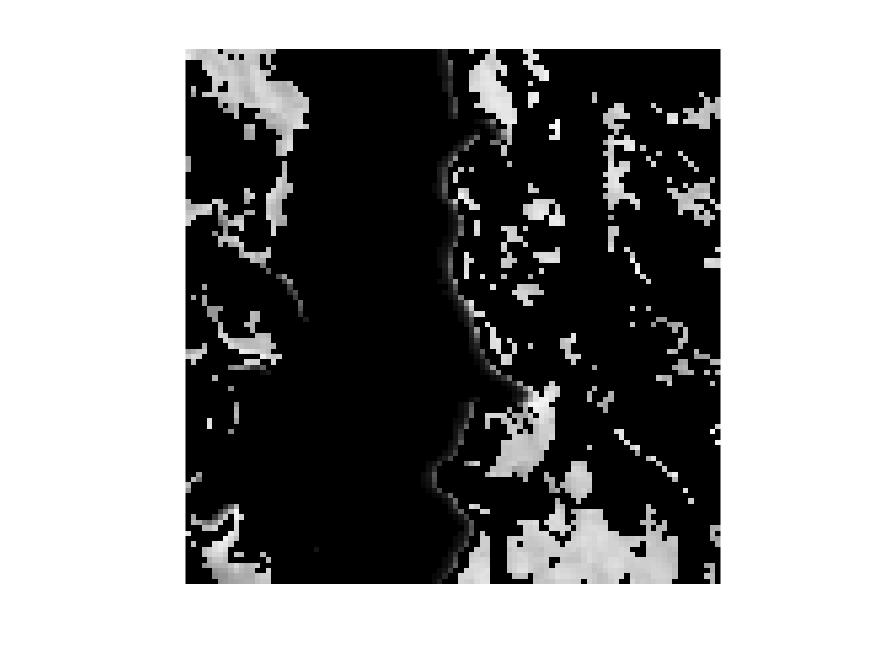}
\includegraphics[width=\textwidth]{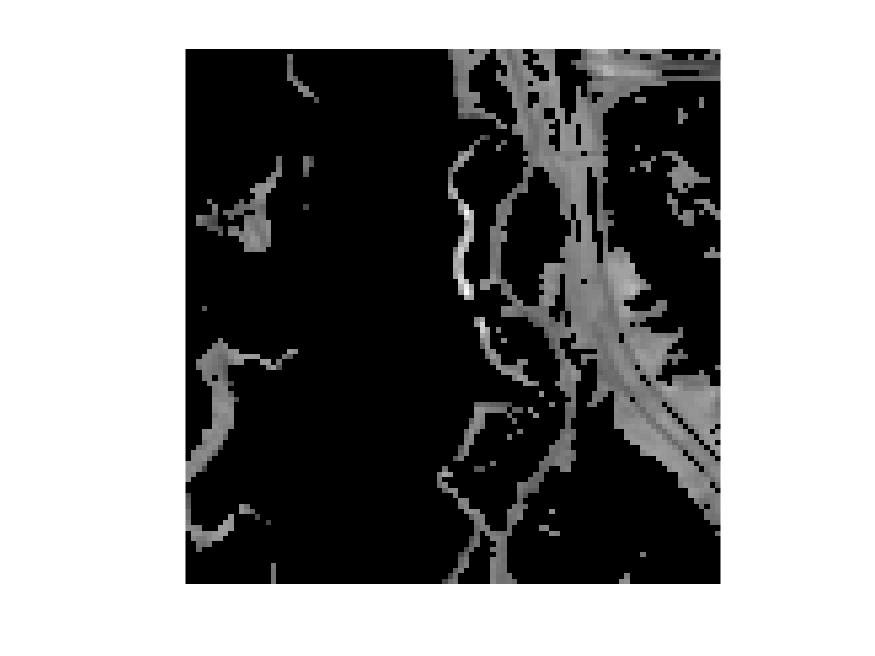}
\end{minipage}
}\hspace{-0.33cm}
\subfigure[EM-onmf]{
\begin{minipage}[b]{0.137\linewidth}
\includegraphics[width=\textwidth]{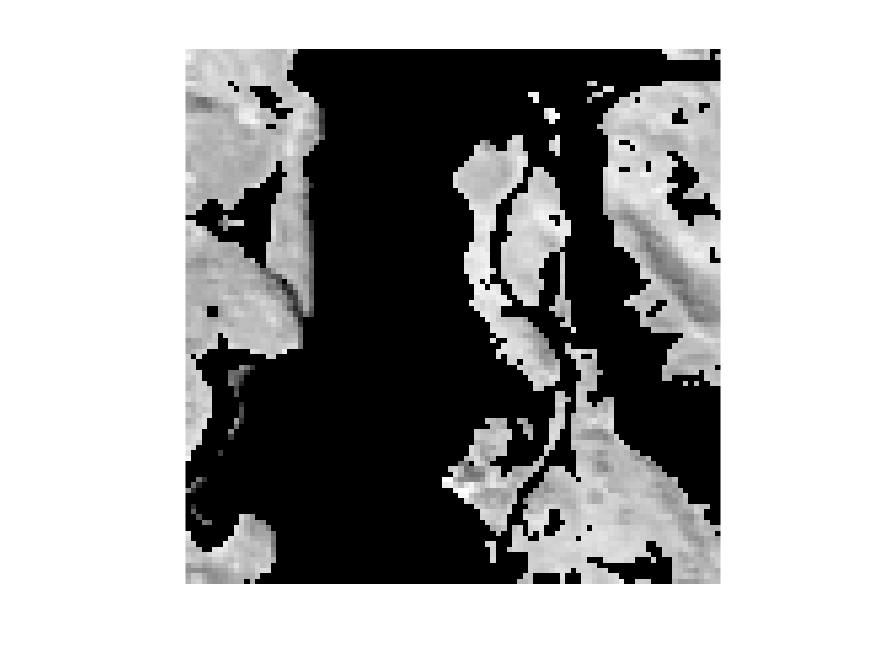}
\includegraphics[width=\textwidth]{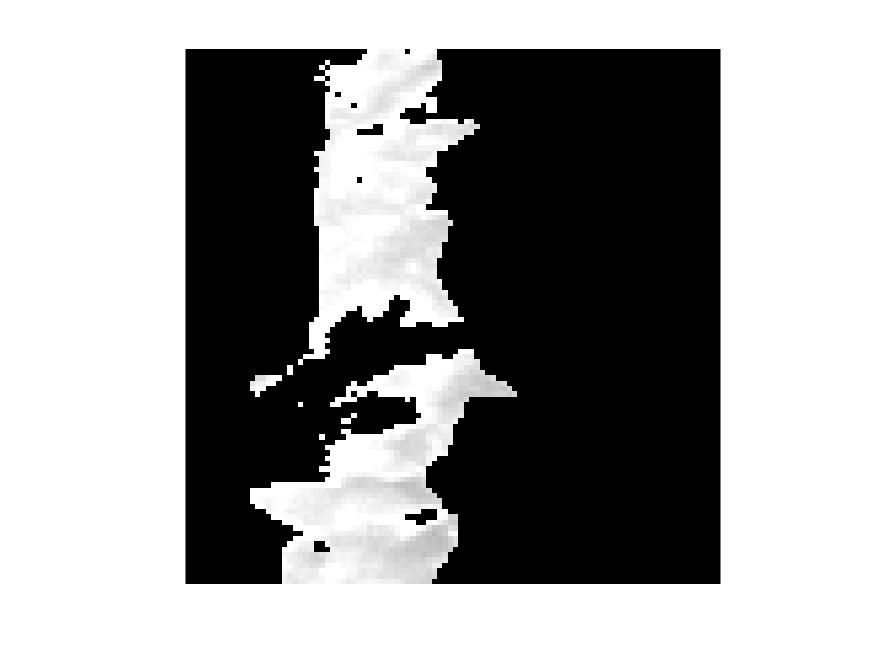}
\includegraphics[width=\textwidth]{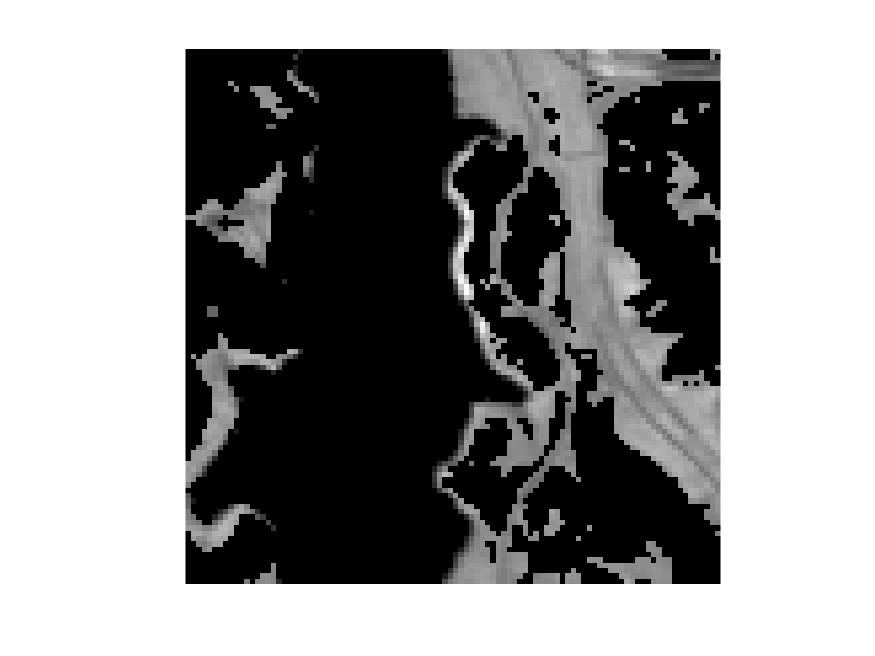}
\includegraphics[width=\textwidth]{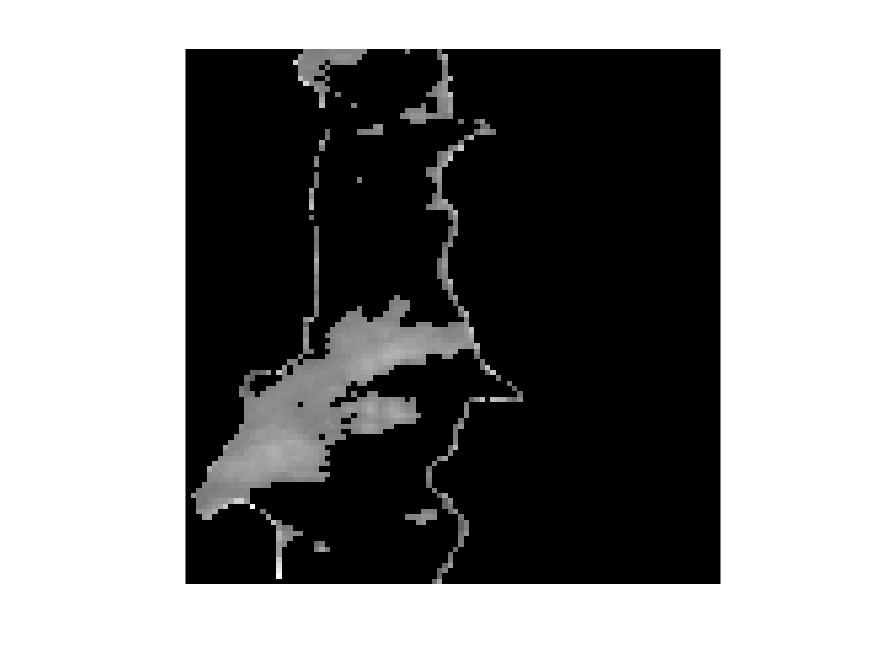}
\end{minipage}
}\hspace{-0.33cm}
\caption{
Unmixing results of Jasper Ridge, from  top to bottom: tree, water, dirt, road}
\label{num:Japser}
\end{figure}

\begin{figure}[!htbp]
\centering
\subfigure[\scriptsize \!\!ground truth]{
\begin{minipage}[b]{0.137\linewidth}
\includegraphics[width=\textwidth]{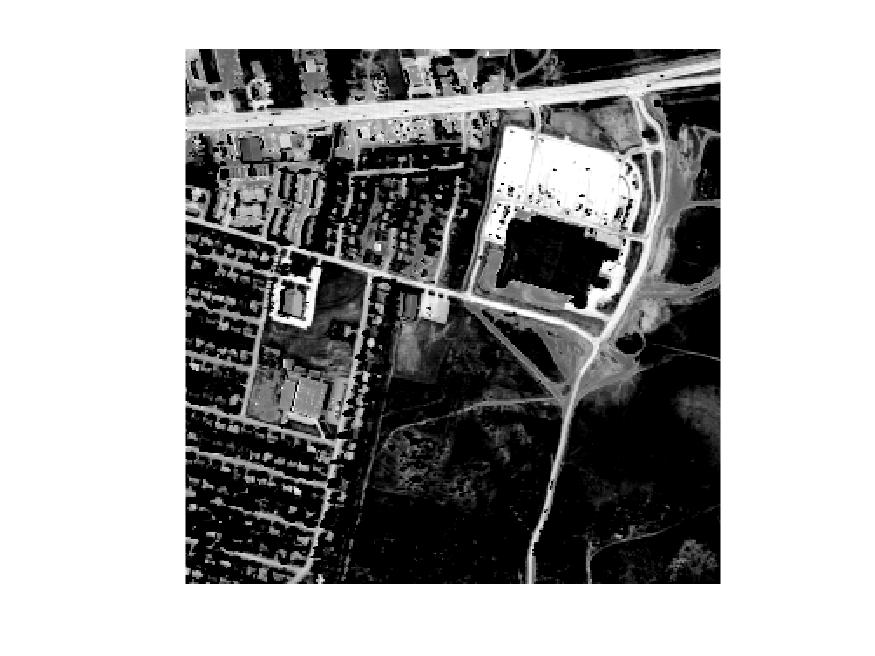}
\includegraphics[width=\textwidth]{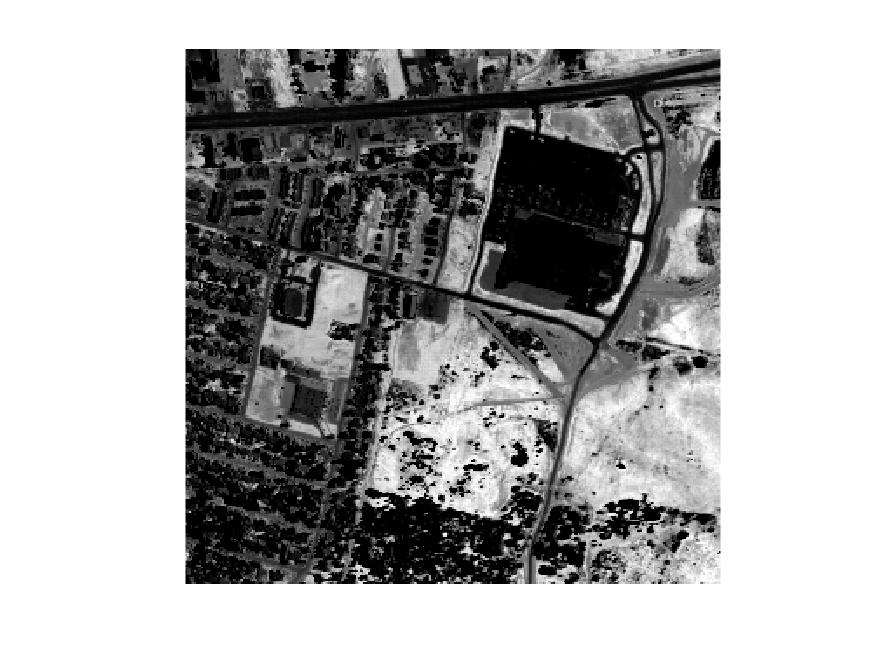}
\includegraphics[width=\textwidth]{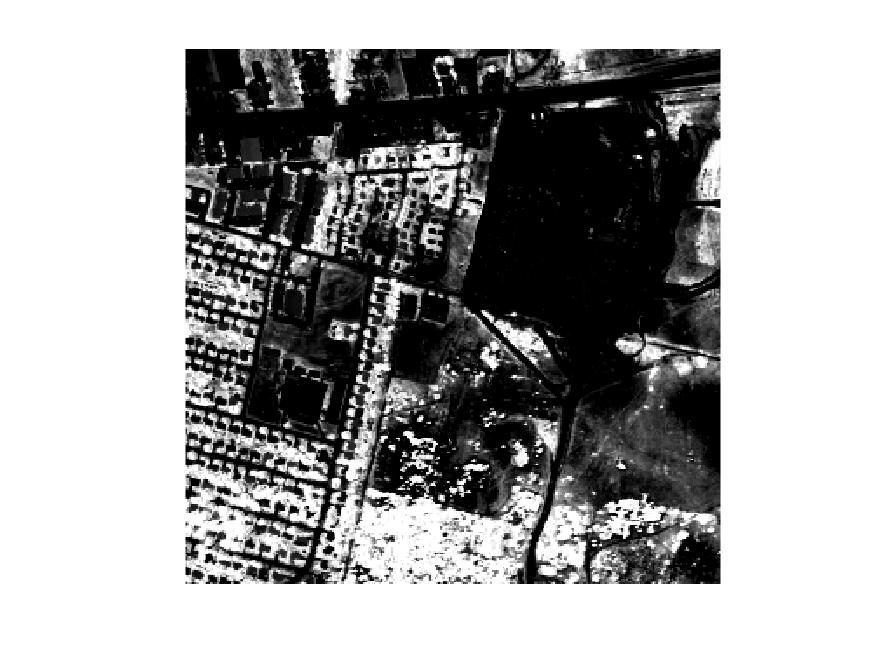}
\includegraphics[width=\textwidth]{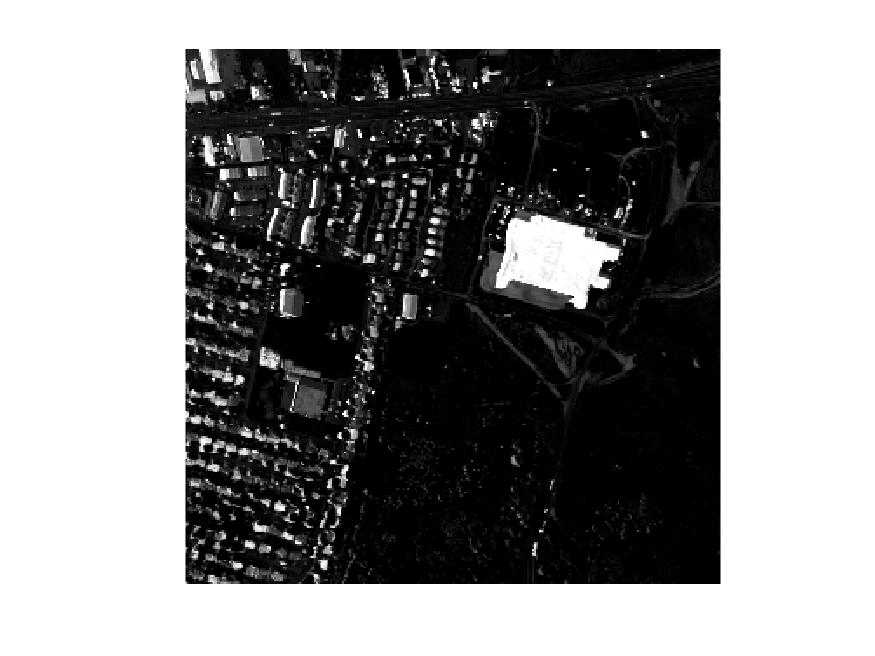}
\end{minipage}
}
\subfigure[\scriptsize \!our method]{
\begin{minipage}[b]{0.137\linewidth}
\includegraphics[width=\textwidth]{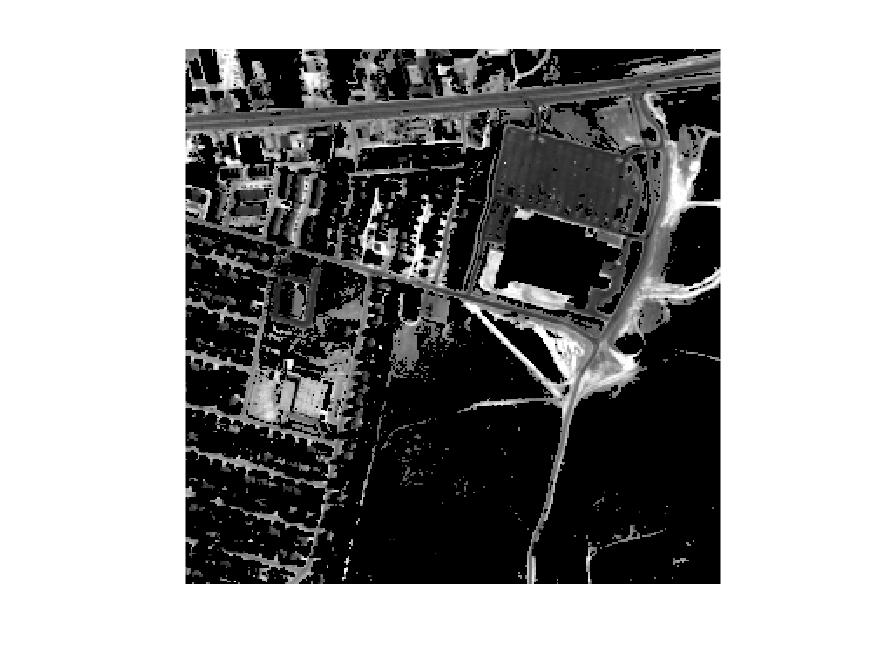}
\includegraphics[width=\textwidth]{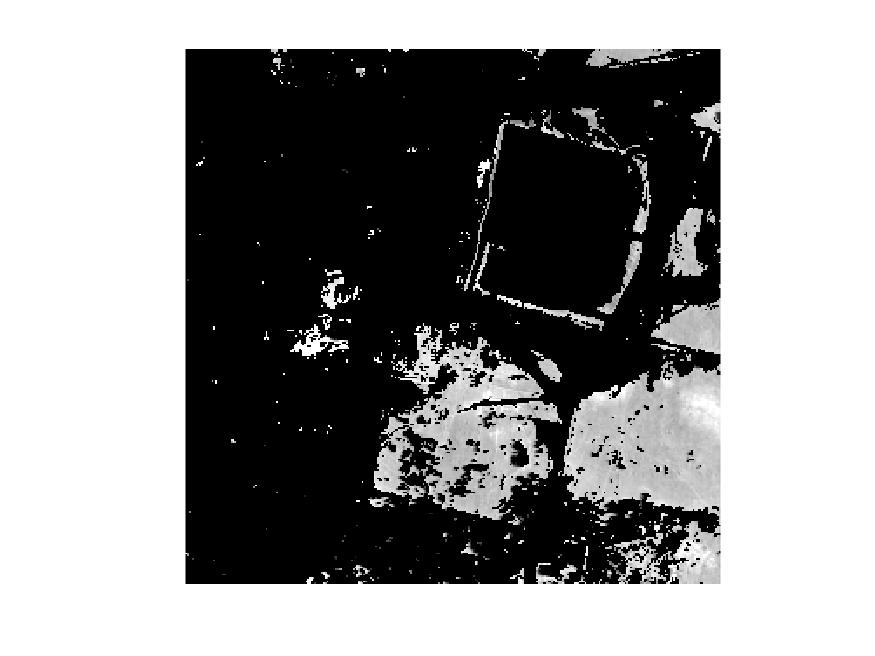}
\includegraphics[width=\textwidth]{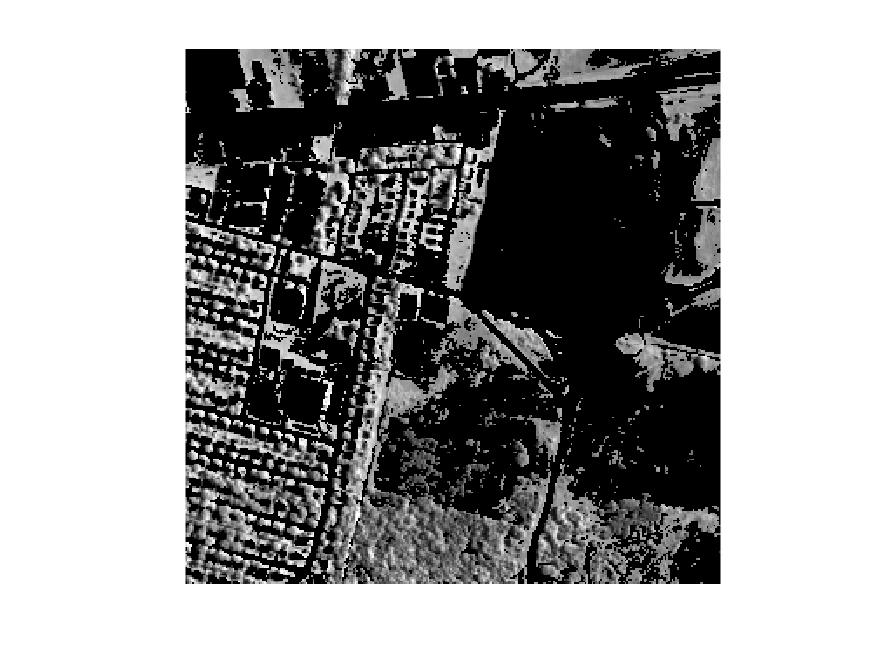}
\includegraphics[width=\textwidth]{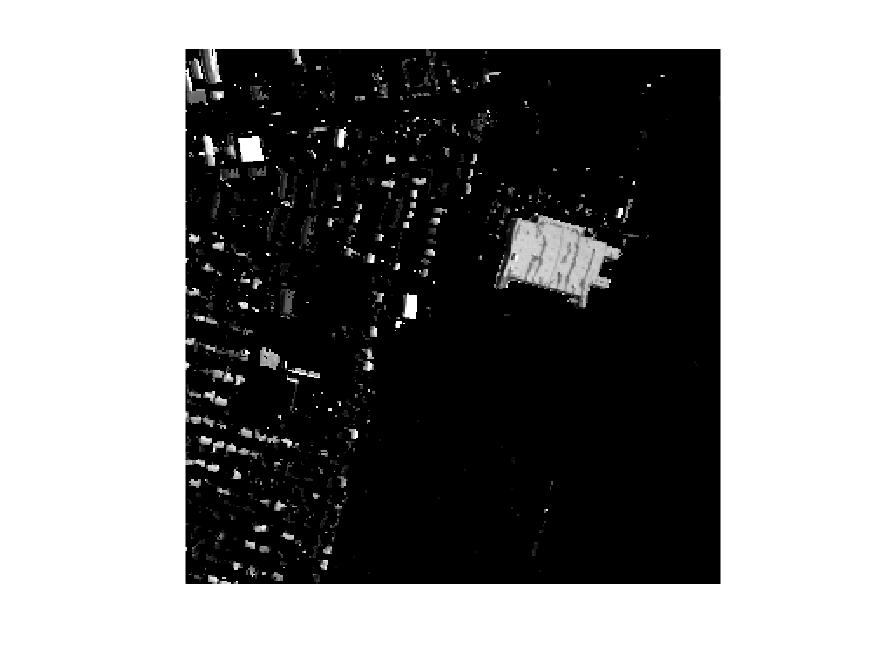}
\end{minipage}
}\hspace{-0.33cm}
\subfigure[U-onmf]{
\begin{minipage}[b]{0.137\linewidth}
\includegraphics[width=\textwidth]{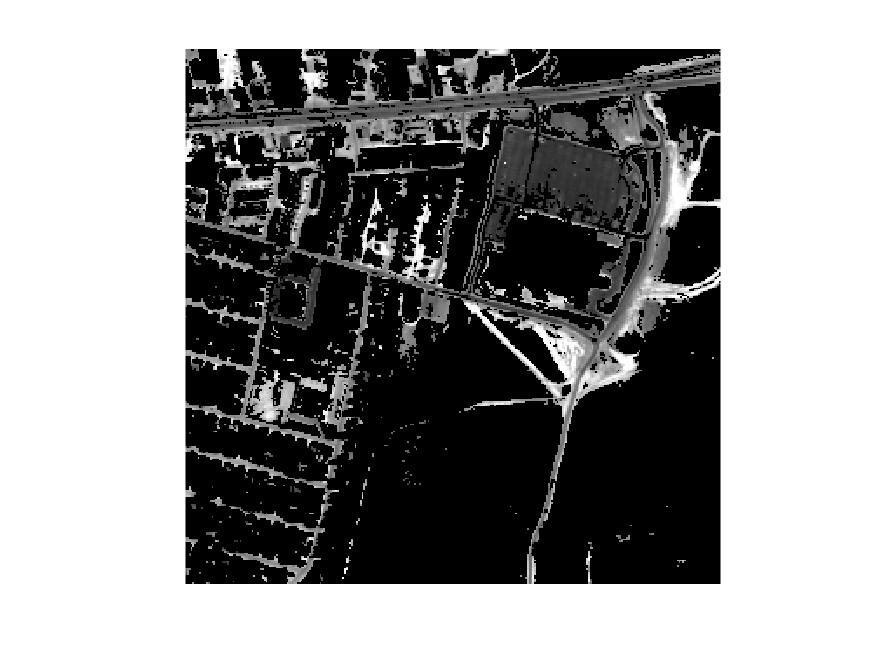}
\includegraphics[width=\textwidth]{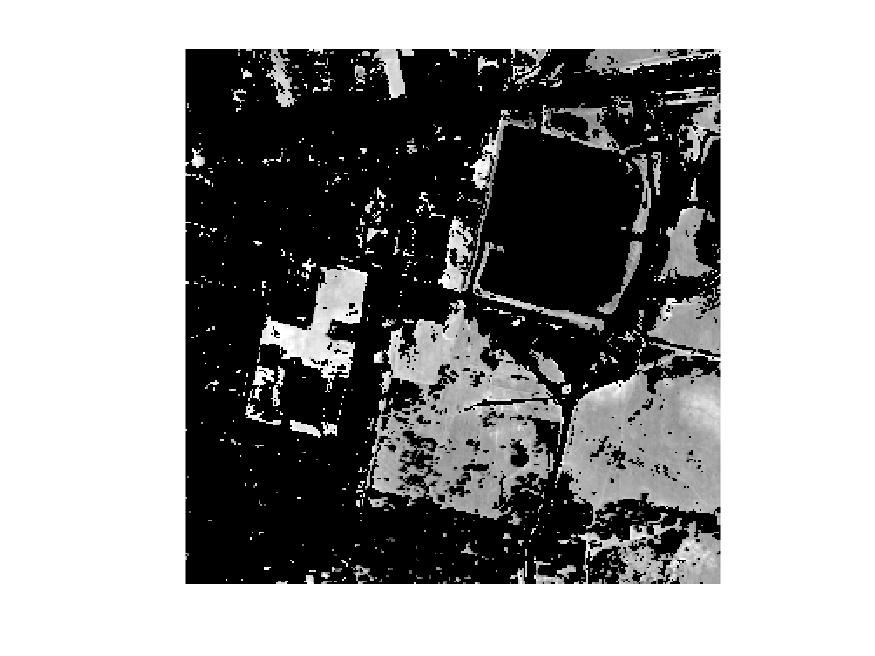}
\includegraphics[width=\textwidth]{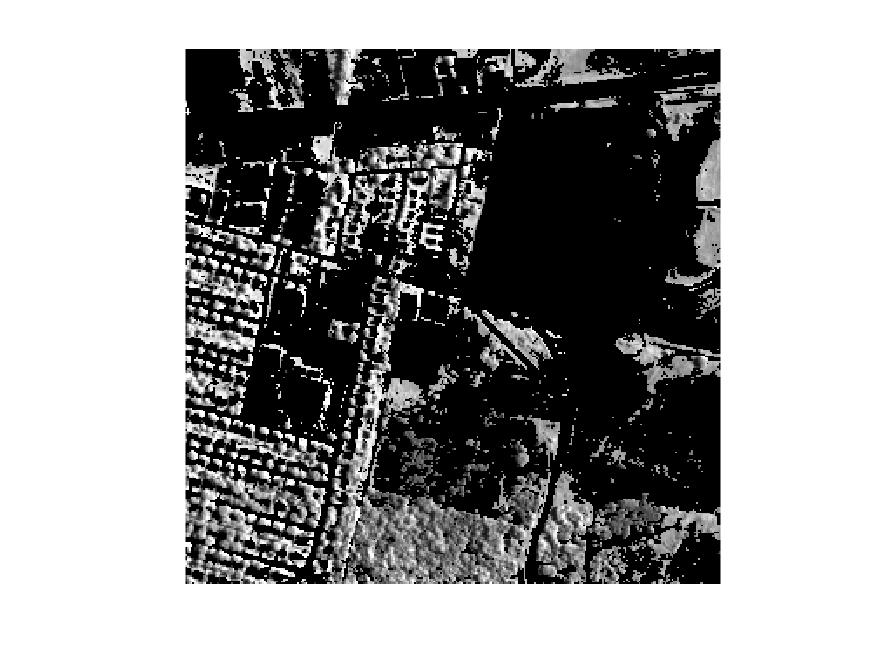}
\includegraphics[width=\textwidth]{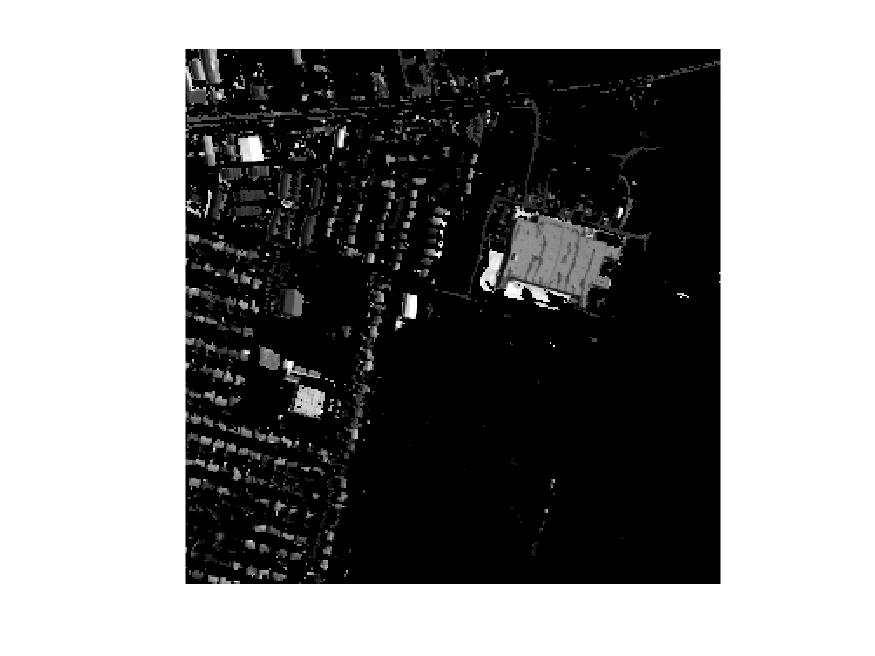}
\end{minipage}
}\hspace{-0.33cm}
\subfigure[OPNMF]{
\begin{minipage}[b]{0.137\linewidth}
\includegraphics[width=\textwidth]{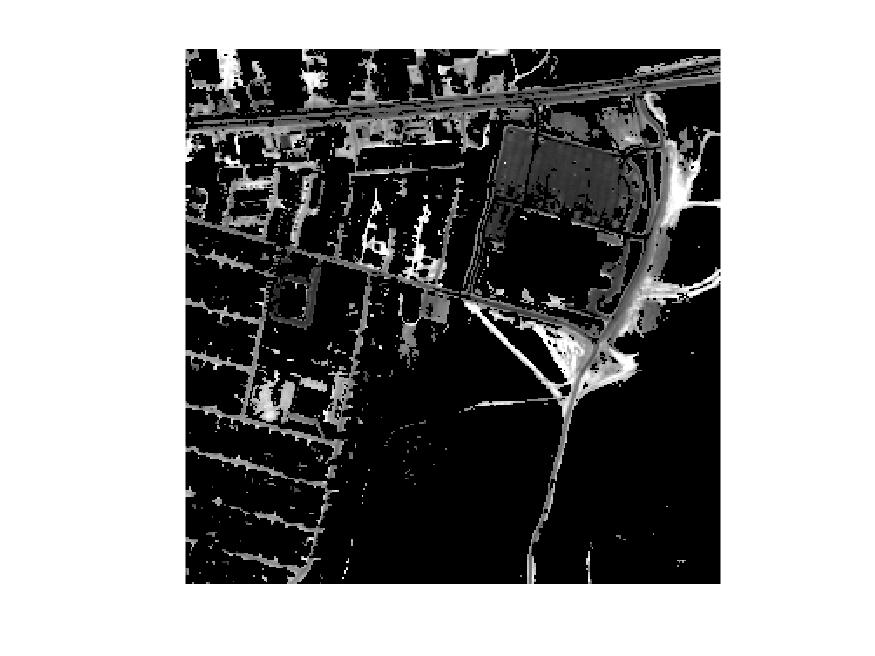}
\includegraphics[width=\textwidth]{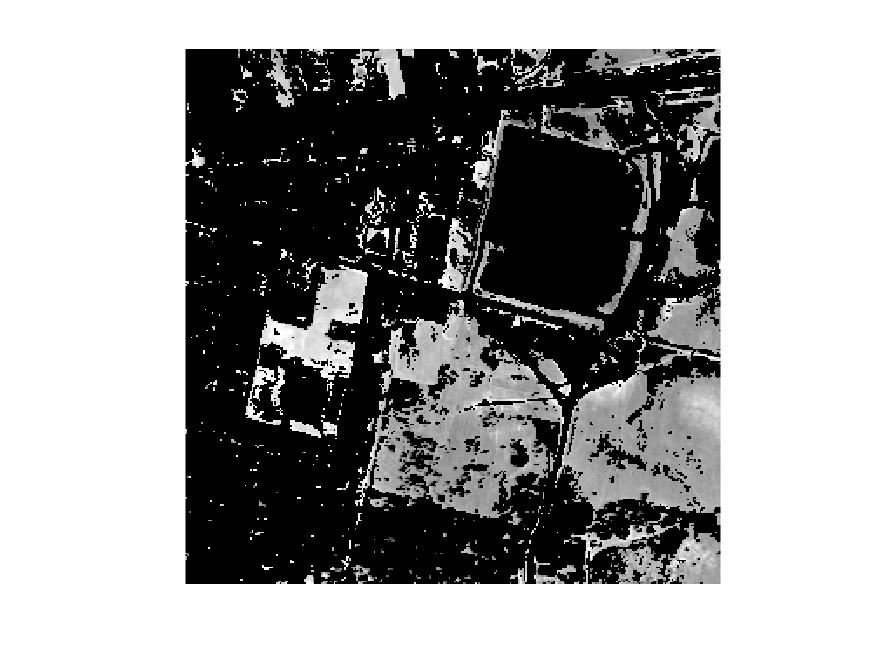}
\includegraphics[width=\textwidth]{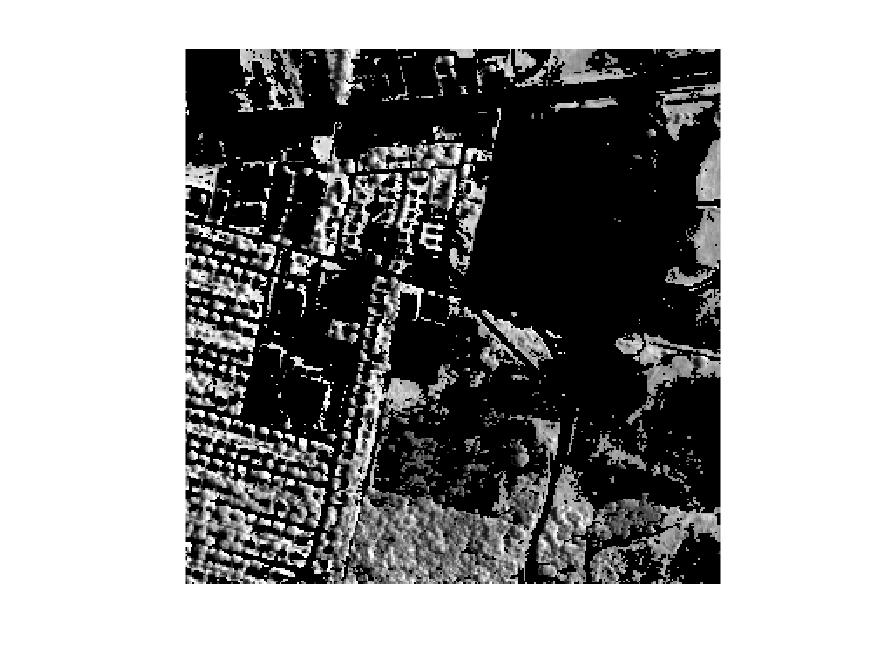}
\includegraphics[width=\textwidth]{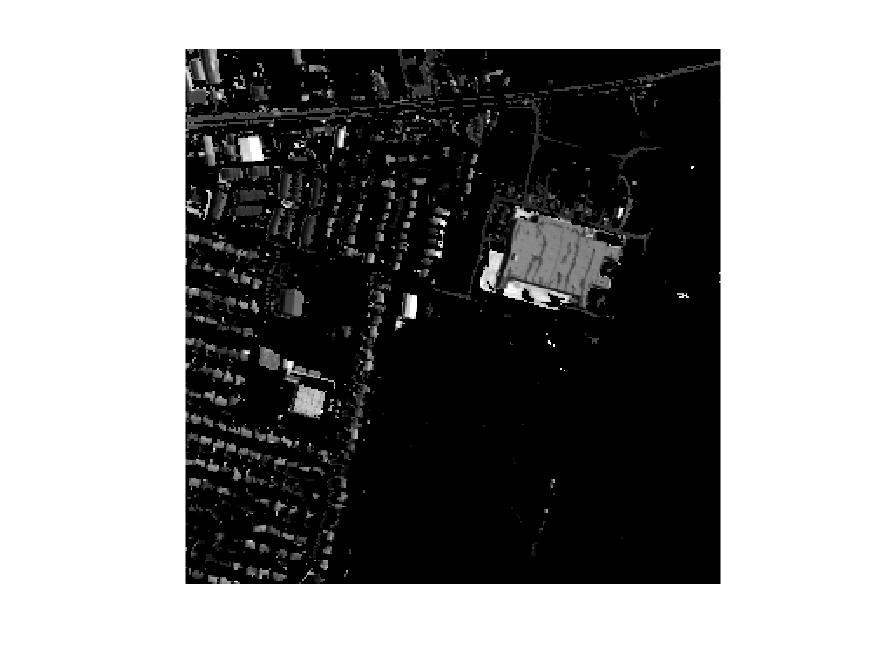}
\end{minipage}
}\hspace{-0.33cm}
\subfigure[K-means]{
\begin{minipage}[b]{0.137\linewidth}
\includegraphics[width=\textwidth]{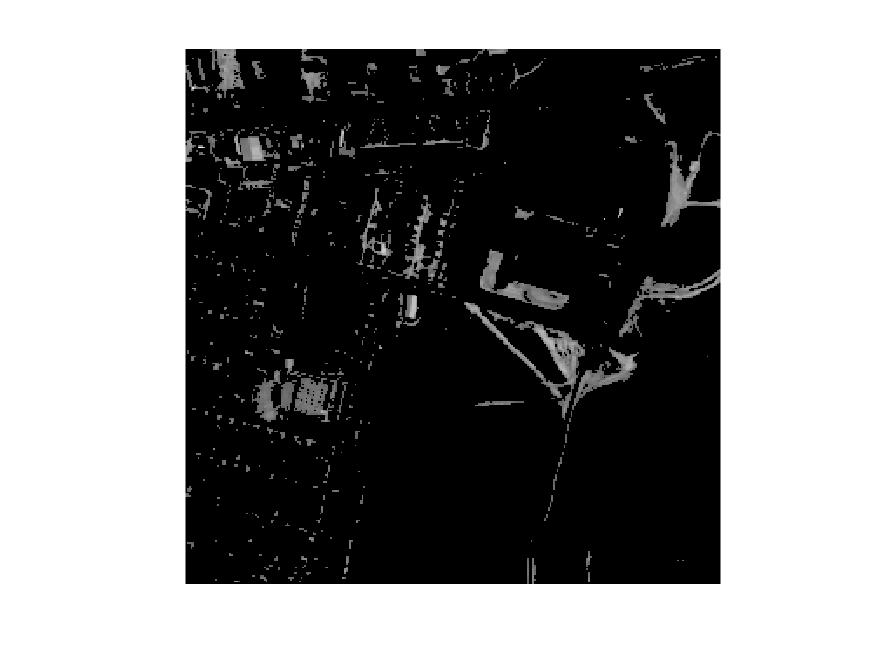}
\includegraphics[width=\textwidth]{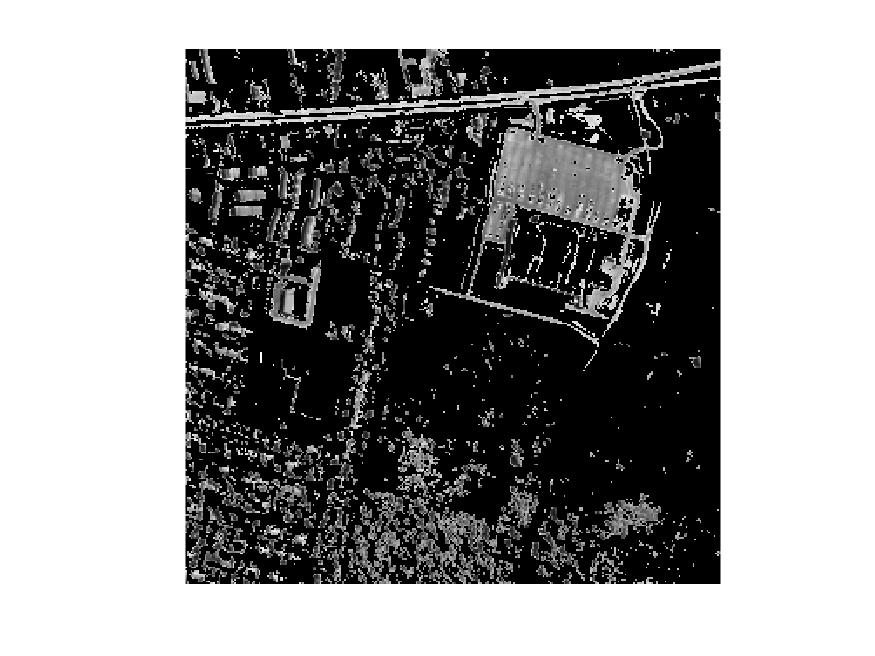}
\includegraphics[width=\textwidth]{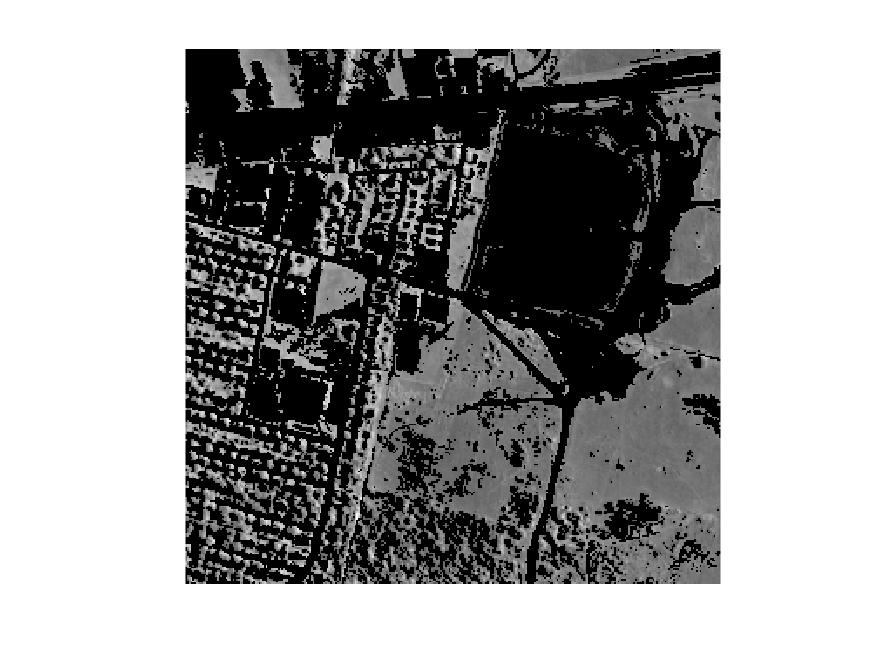}
\includegraphics[width=\textwidth]{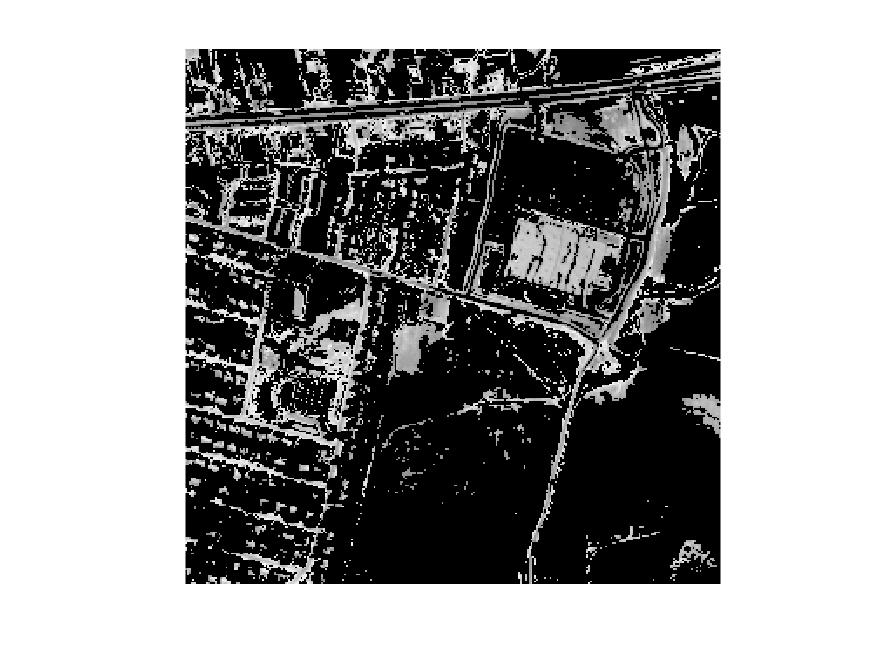}
\end{minipage}
}\hspace{-0.33cm}
\subfigure[ONP-MF]{
\begin{minipage}[b]{0.137\linewidth}
\includegraphics[width=\textwidth]{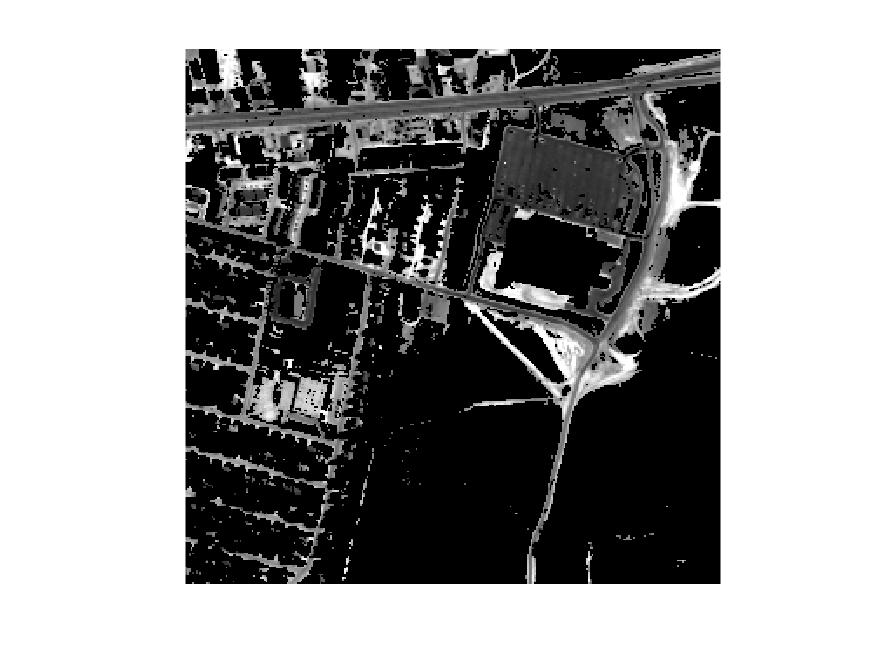}
\includegraphics[width=\textwidth]{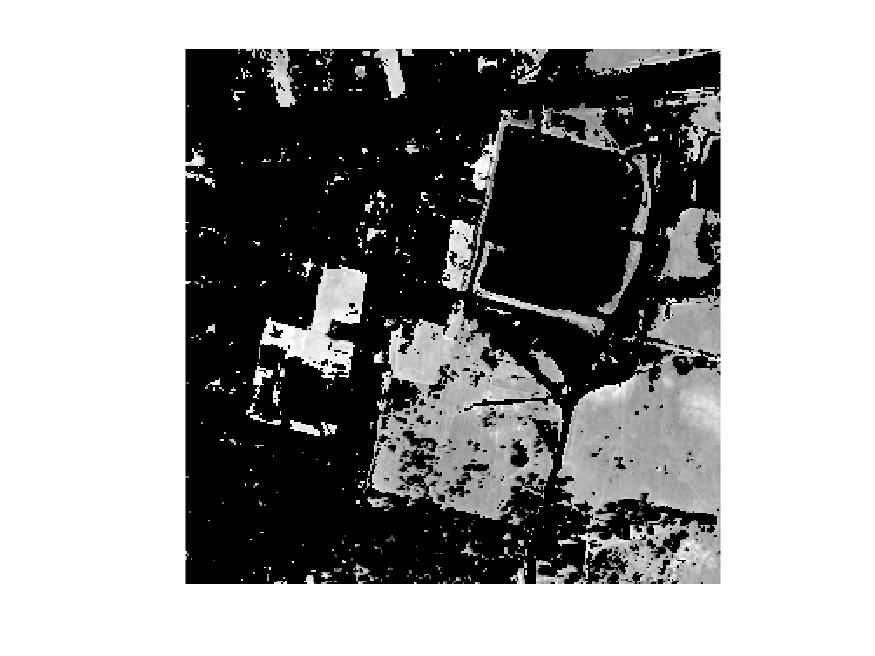}
\includegraphics[width=\textwidth]{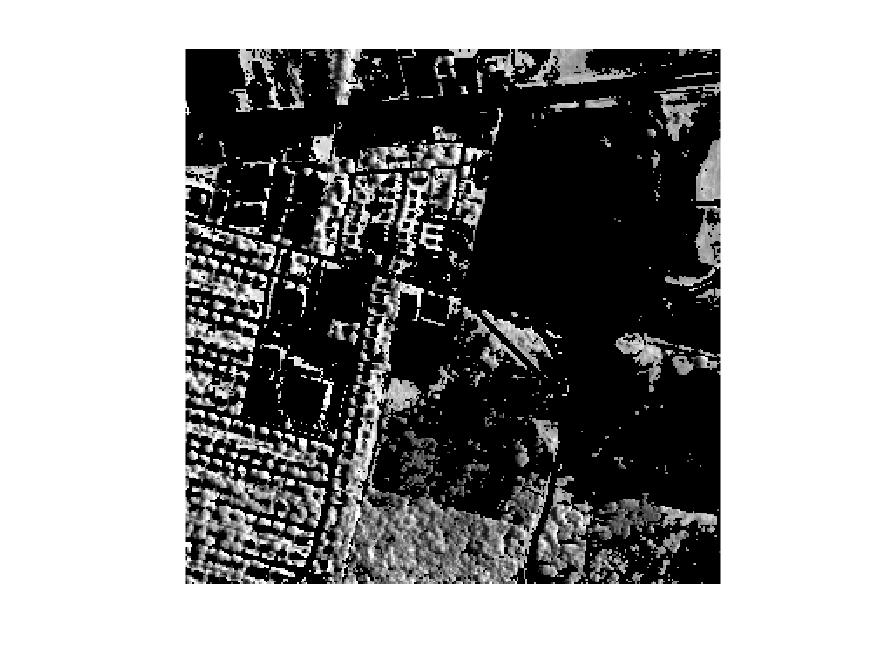}
\includegraphics[width=\textwidth]{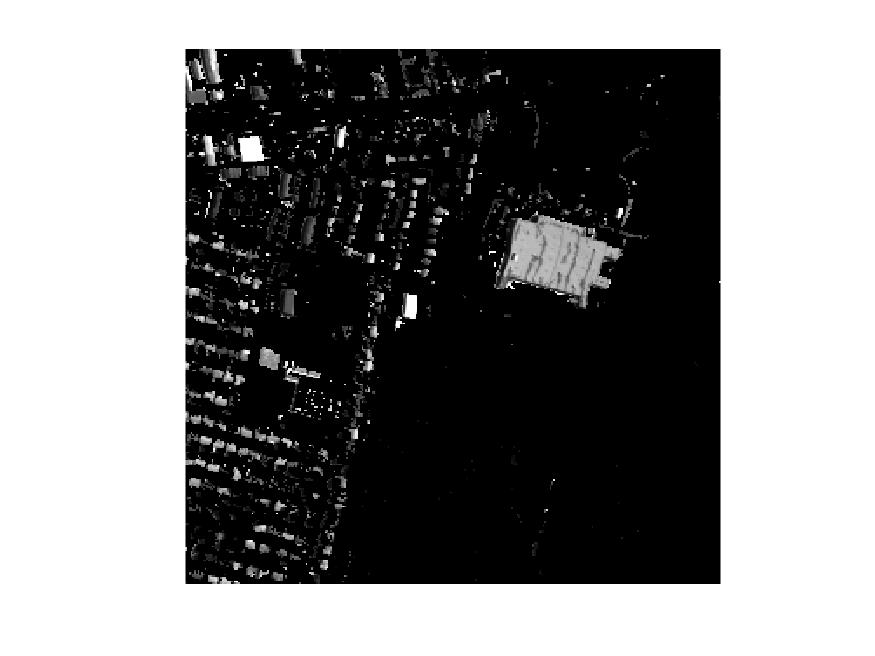}
\end{minipage}
}\hspace{-0.33cm}
\subfigure[EM-onmf]{
\begin{minipage}[b]{0.137\linewidth}
\includegraphics[width=\textwidth]{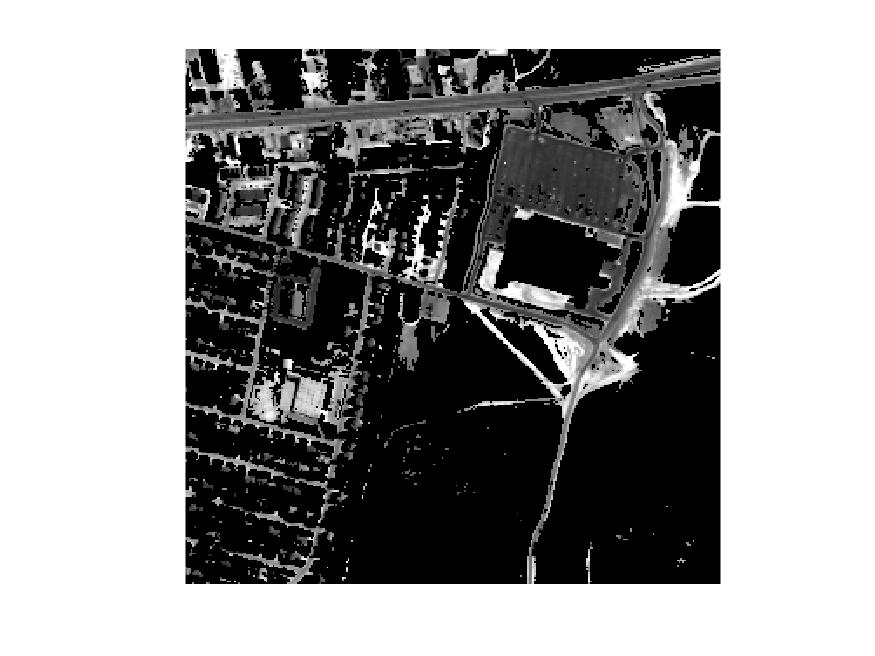}
\includegraphics[width=\textwidth]{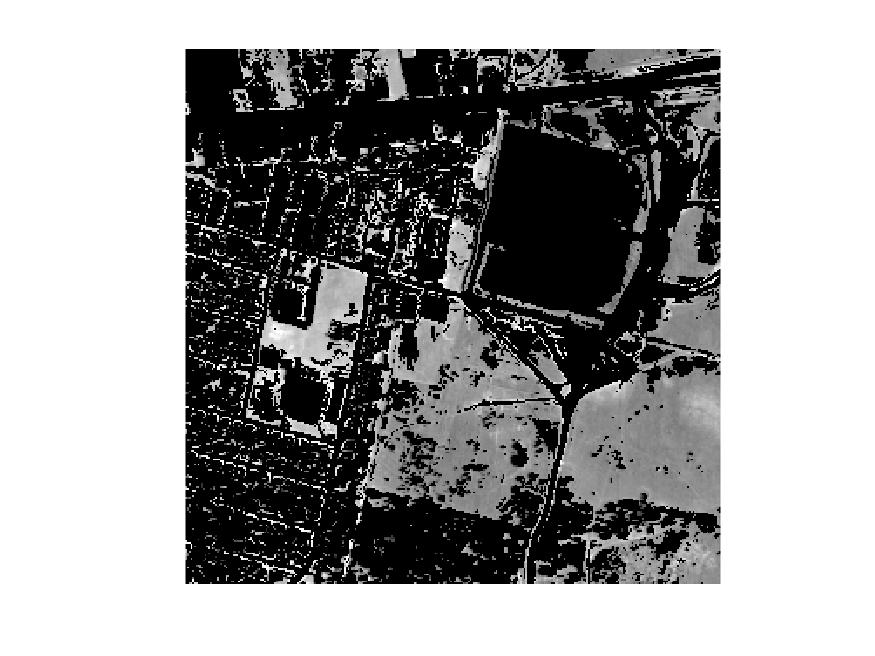}
\includegraphics[width=\textwidth]{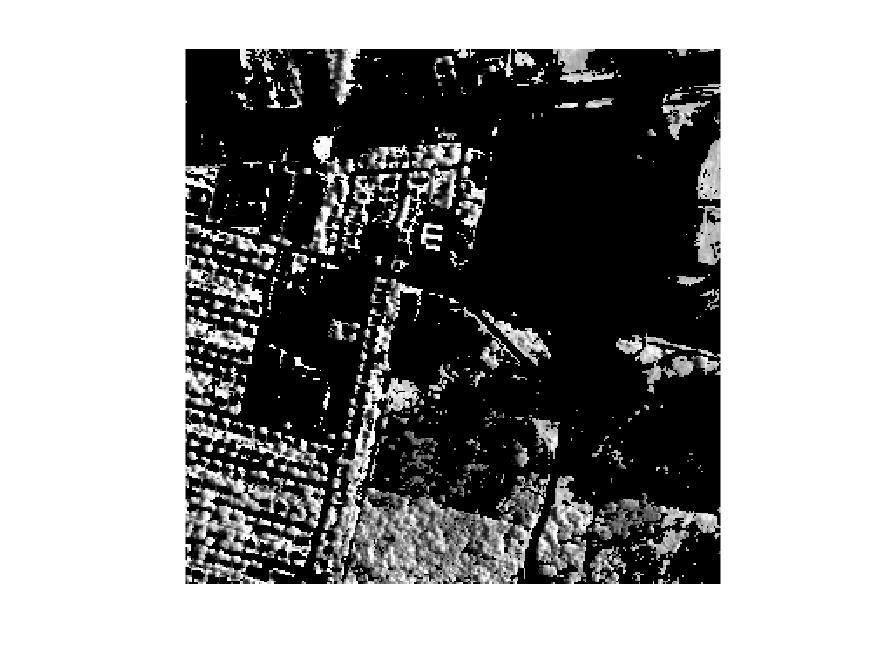}
\includegraphics[width=\textwidth]{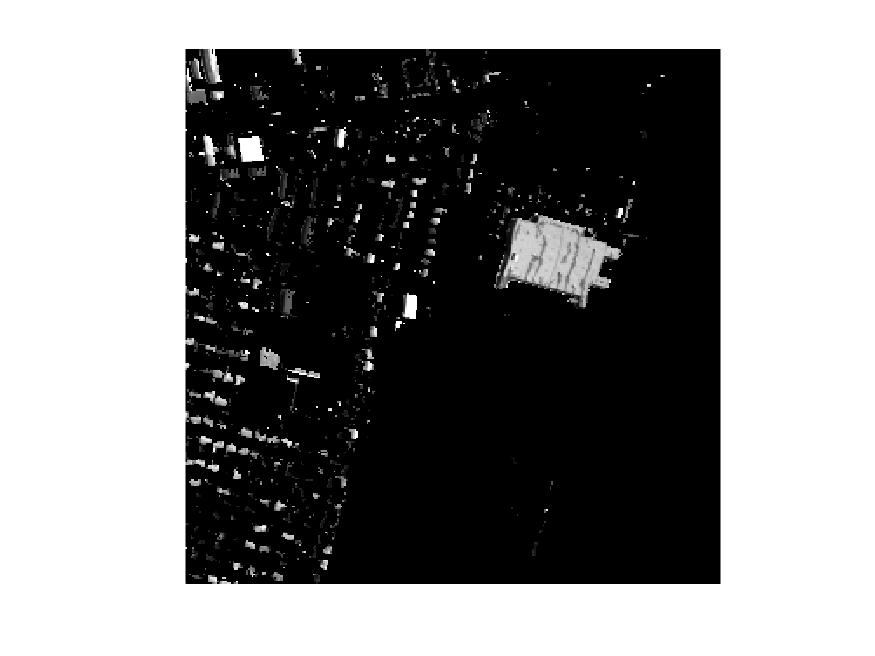}
\end{minipage}
}\hspace{-0.33cm}
\caption{
Unmixing results of Urban, from top to bottom: asphal, grass, tree, roof}
\label{num:Urban}
\end{figure}

Finally, we report in  \cref{tab:table10} the SAD and time cost for the three hyperspectral image datasets. From this table,  we know that the efficiency of the proposed method is competitive to other algorithms. Particularly, our method achieves satisfying SAD among all algorithms. Besides, although EM-onmf is faster than our method on these datasets, the unmixing quality given by EM-onmf is unstable. 
\begin{table}[!htbp]
\centering
\caption{
    Results on the hyperspectral image datasets. 
}
	\setlength{\tabcolsep}{4pt}
\begin{tabular}{|c||cc|cc|cc|}
\hline
 & \multicolumn{2}{c|}{Samson}  &  \multicolumn{2}{c|}{Jasper Ridge} &  \multicolumn{2}{c|}{Urban} \\ \cline{2-7}
method    &SAD    & time(s)& SAD   & time(s)& SAD    & time(s) \\ \hline
EP4Orth+ & \textbf{0.081} &   1.0    & \textbf{0.150} &  1.3     &  0.114 & 22     \\ 
U-onmf    & 0.365 &  10    & 0.306 & 19     &  0.128 &  99     \\ 
OPNMF     & 0.348 &  44    & 0.336 & 85     &  0.132 & 545     \\ 
K-means   & 0.296 &  0.2   & 0.174 & 0.4    &  0.266 &   4     \\ 
ONP-MF    & 0.085 &  16    & 0.276 & 34     &  0.112 & 339     \\ 
EM-onmf   & 0.196 &  0.4   & 0.192 & 0.8    &  \textbf{0.091} &  17  \\ \hline
\end{tabular}
\label{tab:table10}
\end{table}

\subsection{K-indicators model}  
We first remove the zero norm constraints from \eqref{equ:prob:k:indicator}. The exact penalty model \eqref{equ:prob:orth+:new:exact:penalty:general} with $p = 1$, $q = 2$, and $\epsilon = 0$ for solving the K-indicator model becomes 
\be\label{equ:penalty:model:KAP:0}
\min_{X \in \obliqueplus,  Y \in \stiefkk}\, \left\{\widehat P_{\sigma}(X,Y): =  \|U Y - X\|_{\Ftt}^2 + \sigma \|XV\|_{\Ftt}^2\right\},
\ee 
which is further equivalent to 
\be\label{equ:penalty:model:KAP}
\min_{X \in \obliqueplus,  Y \in \stiefkk}\, \left\{P_{\sigma}(X,Y): =  - \frac{1}{\sigma}\langle U Y,   X\rangle  + \frac{1}{2} \|XV\|_{\Ftt}^2\right\}.
\ee 
With a fixed $Y$, \eqref{equ:penalty:model:KAP} is exactly \eqref{equ:penalty:model:projection} with $C = UY$. Similar to the discussion therein, 
we obtain the main PALM iterations \cite{bolte2014proximal} for solving  \eqref{equ:penalty:model:KAP} in  \cref{alg:feasiblePenalty}  as 
\begin{subequations}
\begin{align}
&Y^{l+1} = \proj_{\stiefkk} \left(\beta^{-1}Y^l + U^{\tran} X^{l} \right),\quad  \beta > 0,\label{equ:kind:PALM:Y:1}\\
&X^{l + 1}\ \in\ \proj_{\obliqueplus}\left(X^l - \alpha \left(X^l VV^{\tran} -  UY^{l+1}/\sigma \right)\right), \quad 0< \alpha < 1. \label{equ:kind:PALM:X:1}
\end{align}
\end{subequations}
Theorem 1 in \cite{bolte2014proximal} tells that the sequence $\{(X^{l}, Y^{l})\}$ generated by \eqref{equ:kind:PALM:Y:1} and \eqref{equ:kind:PALM:X:1} converges to a stationary point of \eqref{equ:penalty:model:KAP}. However, we find the convergence is slow if we fix the constant stepsizes $\alpha$ and $\beta$.  Note that the closed form solution of \eqref{equ:penalty:model:KAP} with respect to $Y$ for a fixed $X= X^{l}$ is $\proj_{\stiefkk} \left(U^{\tran} X^{l} \right)$, which corresponds to setting $\beta = +\infty$ in \eqref{equ:kind:PALM:Y:1}. For the tested problem, by some easy calculations, we can see $\alpha_{\LBB}^l \geq 1$.
The practical PALM iterations for solving \eqref{equ:penalty:model:KAP} is thus given as  
 \begin{subequations}
 \begin{align}
 &Y^{l+1} = \proj_{\stiefkk} \left( U^{\tran} X^l \right).  \label{equ:kind:PALM:Y:2}\\
&X^{l + 1}\ \in \ \proj_{\obliqueplus}\left(X^l - \alpha^l \left(X^l VV^{\tran} -  UY^{l+1}/\sigma \right)\right),  \ \alpha^l = \min\{\alpha_{\LBB}^l, 10k\}. \label{equ:kind:PALM:X:2}
\end{align}
\end{subequations}
The flops for \eqref{equ:kind:PALM:Y:2} and \eqref{equ:kind:PALM:X:2} are $2nk^2 + O(k^3) $ and  $2nk^2 + O(nk)$, respectively.

 Chen et al. \cite{chen2019big} proposed a semi-convex relaxation model to solve \eqref{equ:prob:k:indicator}. Their intermediate model  corresponds to \eqref{equ:penalty:model:KAP:0} with $\sigma = 0$ and $\obliqueplus$ replaced by $\{X\in \Rbb^{n\times k}: 0 \leq X \leq 1\}$.  A double-layered  alternating projection framework was investigated in \cite{chen2019big}  to solve the relaxation model. The method was named KindAP. To evaluate the efficiency of our method, we compare it with KindAP (downloaded from \url{https://github.com/yangyuchen0340/Kind}) on data clustering problems. 
We adopt eight  image datasets, including catsndogs,  ORL, CIFAR (train and test), COIL100, flower, omniglot, and UKBench.  We set $\gamma_2 = 10$, $\sigma_0=10$, $\eta = 0.5$ and $\tol^{\feas}=0.1$ and $X^0 = \proj_{\obliqueplus}(U)$ in our method.  The initial points of KindAP is set as  $\proj_{\Rbb^{n,k}_+}(U)$.   Similar as in section \ref{subsection:onmf:2}, purity, entropy and NMI are adopted to judge the performance of proposed algorithms.
The results are presented in \cref{table:KindAP:ours}. 
It shows that the clustering results given by our methods are comparable to that provided by KindAP, which means both methods are able to solve \eqref{equ:prob:k:indicator} with a relatively high quality. On the other hand, our algorithm is generally faster than KindAP. Our algorithm is especially efficient on datasets omniglot and UKbench, in which the number of clusters is relatively large.  Besides, it should be mentioned that although we relax the zero norm constraints from problem \eqref{equ:prob:k:indicator}, the matrix $X$ we obtained is always feasible to \eqref{equ:prob:k:indicator}. By contrast,  the matrix $X$ returned by KindAP may not be   an orthogonal nonnegative matrice although it always satisfies  the zero norm constraints.

 \begin{table}[!htbp]
\centering
\setlength{\tabcolsep}{1.5pt}
\caption{
Comparison of KindAP and our methods on data clustering problems. In the table, ``a'' and ``b'' stand for KindAP and EP4Orth+, respectively. Results marked in bold mean better performance in the corresponding index. }\label{table:KindAP:ours}
\begin{tabular}{|c|c|c||cc|cc|cc|cc|}
\hline 
&  &   &  \multicolumn{2}{c|}{purity(\%)}  &   \multicolumn{2}{c|}{NMI(\%)} &  \multicolumn{2}{c|}{entropy(\%)} &  \multicolumn{2}{c|}{time(s)}    \\  \cline{4-11}
datasets & $n$ & $k$ &a & b& a & b& a & b& a & b \\ \hline
catsndogs &4000&  2 & 96.20 & \textbf{96.23} & 76.80 & \textbf{76.96} &  23.20 &  \textbf{23.04} & 0.03 & \textbf{0.02}  \\ \hline
ORL &400 &  40 & 87.75 & \textbf{88.00} & \textbf{92.94} & 92.76 & \textbf{7.06} & 7.24 & 0.04 & \textbf{0.01} \\ \hline
CIFAR100-test &10000 &  100 & 69.42& \textbf{69.44} & 71.34 & \textbf{71.36} &  28.66 &  \textbf{28.64} & 0.63 & \textbf{0.41} \\ \hline
CIFAR100-train &50000 &  100 & 99.63 & \textbf{99.63} & 99.57 & \textbf{99.57} & 0.43&  \textbf{0.43} & 3.13 & \textbf{1.66} \\ \hline
COIL100 &7200 &  100 & \textbf{91.93} & \textbf{91.93} & 97.30 & \textbf{97.41} & 2.70  & \textbf{2.59} & 1.47 & \textbf{0.44} \\ \hline
flower &2040 &  102 & \textbf{44.95} & 44.90 & \textbf{63.52} & 63.40 & \textbf{36.48}& 36.60 & 0.58 & \textbf{0.42} \\ \hline
omniglot &17853 &  1623 & 21.95 & \textbf{21.97} & 70.86 &\textbf{70.94} & 29.14 &\textbf{29.06} & 1176 & \textbf{432}\\ \hline
UKBench &10200 & 2550 & 90.64 & \textbf{91.04} & 97.64 & \textbf{97.76} & 2.36 & \textbf{2.24} & 3215 & \textbf{1268} \\ \hline
\end{tabular}
\end{table}

\section{Concluding remarks} \label{section:concluding} In this paper, we
consider optimization with nonnegative and orthogonality  constraints. \rev{We focus
on an equivalent formulation of the concerned problem,  
and show that the two formulations share the same minimizers and first- and
second-order optimality conditions. By estimating a local error bound of $\stiefplus$,} we  provide a
general class of exact and possibly smooth penalty models as well as  a
practical penalty algorithm with postprocessing.  \rev{We investigate the
asymptotic convergence of the penalty method and show that any limit point is a
weakly stationary point of the concerned problem and becomes a stationary point
under some more mild conditions.} A second-order method for solving the penalty subproblem, namely, optimization with nonnegative and multiple spherical constraints, is also given. Our numerical results show that the proposed penalty method performs well for the projection problem, 
ONMF and the K-indicators model  and it can always return high quality orthogonal  nonnegative matrices.

\appendix 

\bibliography{orth+}
\bibliographystyle{siam}

\section{Construction of problem \eqref{equ:prob:stiefelplus:proj} with unique  solution}\label{section:proj}
\begin{proposition}\label{prop:proj}
Choose $X^* \in \stiefplus$ and  $L \in \Rbb^{k \times k}$ with positive diagonal elements satisfying  
$L_{ii}L_{jj} > \max\{L_{ij}, L_{ji}, 0\}^2\ \forall i, j \in [k], i \neq j.$
Then the optimal solution of \eqref{equ:prob:stiefelplus:proj} with $C = X^*L^{\tran}$ is unique and exactly $X^*$.
\end{proposition}
\begin{proof}
For simplicity of notation, we use $\sum_i$ to denote $\sum_{i \in [k]}$ in the proof.  Since problem \eqref{equ:prob:stiefelplus:proj} is equivalent to
 $\max_{X \in \stiefplus} \, \langle C, X \rangle$, we only need to show that 
 $\langle C,Y\rangle<\langle C,X^*\rangle=\sum_{i}L_{ii}$,   $\forall\ \stiefplus \owns  Y \neq X^*$. Let $Z = \sgn(Y)$ and $P = \proj_{\Rbb_+^n}(L)$. We have 
\be \label{equ:CY}
\langle C,Y \rangle  = \mtr (L(X^*)^{\tran}Y) = \sum\nolimits_{i} \sum\nolimits_j L_{ji}  \ybf_i^{\tran} \xbf^*_j
\le\sum\nolimits_{i} \sum\nolimits_j  P_{ji}\ybf_i^T(\xbf^*_j\circ \zbf_i).
\ee
 Define $
w_{ji} = \|\xbf_j^* \circ \zbf_i\|^2$.  With $X^* \in \stiefplus$, we have  $\|\sum_{j} P_{ji} (\xbf_j^* \circ \zbf_i)\|$ = $(\sum_{i} P_{ji}^2 w_{ji})^{1/2}$. 
 Using the Cauchy-Schwarz inequality, $\|\ybf_i\|=1$ and  the requirements on $L$, we have  
\be\label{equ:CY2}
\sum\nolimits_{j} P_{ji}\ybf_i^T(\xbf^*_j\circ \zbf_i) \leq \Big(\sum\nolimits_{j} P^2_{ji}w_{ji}\Big)^{\frac12} \leq P_{ii} \Big(\sum\nolimits_{j} \frac{P_{jj}}{P_{ii}}w_{ji}\Big)^{\frac{1}{2}}.
\ee
With \eqref{equ:CY} and $\langle C, X^*\rangle = \sum_{i}L_{ii} = \sum_{i} P_{ii}$, we further have 
\be \label{equ:CY3}
\langle C, Y \rangle \leq \sum\nolimits_{i} P_{ii} \Big(\sum\nolimits_{j} \frac{P_{jj}}{P_{ii}}w_{ji}\Big)^{\frac{1}{2}}   \leq   \Big(\sum\nolimits_{i} P_{ii} \Big)^{\frac12} \Big( \sum\nolimits_{i} \sum\nolimits_{j} P_{jj} w_{ji}\Big)^{\frac12}  \leq      \langle C, X^* \rangle, 
\ee
where the second inequality uses the fact that $\sum_{i} a_i x_i^{\frac12} \leq (\sum_{i} a_i)^{\frac12}(\sum_{i} a_i x_i)^{\frac12}$ for  $a_i>0$ and  $x_i\geq 0$,  and the third inequality uses $\sum_{i} w_{ji} \leq 1$.  Obviously, the equalities in  \eqref{equ:CY2} and \eqref{equ:CY3} hold if and only if $Y=X^*$. The proof is completed. 
\end{proof}

\end{document}